\documentclass[a4paper, twoside, 10pt]{amsart}

\usepackage[utf8]{inputenc} 
\usepackage[T1]{fontenc}
\usepackage{lmodern}
\usepackage{fancyhdr}
\usepackage{fullpage}
\usepackage{geometry}
\usepackage{young}
\usepackage{amsmath}
\usepackage{amssymb}
\usepackage{mathrsfs}
\usepackage{enumerate}

\usepackage{pstricks}
\usepackage{pstricks-add}
\usepackage{pst-plot}
\usepackage{graphicx}

\newtheorem{lemme}{Lemma}[section]
\newtheorem{prop}{Proposition}[section]
\newtheorem{theo}{Theorem}[section]

\newcommand{\Z}{\mathbb{Z}}
\newcommand{\N}{\mathbb{N}}

\newcommand{\C}{\mathbb{C}}
\newcommand{\F}{\mathbb{F}}
\newcommand{\Q}{\mathbb{Q}}
\newcommand{\T}{\mathbb{T}}
\DeclareMathOperator{\diag}{diag}
\DeclareMathOperator{\tr}{Tr}
\DeclareMathOperator{\lcm}{lcm}
\DeclareMathOperator{\Gal}{Gal}

\title{Image of the Artin groups of classical types inside the finite Iwahori-Hecke algebras}
\author{Alexandre Esterle}
\date{}

\begin{document}
\address{Laboratoire Ami\'enois
de Math\'ematiques Fondamentales et Appliqu\'ees, CNRS-UMR 7352,
Universit\'{e} de Picardie Jules Verne, 33 rue Saint Leu, 80039
Amiens Cedex 1, France.} 

\begin{abstract}

 We determine the image of the Artin groups of types B and D inside the Iwahori-Hecke algebras, when defined over finite fields, in the semisimple case. This generalizes earlier work on type A by Brunat, Magaard and Marin. In this multi-parameter case, this image depends heavily on the parameters.
 \end{abstract}

\email{alexandre.esterle@u-picardie.fr}

% Pour plusieurs classifications : \subjclass{Primary 37B10 ; Secondary 20C10}

\maketitle

\section{Introduction and notation}

In this article, we determine the image of the Artin groups of types $B$ and $D$ inside the Iwahori-Hecke algebras defined over finite fields in the semisimple case. We now recall previous work done on this subject. O. Brunat and I. Marin determine the image of the usual braid group inside the finite Temperley-Lieb algebra \cite{BM} and O. Brunat, K. Magaard and I. Marin determine the image of the usual braid group inside its finite Iwahori-Hecke algebra \cite{BMM}. In \cite{IH2}, I. Marin determines the Zariski closure of the image of the Artin groups inside the corresponding Iwahori-Hecke algebra in characteristic $0$ and for generic parameters. In this paper we extend and conclude the study of classical types over finite fields and establish results which are a natural sequel to the results in \cite{IH2}. For a Coxeter group $W$, we write $A_W$ for its associated Artin group and $\mathcal{A}_W$ for the derived subgroup of $A_W$. In theorems analogous to Theorem 1.1 of \cite{BMM}, we determine, under certain conditions on the parameters of our finite Iwahori-Hecke algebras, the image of the derived subgroups of the Artin groups $\mathcal{A}_{B_n}$ and $\mathcal{A}_{D_n}$ inside their associated Iwahori-Hecke algebras. Restricting to the derived subgroup is, as it was in type $A$, more convenient, and does not significantly weaken our results since $A_{B_n}/\mathcal{A}_{B_n}\simeq \Z^2$ and $A_{D_n}/\mathcal{A}_{D_n}\simeq \Z$. For type $B$, the main results are given in Theorems \ref{result1}, \ref{result2} and \ref{result3} in Section \ref{lalala}. The main result for type $D$ is given in Theorem \ref{result4}.

As in \cite{BHRC}, we write $SL_n(q), SU_n(q^\frac{1}{2}), SP_n(q)$ and $\Omega_n^\pm(q)$ for the finite classical groups acting naturally on the vector space $\F_q^n$. As in \cite{BM} and \cite{BMM} for type $A$, the irreducible representations of the Iwahori-Hecke algebras in types $B$ and $D$ are explicitly described by the Hoefsmit model (see \cite{G-P} or \cite{HOEF}). The irreducible representations of the Iwahori-Hecke algebras are indexed by double-partitions $\lambda$ of $n$ and have a basis formed by the standard double-tableaux $\T$ associated with those double-partitions. A double-partition of $n$ is a partition of $r$ and a partition of $n-r$ for a given $r\in [\![0,n]\!]$. Each double-partition is therefore associated with two Young diagrams, a standard double-tableau associated with this double-partition is the result of some filling of these Young diagrams with the integers from $1$ to $n$ which increase towards the right and towards the bottom within each diagram. The image of $\mathcal{A}_{B_n}$ and $\mathcal{A}_{D_n}$ by a given irreducible representation associated to a double-partition is one of the above finite classical groups defined over $\F_q$ or $\F_{q^\frac{1}{2}}$, depending on properties of the double-partitions and of the field on which the Iwahori-Hecke algebra is defined. In order to recover the image inside the full Iwahori-Hecke algebra over the field on which it is defined, we find the factorizations between different irreducible representations (see Propositions \ref{isomorphisme}, \ref{patate} and \ref{isomorphisme2}) depending on combinatorial properties of the double-partitions, and then use Goursat's Lemma (see \cite{Gour}). 

 Certain problems arise when treating types $B$ and $D$ which were not present in the case of type $A$. The main difficulty which arises in type $B$ comes from the second parameter which forces us to deal with a larger variety of field extensions (see Section \ref{lalala}). The image varies according to the field extension which is considered, in the last three cases, the results obtained are quite different from the ones in type $A$. The outline of the proof is the same as in \cite{BMM}, we first find the factorizations between the different representations and prove the result for small $n$. Then we get the result for all $n$ by induction using a theorem by  Guralnick and Saxl \cite{GS} and the branching rule. The small cases are the most interesting parts of the study and require techniques different from the ones in \cite{BMM}. One of these techniques is to use the maximal subgroups of low-dimensional classical groups which were determined in \cite{BHRC}.
 
 In type $D$, the main difficulty comes from the fact that the representations associated with double-partitions with the same components split into two irreducible representations and we have to consider the double-partitions up to transposition of the two components. For example, this produces a more complex branching rule (Lemma \ref{branch}).

Our results about the image of Artin groups inside the finite Hecke algebras may have various applications. For instance, finite classical groups and direct products of finite classical groups appear as finite quotients of the Artin groups. Since the latter are fundamental groups of algebraic varieties, this also defines interesting finite coverings of these varieties. Since these varieties are defined over $\Q$, this may have applications to inverse Galois theory (see for example \cite{SV}). It is also interesting in terms of finite classical groups because we get explicit generators verifying the braid relations for those groups. This can provide interesting constructions of these groups and some of their subgroups by looking at restrictions to parabolic subgroups of the Artin groups.

We now introduce various notations which we use throughout the article. For a finite group $G$, we write $O_p(G)$ for its maximal normal $p$-subgroup and $G'=[G,G]$ for its derived subgroup. We  write $k$ for the cyclic group of order $k$, $N.H$ for an extension of $N$ by $H$ which can be split and $N:H$ for a split extension of $N$ by $H$ where in both cases $N$ is the normal subgroup. We write $E_{p^n}$ for the elementary abelian group of order $p^n$. If $\lambda$ is a double-partition of $n$, we write $\lambda\vdash\!\vdash n$ and if $\T=(\T_1,\T_2)$ is a double-tableau associated with $\lambda$, we write $\T\in \lambda$ and we call $\T_1$ and $\T_2$ the components of $\T$. For $i\in [\![1,n]\!]$ and $\T\in \lambda$, we write $\tau_{\T}(i)=1$ if $i$ is in the left component of $\T$ and $\tau_{\T}(i)=2$ if $i$ is in the right component of $\T$. We write $n_\lambda$ for the number of standard double-tableaux associated with $\lambda$. We write $I_N$ the identity matrix and $E_{i,j}$ the elementary matrices.

\bigskip

\textbf{Acknowledgments.} This article is part of a doctoral thesis directed by O. Brunat and I. Marin. The author thanks I. Marin for a careful reading of the manuscript and suggested improvements. The author thanks O. Brunat and I. Marin for help understanding \cite{BMM} and \cite{BM} and discussions on some proofs of this article. The author also thanks K. Magaard for suggestions which were a great help to simplify proofs for the low dimensional cases. The author thanks R. Chaneb for pointing out \cite{BHRC} for maximal subgroups of low-dimensional finite classical groups.

\section{Type B}

Let $p$ be a prime, $n\in \N^\star$, $\alpha\in \overline{\F_p}$ of order $a$ greater than $n$ and not in $\{1,2,3,4,5,6,8,10\}$ and $\beta \in \overline{\F_p}\setminus\{-\alpha^i,-(n-1)\leq i\leq n-1\}$ different from $1$. We set $\F_q =\F_p(\alpha,\beta)$. The Artin group of type $B$ is the group generated by the elements $T=S_0,S_1,\dots,S_{n-1}$ verifying the relation $S_0S_1S_0S_1=S_1S_0S_1S_0$, for $i\in [\![1,n-2]\!]$, $S_iS_{i+1}S_i=S_{i+1}S_iS_{i+1}$ and for $(i,j)\in [\![0,n-1]\!]$ such that $\vert i-j\vert \geq 2$, $S_iS_j=S_jS_i$. The associated Iwahori-Hecke $\F_q$-algebra is defined by the generators indexed in the same way as for the Artin group and verifying the previous relations and deformations of the relations of order $2$ of the Coxeter groups : $(T-\beta)(T+1)=0$ and for $i\in [\![1,n-1]\!]$, $(S_i-\alpha)(S_i+1)=0$. In the following, we identify the Artin group with its image inside the Iwahori-Hecke algebra. We write $\ell_1,\ell_2$ for the length functions on $A_{B_n}=\langle T,S_i\rangle_{i\in [\![1,n-1]\!]}$ such that for all $i\in [\![1,n-1]\!]$, $\ell_1(S_i)=1, \ell_1(T)=0, \ell_2(S_i)=0$ and $\ell_2(T)=1$.

In Section $2.1$ we give the irreducible representations described by the Hoefsmit model in \cite{G-P} and \cite{HOEF} (Theorem \ref{mod})  and define a weight on standard tableaux and double-partitions of $n$ which allows us to define a bilinear form verifying nice properties (Proposition \ref{bilin}).

In Section $2.2$, we determine all the isomorphisms between different irreducible representations and then state the main results for type $B$ (Theorem \ref{result1}, \ref{result2} and \ref{result3}).

In Section $2.3$, we prove the result in all possible cases depending on the properties of the field extensions $\F_q$ of $\F_p(\alpha+\alpha^{-1},\beta+\beta^{-1})$ and $\F_p(\alpha)$ of $\F_p(\alpha+\alpha^{-1})$.

\subsection{Hoefsmit model and first properties}

\begin{theo}\label{mod}
Assume $\alpha$ is of order greater than $n$ and $\beta \in \overline{\F_p}\setminus\{-\alpha^i,-(n-1)\leq i\leq n-1\}$. The following matrix model gives a list of the pairwise non-isomorphic absolutely irreducible modules $V_\lambda$ of the Iwahori-Hecke algebra $\mathcal{H}_{n,\alpha,\beta}$ indexed by double-partitions of $n$.\\
If $\T =(\T_1,\T_2)$ is a standard double-tableau, then
\begin{itemize}
\item if $1\in \T_1, T.\T= \beta \T$ and if $1\in \T_2, T.\T = -\T$,
\item $S_i .\T = m_i(\T)\T+(1+m_i(\T))\tilde{T}$ where $\tilde{T} = \T_{i\leftrightarrow i+1}$ if $\T_{i\leftrightarrow i+1}$ is standard and $0$ otherwise.
\end{itemize}
Above we have $m_i(\T)=\frac{\alpha-1}{1-\frac{ct(\T:i)}{ct(\T:i+1)}}$, $ct(\T:j) = \alpha^{c_j(\T)-r_j(\T)}\beta$ if $j \in \T_1$ and $ct(\T:j)=-\alpha^{c_j(\T)-r_j(\T)}$ otherwise and $r_j(\T)$ (resp $c_j(\T)$) is the row (resp column) of $j$ in the tableau of $\T$ containing $j$.
\end{theo}

\begin{proof}
This theorem is a standard result and the proof in the generic case can be found in \cite{G-P} or Hoefsmit's Ph.D. thesis \cite{HOEF}. The proof in the finite field case is similar. For more details, we refer the interested reader to \cite{phD}.
\end{proof}

\bigskip

We now generalize the work done in \cite{BMM} for type $A$ to type $B$, that is we define a bilinear form which is fixed by the image of the derived subgroup of the Artin group of the Iwahori-Hecke algebra. We  define a weight by $\omega(\mathbb{T}) =\omega_1(\mathbb{T})\omega_2(\mathbb{T})\omega_3(\mathbb{T})$ where $\omega_1(\mathbb{T})= \underset{ i < j , i\in \mathbb{T}_1, j \in \mathbb{T}_2}\prod (-1)$, $\omega_2(\mathbb{T})=\underset{ i<j, i,j \in \mathbb{T}_2, r_i(\mathbb{T})>r_j(\mathbb{T})}\prod (-1)$ and $\omega_3(\mathbb{T})=\underset{ i<j, i,j \in \mathbb{T}_1, r_i(\mathbb{T})>r_j(\mathbb{T})}\prod(-1)$ and a bilinear form $(.|.)$ by $(\T|\tilde{\T}) = \omega(\T)\delta_{\T',\tilde{\T}}$ where $\T' = (\T_2',\T_1')$ for $\T=(\T_1,\T_2)$. In the same way $\lambda'=(\lambda_2',\lambda_1')$ is the transpose of $\lambda=(\lambda_1,\lambda_2)$.

If $\mu$ is a partition of $m$ with diagonal size $b(\mu) =\max\{ i, \mu_i \geq i\}$, we let $\nu(\mu) =(-1)^{\frac{m-b(\mu)}{2}}$ and if $\lambda=(\lambda_1,\lambda_2)$ is a double-partition with $\lambda_1$ a partition of $r$ and $\lambda_2$ a partition of $n-r$, we let $\tilde{\nu}(\lambda) = \nu(\lambda_1)\nu(\lambda_2)(-1)^{r(n-r)}$.
 
 We now give a proposition similar to Proposition 3.1 in \cite{BMM}.

\begin{prop}\label{bilin}
For all standard double-tableaux $\T, \tilde{\T}$, we have the following properties.
\begin{enumerate}
\item $(S_i.\T|S_i.\tilde{\T}) = (-\alpha)(\T|\tilde{\T})$ and $(T.\T|T.\tilde{\T}) = (-\beta)(\T|\tilde{\T})$.
\item
For all $b\in \mathcal{B}_n, (b.\T|b.\tilde{\T}) = (\T|\tilde{\T})$.
\item
The restriction of $(.|.)$ to $V_\lambda$ when $\lambda =\lambda'$ and to $V_\lambda \oplus V_{\lambda'}$ when $\lambda \neq \lambda'$ is non-degenerate.\\
Suppose that $\lambda=\lambda'$. Then $(.,.)$ is symmetric on $V_\lambda$ if $\tilde{\nu}(\lambda)=1$ and skew-symmetric otherwise.
Moreover, its Witt index is positive.
\end{enumerate}
\end{prop}

\begin{proof}
Let $\T$ and $\tilde{\T}$ be two double standard-tableaux.

1. We have $(T.\T|T.\tilde{\T})\neq 0\Leftrightarrow (\T|\tilde{\T})\neq 0 \Leftrightarrow \T'=\tilde{\T} \Rightarrow (T.\T|T.\tilde{\T})=-\beta(\T|\tilde{\T})$ because $\tau_\T(1)=3-\tau_{\T'}(1)$. Let $i\in [\![1,n-1]\!]$, if $\tau_{\T}(i)=\tau_{\T}(i+1)$ then by \cite[Prop 2.4.]{BMM}, we have $(S_i.\T|S_i.\tilde{\T})=(\T|\tilde{\T})$ because in the same way, we have $\omega(\T)=-\omega(\T_{i \leftrightarrow i+1})$ for any standard double-tableau $\T$ and $m_i(\T)=m_i(\T_{\tau_{\T}(i)})$ when $\tau_{\T}(i)=\tau_{\T}(i+1)$.  We now assume $\tau_{\T}(i)\neq \tau_{\T}(i+1)$. We have that $S_i.\T = m_i(\T)\T+(1+m_i(\T))\T_{i\leftrightarrow i+1}$ and $S_i.\tilde{\T}=m_i(\tilde{\T})\tilde{\T}+(1+m_i(\tilde{\T}))\tilde{\T}_{i\leftrightarrow i+1}$. It follows that

\begin{eqnarray*}
(S_i.\T|S_i.\tilde{\T}) & = & m_i(\T)m_i(\tilde{\T})(\T|\tilde{\T})+(1+m_i(\T))m_i(\tilde{\T})(\T_{i\leftrightarrow i+1},\tilde{\T})   +  m_i(\T)(1+m_i(\tilde{\T}))(\T|\tilde{\T}_{i \leftrightarrow i+1})\\
 & + & (1+m_i(\T))(1+m_i(\tilde{\T}))(\T_{i\leftrightarrow i+1}|\tilde{\T}_{i\leftrightarrow i+1}).
 \end{eqnarray*}
 
This is non-zero only if $\tilde{\T}=\T'$ or $\tilde{\T}=\T_{i\leftrightarrow i+1}'$. We now have two possible cases.

The first case is $\tilde{\T}=\T'$. We write $a=c_i-r_i+r_{i+1}-c_{i+1}$ with $(r_i,c_i)$ and $(r_{i+1},c_{i+1})$ the boxes in $\T$, we then have
\begin{eqnarray*}
(S_i.\T|S_i.\tilde{\T}) & = & m_i(\T)m_i(\T')\omega(\T)+(1+m_i(\T))(1+m_i(\T'))\omega(\T_{i\leftrightarrow i+1})\\
& = & -\omega(\T)(1+m_i(\T)+m_i(\T'))\\
& = & -\omega(\T)\left(1+\frac{\alpha-1}{1+\frac{u_{\tau(i)}}{u_{\tau(i+1)}}\alpha^{a}}+\frac{\alpha-1}{1+\frac{u_{\tau(i+1)}}{u_{\tau(i)}}\alpha^{-a}}\right)\\
& = & -\omega(\T)1+(\alpha-1)\left(\frac{1+\frac{u_{\tau(i)}}{u_{\tau(i+1)}}\alpha^{a}+1+\frac{u_{\tau(i+1)}}{u_{\tau(i)}}\alpha^{-a}}{1+\frac{u_{\tau(i)}}{u_{\tau(i+1)}}\alpha^{a}+\frac{u_{\tau(i+1)}}{u_{\tau(i)}}\alpha^{-a}+1}\right)\\
& = & -\alpha(\T|\tilde{\T}).
\end{eqnarray*}

The second case is $\tilde{\T}=\T_{i\leftrightarrow i+1}'$, we then have
\begin{eqnarray*}
(S_i.\T|S_i.\tilde{\T}) & = & (1+m_i(\T))(m_i(\T_{i\leftrightarrow i+1}'))\omega(\T_{i\leftrightarrow i+1})+m_i(\T)(1+m_i(\T_{i\leftrightarrow i+1}'))\omega(\T)\\
& = & -\omega(\T)(m_i(\T_{i\leftrightarrow i+1}')-m_i(\T))\\
& = & -\omega(\T)\left(\frac{\alpha-1}{1+\frac{u_{\tau(i)}\alpha^{r_{i+1}-c_{i+1}}}{u_{\tau(i+1)}\alpha^{r_i-c_i}}}-\frac{\alpha-1}{1+\frac{u_{\tau(i)}\alpha^{c_i-r_i}}{u_{\tau(i+1)}\alpha^{c_{i+1}-r_{i+1}}}}\right)\\
& = & 0\\
& = & -\alpha(\T|\tilde{\T}).
\end{eqnarray*}
This concludes the proof of the first result, $(2)$ follows from the first result.

3. By definition, the bilinear form is non-degenerate. Assume $\lambda=(\lambda_1,\lambda_2)=\lambda'$. We consider $\T=(\T_1|\T_2)\vdash\!\vdash n$. Since substituting $i_1<i_2<...<i_l$ in $\T_1$ by $1<2<...<l$ does not change the product and the  weight $\omega$ on $\T\in \mu$ for $\mu\vdash\!\vdash n$ satisfies $\omega(\T)\omega(\T')=\nu(\mu)$ by \cite{M2}(Lemme 6), we have  $\omega(\T)\omega(\T')=\nu(\lambda_1)\nu(\lambda_2)\underset{i<j, i\in \T_1,j\in \T_2 ~\texttt{or}~ i\in \T_2,j\in \T_1}\prod (-1)$. The cardinal of the set  $\{i<j,i\in \T_1,j\in \T_2 ~\mbox{or}~ i\in \T_2,j\in \T_1\}$ is equal to the number of pairs $(i,j)$ with $i$ in $\T_1$ and $j$ in $\T_2$, which equals $(\frac{n}{2})^2$. It follows that $\omega(\T)\omega(\T') =\tilde{\nu}(\lambda)$ for any standard double-tableau $\T$ associated with $\lambda$. For any pair $(\T,\tilde{\T})$, we have that
 $$(\tilde{\T}|\T)=\omega(\tilde{\T})\delta_{\T,\tilde{\T}'}=\tilde{\nu}(\lambda)\omega(\tilde{\T}')\delta_{\T,\tilde{\T}'}=\tilde{\nu}(\lambda)\omega(\T)\delta_{\tilde{\T},\T'}=\tilde{\nu}(\lambda)(\T|\tilde{\T}).$$
 The Witt index is positive since the basis can be partitioned in pairs $(\T|\T')$.
 \end{proof}

We remark that we have proved that $\omega(\T)\omega(\T')=\tilde{\nu}(\lambda)$ for any double-partition $\lambda$ and standard double-tableau $\T$ in $V_\lambda$.

\subsection{Factorization of the image of the Artin group in the Iwahori-Hecke algebra}

\subsubsection{Isomorphisms between representations}

Let $\mathcal{L}\in End(V)$ be defined by  $\mathbb{T}\mapsto \omega(\T)\T'$. We give a generalization of Lemma 3.2. of \cite{BMM}.

\begin{prop}\label{transpose}
Let $\lambda$ be a double-partition of $n$ such that $\lambda \neq \lambda'$ (resp $\lambda=\lambda'$). $\mathcal{L}$ induces an endomorphism of $V_{\lambda} \oplus V_{\lambda'}$ (resp $V_\lambda$) which switches $V_{\lambda}$ and $V_{\lambda'}$ (resp leaves $V_\lambda$ stable) such that the actions of $S_r$ and $T$ satisfy
$$\mathcal{L}S_r\mathcal{L}^{-1}(-\alpha)^{-1} = {}^t\!S_r^{-1}, \mathcal{L}T\mathcal{L}^{-1}(-\beta)^{-1} = {}^t\!T^{-1}.$$
\end{prop}

\begin{proof}
This follows directly from Proposition \ref{bilin} by writing the matrix of the bilinear form and the matrix of $\mathcal{L}$. 
\end{proof}
We now suppose $\F_p(\alpha,\beta)= \F_p(\alpha+\alpha^{-1},\beta) =\F_p(\alpha,\beta+\beta^{-1})\neq \F_p(\alpha+\alpha^{-1}, \beta+\beta^{-1})$.

We then have an $\F_q$-automorphism $\epsilon$ of order $2$ such that $\epsilon(\alpha)=\alpha^{-1}$ and $\epsilon(\beta)=\beta^{-1}$.

 Let $\langle .,.\rangle $ be the hermitian form defined by $\langle\mathbb{T},\tilde{\mathbb{T}}\rangle = d(\mathbb{T})\delta_{\mathbb{T},\tilde{\mathbb{T}}}$ where $$d(\mathbb{T})=\tilde{d}(\mathbb{T}_1)\tilde{d}(\mathbb{T}_2)\underset{i\in \T_1, j\in \mathbb{T}_2, i<j}\prod \frac{2+\beta\alpha^{a_{i,j}-1}+\beta^{-1}\alpha^{1-a_{i,j}}}{\alpha+\alpha^{-1}+\beta\alpha^{a_{i,j}}+\beta^{-1}\alpha^{-a_{i,j}}}$$
with $a_{i,j} = c_i-r_i+r_j-c_j$ and $\tilde{d}$ induced by the $d$ defined in  \cite{BMM} applied to $\mathbb{T}_1$ and $\mathbb{T}_2$ by seeing them as standard tableaux using the ordered bijections onto $[\![1,r]\!]$ and $[\![1,n-r]\!]$.

 We now check that $d(\T)$ is well defined and non-zero for any standard double-tableau. We prove in what follows that the big product in the expression of $d$ is indeed well-defined and non-zero for any double-tableau with no empty components.
 
 Let $\lambda \vdash\!\vdash n, \T=(\T_1,\T_2)\in \lambda$ and $(i,j)$ a pair of integers such that $i<j, i\in \T_1$ and $j\in \T_2$. We set $r$ to be the number of boxes of $\T_1$, we have $1-r\leq c_i-r_i\leq r-1$ and $1-(n-r)\leq c_j-r_j \leq n-r-1$ so $2-n \leq a=a_{i,j}\leq n-2$. We have $\alpha+\alpha^{-1}+\beta\alpha^{a}+\beta^{-1}\alpha^{-a} = \alpha(1+\beta\alpha^{a-1})+\alpha^{-a}\beta^{-1}(1+\beta\alpha^{a-1}) =\alpha(1+\beta\alpha^{a-1})(1+\alpha^{-a-1}\beta^{-1}).$ This product never cancels because $\beta\notin \{-\alpha^i,1-n\leq i \leq n-1\}$. In the same way $2+\beta\alpha^{a-1}+\beta^{-1}\alpha^{1-a}=(1+\beta\alpha^{a-1})(1+\beta^{-1}\alpha^{1-a})$ never cancels.  
 
 Now we have defined this hermitian form, we can generalize Proposition $3.6$ from \cite{BMM}.

\begin{prop}\label{unitary}
The group $A_{B_n}$ acts in a unitary way on $V$ with respect to this hermitian form and this form is non-degenerate on $V_\lambda$ for any double-partition $\lambda$ of $n$. In particular, for any double-partition $\lambda$ of $n$, there exists a matrix $P\in GL_{n_\lambda}(q)$ such that $PR_{\lambda}(T)P^{-1}=\epsilon(R_{\lambda}^\star(T))=R_{\lambda}(T)$ and $PR_{\lambda}(S_r)P^{-1}=\epsilon(R_{\lambda}^{\star}(S_r))$.
\end{prop}

\begin{proof}
The action of $T$ is indeed unitary with regards to this hermitian form  because $\beta\epsilon(\beta)=(-1)\epsilon(-1)=1$. Let $\T$ be a standard double-tableau and $r\in [\![1,n-1]\!]$. If $\tau_{\T}(r)=\tau_{\T}(r+1)$ then the result is a consequence of Proposition 3.6. in \cite{BMM}.

We now assume that $\tau_{\T}(r)\neq\tau_{\T}(r+1)$, up to switching $\T$ and $\T_{r\leftrightarrow r+1}$ we can assume that $\tau_{\T}(r)=1$ and $\tau_{\T}(r+1)=2$.  It remains to show that $\langle\T,\T\rangle=\langle S_r.\T,S_r.\T\rangle, \langle\T_{r\leftrightarrow r+1},\T_{r\leftrightarrow r+1}\rangle = \langle S_r.\T_{r\leftrightarrow r+1},S_r.\T_{r\leftrightarrow r+1}\rangle$ and  $\langle S_r.\T,S_r.\T_{r\leftrightarrow r+1}\rangle =\langle \T,\T_{r\leftrightarrow r+1}\rangle $. In the following computation, we write $a=a_{i,i+1}$ and $\tilde{\T}=\T_{r\leftrightarrow r+1}$. We have
\begin{eqnarray*}
 \langle S_r.\T,S_r.\T\rangle & = & m_r(\T)\epsilon(m_r(\T))d(\T)+(1+m_r(\T))\epsilon(1+m_r(\T))d(\tilde{\T})\\
 & = & d(\T)(m_r(\T)\epsilon(m_r(\T))+\left(\frac{\alpha+\alpha^{-1}+\beta\alpha^{a}+\beta^{-1}\alpha^{-a}}{2+\beta\alpha^{a-1}+\beta^{-1}\alpha^{1-a}}(1+m_r(\T))\epsilon(1+m_r(\T)))\right)\\
 & = & d(\T)\left(\frac{\alpha-1}{1+\beta\alpha^{a}}\frac{\alpha^{-1}-1}{1+\beta^{-1}\alpha^{-a}}+\frac{\alpha+\alpha^{-1}+\beta\alpha^{a}+\beta^{-1}\alpha^{-a}}{2+\beta\alpha^{a-1}+\beta^{-1}\alpha^{1-a}}\frac{\alpha+\beta\alpha^{a}}{1+\beta\alpha^{a}}\frac{\alpha^{-1}+\beta^{-1}\alpha^{-a}}{1+\beta^{-1}\alpha^{-a}}\right)\\
 & = & d(\T)\left(\frac{2-\alpha-\alpha^{-1}}{2+\beta\alpha^a+\beta^{-1}\alpha^{-a}}+\frac{\alpha+\alpha^{-1}+\beta\alpha^{a}+\beta^{-1}\alpha^{-a}}{2+\beta\alpha^a+\beta^{-1}\alpha^{-a}}\right)\\
 & = & d(\T)\\
 & = & \langle\T,\T\rangle.
\end{eqnarray*}
We also have
\begin{eqnarray*}
\langle S_r.\tilde{\T},S_r.\tilde{\T}\rangle & = & m_r(\tilde{\T})\epsilon(m_r(\tilde{\T}))d(\tilde{\T})+(1+m_r(\tilde{\T}))\epsilon(1+m_r(\tilde{\T}))d(\T)\\
& = & d(\tilde{\T})\left(\frac{\alpha-1}{1+\beta^{-1}\alpha^{-a}}\frac{\alpha^{-1}-1}{1+\beta\alpha^a}+\frac{2+\beta\alpha^{a-1}+\beta^{-1}\alpha^{1-a}}{\alpha+\alpha^{-1}+\beta\alpha^a+\beta^{-1}\alpha^{-a}}\frac{\alpha+\beta^{-1}\alpha^{-a}}{1+\beta^{-1}\alpha^{-a}}\frac{\alpha^{-1}+\beta\alpha^a}{1+\beta\alpha^a}\right)\\
& = & d(\tilde{\T})\left(\frac{2-\alpha-\alpha^{-1}}{2+\beta\alpha^a+\beta^{-1}\alpha^{-a}}+\right.\\
& &\left. \frac{4+2\beta\alpha^{a+1}+2\beta^{-1}\alpha^{-a-1}+2\beta\alpha^{a-1}+\beta^2\alpha^{2a}+\alpha^{-2}+2\beta^{-1}\alpha^{1-a}+\alpha^2+\beta^{-2}\alpha^{-2a}}{(\alpha+\alpha^{-1}+\beta\alpha^a+\beta^{-1}\alpha^{-a})(2+\beta\alpha^a+\beta^{-1}\alpha^{-a})}\right)\\
& = & d(\tilde{\T})\left(\frac{2-\alpha-\alpha^{-1}}{2+\beta\alpha^a+\beta^{-1}\alpha^{-a}}+\frac{(\alpha+\alpha^{-1}+\beta\alpha^a+\beta^{-1}\alpha^{-a})^2}{(\alpha+\alpha^{-1}+\beta\alpha^a+\beta^{-1}\alpha^{-a})(2+\beta\alpha^a+\beta^{-1}\alpha^{-a})}\right)\\
& = & d(\tilde{\T})\\
& = & \langle\tilde{\T},\tilde{\T}\rangle.
\end{eqnarray*}
Finally, we have
\begin{eqnarray*}
\langle S_r.\T,S_r.\tilde{\T}\rangle  & = & m_r(\T)\epsilon(1+m_r(\tilde{\T}))d(\T)+(1+m_r(\T))\epsilon(m_r(\tilde{\T}))d(\tilde{\T})\\
& = & d(\T)\left(\frac{\alpha-1}{1+\beta\alpha^a}\frac{\alpha^{-1}+\beta\alpha^a}{1+\beta\alpha^a}+\frac{\alpha+\alpha^{-1}+\beta\alpha^{a}+\beta^{-1}\alpha^{-a}}{2+\beta\alpha^{a-1}+\beta^{-1}\alpha^{1-a}}\frac{\alpha+\beta\alpha^a}{1+\beta\alpha^a}\frac{\alpha^{-1}-1}{1+\beta\alpha^a}\right)\\
& = &d(\T)\frac{\alpha-1}{(1+\beta\alpha^a)^2(2+\beta\alpha^{a-1}+\beta^{-1}\alpha^{1-a})}\left((\alpha^{-1}+\beta\alpha^a)(2+\beta\alpha^{a-1}+\beta^{-1}\alpha^{1-a})-\right.\\
& &\left. \alpha^{-1}(\alpha+\alpha^{-1}+\beta\alpha^a+\beta^{-1}\alpha^{-a})(\alpha+\beta\alpha^a)\right)\\
& = & d(\T)\frac{\alpha-1}{(1+\beta\alpha^a)^2(2+\beta\alpha^{a-1}+\beta^{-1}\alpha^{1-a})}\left(2\alpha^{-1}+\beta\alpha^{a-2}+\beta^{-1}\alpha^{-a}+2\beta\alpha^a+\right.\\
& & \left.\beta^2\alpha^{2a-1}+\alpha-\alpha-\alpha^{-1} -\beta\alpha^a-\beta^{-1}\alpha^{-a}-\beta\alpha^a-\beta\alpha^{a-2}-\beta^{2}\alpha^{2a-1}-\alpha^{-1}\right)\\
& = & 0\\
& = & \langle \T,\tilde{\T}\rangle.
\end{eqnarray*}
\end{proof}

\bigskip
We recall that $\mathcal{A}_{B_n} = [A_{B_n},A_{B_n}]$ is the derived subgroup of $A_{B_n}$. When it exists, we write $\epsilon$ for the automorphism of order $2$ of $\F_q= \F_p(\alpha,\beta)$.

\begin{lemme}\label{Lincoln}
If $\lambda$ is a double-partition of $n$ then the restriction of  $R_{\lambda}$ to $\mathcal{A}_{B_n}$ is absolutely irreducible.
\end{lemme}

\begin{proof}
For $n\leq 3$, this is a consequence of the results on the representations with two columns by \cite{M} (Proposition 5). Since $A_{B_n}$ is generated by $A_{B_{n-1}}$ and $\mathcal{A}_{B_n}$, we have the result by the same method as in the Lemma 3.4(i) of \cite{BMM}.
\end{proof}

We now recall Lemma $2.2$ of \cite{BMM}.
\begin{lemme}\label{abel}
Let $G$ be a group, $k$ a field and $R_1,R_2$ two representations of $G$ in $GL_N(k)$ such that the restrictions to the derived subgroup of $G$ are equal and the restriction of at least one of them is absolutely irreducible. There exists a character $\eta : G \rightarrow k^\star$ such that $R_2 = R_1 \otimes \eta$.
\end{lemme}

\begin{prop}
If $\lambda_1$ a partition of $n$ then $R_{(\lambda_1,\emptyset)|\mathcal{A}_{B_n}}\simeq R_{(\emptyset,\lambda_1)|\mathcal{A}_{B_n}}$.
\end{prop}

\begin{proof}
The action of $T$ is diagonal and the action of $S_i$ on $(\T_1,\emptyset)$ is identical to the one on $(\emptyset,\T_1)$ so the proof of the result is straightforward.
\end{proof}

\begin{prop}\label{isomorphisme}
Let $\lambda$ and $\mu$ be double-partitions of $n$ with no empty components. We then have the following properties.
\begin{enumerate}
\item If $\F_q = \F_p(\alpha,\beta)=\F_p(\alpha+\alpha^{-1},\beta+\beta^{-1})$, then
                \begin{enumerate}
                 \item $R_{\lambda|\mathcal{A}_{B_n}} \simeq R_{\mu|\mathcal{A}_{B_n}} \Leftrightarrow \lambda = \mu$,
                 \item $R_{\lambda|\mathcal{A}_{B_n}} \simeq R_{\mu|\mathcal{A}_{B_n}}^{\star} \Leftrightarrow \lambda = \mu'$.
                 \end{enumerate}
\item If $\F_q=\F_p(\alpha,\beta)=\F_p(\alpha+\alpha^{-1},\beta) = \F_p(\alpha,\beta+\beta^{-1}) \neq \F_p(\alpha+\alpha^{-1},\beta+\beta^{-1})$, then 
                \begin{enumerate}
               \item $R_{\lambda|\mathcal{A}_{B_n}} \simeq R_{\mu|\mathcal{A}_{B_n}} \Leftrightarrow \lambda = \mu$,
                 \item $R_{\lambda|\mathcal{A}_{B_n}} \simeq R_{\mu|\mathcal{A}_{B_n}}^{\star} \Leftrightarrow \lambda = \mu'$,
                 \item   $ R_{\lambda|\mathcal{A}_{B_n}} \simeq \epsilon \circ R_{\mu|\mathcal{A}_{B_n}} \Leftrightarrow  \lambda =\mu'$,
                 \item $R_{\lambda|\mathcal{A}_{B_n}} \simeq \epsilon\circ R_{\mu|\mathcal{A}_{B_n}}^{\star} \Leftrightarrow   \lambda=\mu$.
                 \end{enumerate}
\item If $\F_q=\F_p(\alpha,\beta)=\F_p(\alpha,\beta+\beta^{-1}) \neq \F_p(\alpha+\alpha^{-1},\beta) = \F_p(\alpha+\alpha^{-1},\beta+\beta^{-1})$, then 
                \begin{enumerate}
                 \item $R_{\lambda|\mathcal{A}_{B_n}} \simeq R_{\mu|\mathcal{A}_{B_n}} \Leftrightarrow \lambda = \mu$,
                 \item $R_{\lambda|\mathcal{A}_{B_n}} \simeq R_{\mu|\mathcal{A}_{B_n}}^{\star} \Leftrightarrow \lambda = \mu'$,
                 \item   $ R_{\lambda|\mathcal{A}_{B_n}} \simeq \epsilon \circ R_{\mu|\mathcal{A}_{B_n}} \Leftrightarrow  (\lambda_1,\lambda_2) =(\mu_1',\mu_2')$,
                 \item $R_{\lambda|\mathcal{A}_{B_n}} \simeq \epsilon\circ R_{\mu|\mathcal{A}_{B_n}}^{\star} \Leftrightarrow   (\lambda_1,\lambda_2)=(\mu_2,\mu_1)$.
                 \end{enumerate}
\item If $\F_q=\F_p(\alpha,\beta)=\F_p(\alpha+\alpha^{-1},\beta) \neq \F_p(\alpha,\beta+\beta^{-1}) = \F_p(\alpha+\alpha^{-1},\beta+\beta^{-1})$, then 
                \begin{enumerate}
               \item $R_{\lambda|\mathcal{A}_{B_n}} \simeq R_{\mu|\mathcal{A}_{B_n}} \Leftrightarrow \lambda = \mu$,
                 \item $R_{\lambda|\mathcal{A}_{B_n}} \simeq R_{\mu|\mathcal{A}_{B_n}}^{\star} \Leftrightarrow \lambda = \mu'$,
                 \item   $ R_{\lambda|\mathcal{A}_{B_n}} \simeq \epsilon \circ R_{\mu|\mathcal{A}_{B_n}} \Leftrightarrow  (\lambda_1,\lambda_2) =(\mu_2,\mu_1)$,
                 \item $R_{\lambda|\mathcal{A}_{B_n}} \simeq \epsilon\circ R_{\mu|\mathcal{A}_{B_n}}^{\star} \Leftrightarrow   (\lambda_1,\lambda_2)=(\mu_1',\mu_2')$.
                 \end{enumerate}
\end{enumerate}
\end{prop}

\begin{proof}
In all cases,  $(a)$ and $(b)$ are the same and the proof is identical.

(a) By the preceding lemmas, it is sufficient to show that if there exists $\eta : A_{B_n} \rightarrow \F_q^{\star}$ such that $R_\lambda \simeq R_\mu \otimes \eta$ then $\lambda = \mu$. Assume such a character exists, since the abelianization of $A_{B_n}$ is $<\overline{T},\overline{S_1}> \simeq\Z^2$, up to conjugation we have $R_\lambda(b) = R_\mu(b)u^{\ell_1(b)}v^{\ell_2(b)}$ for some $u,v\in \F_q^\star$. Taking the eigenvalues of $S_1$ and $T$ on both sides of the equality we get that $\{\alpha,-1\} = \{u\alpha,-u\}$ and $\{\beta,-1\}=\{v\beta,-v\}$. Since $\alpha^2\neq 1\neq \beta^2$, we have $u=v=1$ which implies that $R_\lambda$ and $R_\mu$ are isomorphic representations. By theorem \ref{mod}, this implies $\lambda=\mu$.

\bigskip

(b) The implication $\lambda=\mu' \Rightarrow R_{\lambda|\mathcal{A}_{B_n}} \simeq R_{\mu|\mathcal{A}_{B_n}}^{\star}$ follows from Proposition \ref{transpose}.

Assume now $R_{\lambda|\mathcal{A}_{B_n}}\simeq R_{\mu|\mathcal{A}_{B_n}}^{\star}$, we then have $R_{\lambda'|\mathcal{A}_{B_n}}\simeq R_{\lambda|\mathcal{A}_{B_n}}^\star \simeq (R_{\mu|\mathcal{A}_{B_n}}^{\star})^{\star} = R_{\mu|\mathcal{A}_{B_n}}$. The result follow from $(a)$. 

In the same way, it is enough to show the converse implication in the remainder of the proof.

2. d) This result follows directly from Proposition \ref{unitary}.

2. c) By  2. d) and 2. b), we have 
$$\epsilon \circ R_{\lambda'|\mathcal{A}_{B_n}} \simeq \epsilon \circ (\epsilon \circ R_{\lambda'|\mathcal{A}_{B_n}}^{\star}) = R_{\lambda'|\mathcal{A}_{B_n}}^{\star} \simeq R_{\lambda|\mathcal{A}_{B_n}}.$$

3) In this case $\epsilon(\alpha) = \alpha^{-1}$ and $\epsilon(\beta) = \beta$.

3. c) For every standard double-tableau $\T = (\T_1,\T_2)$, we define $\tilde{\T}$ by $\tilde{\T}=(\T_1',\T_2')$. Let $\eta : A_{B_n} \rightarrow \F_q^{\star}$ be the character of $A_{B_n}$ defined by $\eta(T) =1$ and $\eta(S_r) =-\alpha$ for all $r$.

Let $Q : V_{(\lambda_1,\lambda_2)} \rightarrow V_{(\lambda_1',\lambda_2')},(\T_1,\T_2)\mapsto (\T_1',\T_2'), U : V_\lambda \rightarrow V_\lambda, \T \mapsto \omega(\T)\T$.

Using the same notations as in Proposition \ref{unitary}, we will show that for all $r\in [\![1,n-1]\!]$ :
$$Q^{-1}(-\alpha)\epsilon(R_{\lambda_1',\lambda_2'}(S_r))Q = U^{-1}PR_{\lambda}(S_r)P^{-1}U, Q^{-1}\epsilon(R_{\lambda_1',\lambda_2'}(T))Q = U^{-1}PR_{\lambda}(T)P^{-1}U.$$

Let $\T= (\T_1,\T_2)$ be a standard double-tableau. The second equality follows from $PR_{\lambda}(T)P^{-1} =R_{\lambda}(T)$ and $\epsilon(\beta) = \beta$. If $\T_{r\leftrightarrow r+1}$ is non-standard, the first equality is verified by $S_r$. Assume $\T_{r\leftrightarrow r+1}$  is standard, write $a= a_{r,r+1}$. If $\tau_{\T}(r)=\tau_{\T}(r+1)$ then in the basis $(\T,\T_{r\leftrightarrow r+1})$, we have :
$$R_{\lambda}(S_r) =\begin{pmatrix}
 \frac{\alpha-1}{1-\alpha^a} & \frac{\alpha-\alpha^{-a}}{1-\alpha^{-a}} \\
 \frac{\alpha+\alpha^a}{1-\alpha^a} & \frac{\alpha-1}{1-\alpha^{-a}}
\end{pmatrix}, 
 U^{-1}PR_{\lambda}(S_r)P^{-1}U = -\alpha\begin{pmatrix}
\frac{\alpha^{-1}-1}{1-\alpha^a} & \frac{\alpha^{-1}-\alpha^{-a}}{1-\alpha^{-a}}\\
\frac{\alpha^{-1}-\alpha^{a}}{1-\alpha^a} & \frac{\alpha^{-1}-1}{1-\alpha^{-a}}
\end{pmatrix},$$
$$-\alpha Q^{-1}\epsilon(R_{\lambda_1',\lambda_2'}(S_r))Q = -\alpha\epsilon(\begin{pmatrix} \frac{\alpha-1}{1-\alpha^{-a}} & \frac{\alpha-\alpha^a}{1-\alpha^a}\\
\frac{\alpha-\alpha^{-a}}{1-\alpha^{-a}} & \frac{\alpha-1}{1-\alpha^a}\end{pmatrix} = U^{-1}PR_{\lambda}(S_r)P^{-1}U.$$

If $\tau_{\T}(r)=1$ and $\tau_{\T}(r+1)=2$ then we have

$$R_\lambda(S_r) =\begin{pmatrix}
\frac{\alpha-1}{1+\beta\alpha^a} & \frac{\alpha+\beta^{-1}\alpha^{-a}}{1+\beta^{-1}\alpha^{-a}}\\
\frac{\alpha+\beta\alpha^a}{1+\beta\alpha^a} & \frac{\alpha-1}{1+\beta^{-1}\alpha^{-a}}
\end{pmatrix}, U^{-1}PR_{\lambda}(S_r)P^{-1}U = -\alpha\begin{pmatrix}
\frac{\alpha^{-1}-1}{1+\beta\alpha^a} & \frac{\beta^{-1}\alpha^{-a}+\alpha^{-1}}{1+\beta^{-1}\alpha^{-a}}\\
\frac{\beta\alpha^a+\alpha^{-1}}{1+\beta\alpha^a} & \frac{\alpha^{-1}-1}{1+\beta^{-1}\alpha^{-a}}
\end{pmatrix},$$
$$-\alpha Q^{-1}\epsilon(R_{\lambda_1',\lambda_2'})Q = -\alpha\epsilon(\begin{pmatrix} \frac{\alpha-1}{1+\beta\alpha^{-a}} & \frac{\alpha+\beta^{-1}\alpha^a}{1+\beta^{-1}\alpha^a}\\
\frac{\alpha+\beta\alpha^{-a}}{1+\beta\alpha^{-a}} & \frac{\alpha-1}{1+\beta^{-1}\alpha^a}\end{pmatrix} = U^{-1}PR_{\lambda}(S_r)P^{-1}U.$$
It follows that $ R_{(\lambda_1,\lambda_2)|\mathcal{A}_{B_n}} \simeq \epsilon \circ R_{(\lambda_1',\lambda_2')|\mathcal{A}_{B_n}}$.

3. d) This is a consequence of 3. c) and 3. b).

4. c, 4. d) The proof of these results is analogous to the ones of 3. d) and 3. c).
\end{proof}

For $r\in [\![1,n-1]\!]$, we define the double-partitions $\lambda_{(r)}=([1^{n-r}],[r])$ and $\lambda^{(r)} = ([r],[1^{n-r}])$. The following proposition is a generalization of Proposition $3.5$ of \cite{BMM}.

\begin{prop}\label{patate}
For $r\in [\![1,n-1]\!]$, there exists $\eta_{1,r},\eta_{2,r} : A_{B_n} \rightarrow \F_q^\star$ such that 

$R_{\lambda_{(r)}} \simeq (\Lambda^rR_{\lambda_{(1)}}) \otimes \eta_{1,r}$ and $R_{\lambda^{(r)}}\simeq (\Lambda^rR_{\lambda^{(1)}}) \otimes \eta_{2,r}$ .\end{prop}

\begin{proof}
Every double-tableau associated with $\lambda$ can be mapped in a one-to-one way to a set $\{i_1,i_2,...,i_r\}\subset [\![1,n]\!]$ such that $i_1<i_2<...<i_r$ where $i_k$ is the number in the $k$-th box of the right component. We write $v_I$ the corresponding double-tableau and $v_i=v_{\{i\}}$.\\
After computation, we get the following for $k\in [\![1,n-1]\!]$.
\begin{enumerate}
\item If $1\in I$ then $R_{\lambda_{(r)}}(T)v_I = -v_I$.
\item If $1\notin I$ then $R_{\lambda_{(r)}}(T)v_I = \beta v_I.$
\item If $k,k+1\notin I$ then $R_{\lambda_{(r)}}(S_k)v_I = -v_I$.
\item If $k,k+1\in I$ then $R_{\lambda_{(r)}}(S_k)v_I=\alpha v_I$.
\item If $k\in I, k+1\notin I$ then $R_{\lambda_{(r)}}(S_k)v_I = \frac{\alpha-1}{1+\beta^{-1}\alpha^{k-1}}v_I+\frac{\alpha+\beta^{-1}\alpha^{k-1}}{1+\beta^{-1}\alpha^{k-1}}v_{I\Delta \{k,k+1\}}.$
\item If $k\notin I, k+1\in I$ then $R_{\lambda_{(r)}}(S_k)v_I = \frac{\alpha-1}{1+\beta\alpha^{1-k}}v_I +\frac{\alpha+\beta\alpha^{1-k}}{1+\beta\alpha^{1-k}}v_{I\Delta \{k,k+1\}}.$
\end{enumerate}
Above, $\Delta$ is the symmetric difference : $A\Delta B = (A\cup B)\setminus (A\cap B)$.

To each set $I=\{i_1,i_2,...,i_r\}$ can be given in a one-to-one way an element $u_I$ of $\Lambda^rR_{\lambda_{(1)}}$ writing\\ $u_I = v_{i_1} \wedge v_{i_2} \wedge ... \wedge v_{i_r}$ and these $u_I$ form a basis. For $k\in [\![1,n-1]\!]$, we have the following.

\begin{enumerate}
\item If $1\in I$ then $\Lambda^rR_{\lambda_{(1)}}(T)u_I = -\beta^{r-1}u_I$.
\item If $1\notin I$ then $\Lambda^rR_{\lambda_{(1)}}(T)u_I = \beta^r u_I.$
\item If $k,k+1\notin I$ then $\Lambda^rR_{\lambda_{(1)}}(S_k)u_I = (-1)^r u_I$.
\item If $k,k+1\in I$ then $\Lambda^rR_{\lambda_{(1)}}(S_k)u_I=(-1)^{r-1}\alpha u_I$.
\item If $k\in I, k+1\notin I$ then $\Lambda^rR_{\lambda_{(1)}}(S_k)u_I = (-1)^{r-1}\frac{\alpha-1}{1+\beta^{-1}\alpha^{k-1}}u_I+(-1)^{-r-1}\frac{\alpha+\beta^{-1}\alpha^{k-1}}{1+\beta^{-1}\alpha^{k-1}}u_{I\Delta \{k,k+1\}}.$
\item If $k\notin I, k+1\in I$ then $\Lambda^rR_{\lambda_{(1)}}(S_k)u_I = (-1)^{r-1}\frac{\alpha-1}{1+\beta\alpha^{1-k}}u_I +(-1)^{r-1}\frac{\alpha+\beta\alpha^{1-k}}{1+\beta\alpha^{1-k}}u_{I\Delta \{k,k+1\}}.$
\end{enumerate}

Looking at the basis change $v_I \mapsto u_I$ and the character $\eta_{1,r}(h) =(-1)^{(r-1)\ell_1(h)}\beta^{(r-1)\ell_2(h)}$, we have the first part of the proposition. In the same way, writing $\eta_{2,r}(h) =(-1)^{(r-1)\ell_1(h)}(-1)^{(r-1)\ell_2(h)}$, we have the second part of the proposition.
\end{proof}

\subsubsection{Factorization depending on the field}\label{lalala}

The result depends on the properties of the field extension $\F_q=\F_p(\alpha,\beta)$ of $\F_p(\alpha+\alpha^{-1},\beta+\beta^{-1})$ and the field extension $\F_{\tilde{q}}=\F_p(\alpha)$ of $\F_p(\alpha+\alpha^{-1})$. By elementary field theory, we have the following possibilities.
\begin{enumerate}
\item $\F_q=\F_p(\alpha,\beta)= \F_p(\alpha+\alpha^{-1},\beta+\beta^{-1})$ and $\F_p(\alpha)=\F_p(\alpha+\alpha^{-1})$.
\item $\F_q=\F_p(\alpha,\beta) = \F_p(\alpha+\alpha^{-1},\beta+\beta^{-1})$ and $\F_p(\alpha) \neq \F_p(\alpha+\alpha^{-1})$.
\item $\F_q=\F_p(\alpha,\beta)=\F_p(\alpha+\alpha^{-1},\beta)= \F_p(\alpha,\beta+\beta^{-1})\neq \F_p(\alpha+\alpha^{-1},\beta+\beta^{-1})$.
\item $\F_q = \F_p(\alpha,\beta) = \F_p(\alpha,\beta+\beta^{-1})\neq \F_p(\alpha+\alpha^{-1},\beta) = \F_p(\alpha+\alpha^{-1},\beta+\beta^{-1})$.
\item $\F_q = \F_p(\alpha,\beta) =\F_p(\alpha+\alpha^{-1},\beta) \neq \F_p(\alpha,\beta+\beta^{-1}) = \F_p(\alpha+\alpha^{-1},\beta+\beta^{-1})$ and $\F_p(\alpha) \neq \F_p(\alpha+\alpha^{-1})$.
\item $\F_q = \F_p(\alpha,\beta) = \F_p(\alpha+\alpha^{-1},\beta) \neq \F_p(\alpha,\beta+\beta^{-1})=\F_p(\alpha+\alpha^{-1},\beta+\beta^{-1})$ and $\F_p(\alpha)= \F_p(\alpha+\alpha^{-1})$.
\end{enumerate}

We remark that in the third and fourth cases, we have $\F_p(\alpha)\neq \F_p(\alpha+\alpha^{-1})$.

Before stating the main results for type $B$, we recall the two following lemmas, the first one is Lemma $2.4$ of \cite{BM} and the proof of the second one is included in the proof of Proposition $4.1$ of \cite{BMM}.

\begin{lemme}\label{Harinordoquy}
Let $\rho$ be an absolutely irreducible representation of a group $G$ in $GL_r(q)$ where $\F_q$ is a finite field such that there exists an automorphism $\epsilon$ of order $2$ of $\F_q$. If $\rho \simeq \epsilon \circ \rho^\star$, then there exists $S\in GL_r(q)$ such that $S^{-1}\rho(g)S \in GU_r(q^\frac{1}{2})$ for all $g \in G$.
 \end{lemme}

\begin{lemme}\label{Ngwenya}
Let $\rho$ and $G$ be as in the previous lemma. If $\rho \simeq \epsilon \circ \rho$, then there exists $S\in GL_r(q)$ such that $S^{-1}\rho(g)S \in GL_r(q^\frac{1}{2})$ for all $g \in G$.
 \end{lemme}

\bigskip

In certain cases, $(\lambda_1,\lambda_2)$ factorizes through $(\lambda_2,\lambda_1)$ or $(\lambda_1',\lambda_2')$ so we need a good order on double-partitions of $n$. We first choose for $r\leq n$ an order on partitions of $r$ such that if $r$ has $2l$ partitions different from their transpose $\{a_i,a_{i'}\}_{i\in [\![ 1,l]\!]}$ and $s$ partitions $\{a_{l+i}\}_{i\in [\![1,s]\!]}$ equal to their transpose then $a_1<a_1'<a_2<a_2'<...<a_l<a_l'<a_{l+1}=a_{l+1}'<...<a_{l+s}=a_{l+s}'$.  We also require $a_1=[r]$. This gives us that $\lambda<\mu$ implies that $\lambda'<\mu'$ whenever $\lambda \neq \mu'$. If $\lambda \vdash\!\vdash n_1$ and $\mu \vdash\!\vdash n_2$ then we say $\lambda >\mu$ if $n_1>n_2$ or $n_1=n_2$ and $\lambda>\mu$. We then define the order $<$ on double-partitions of $n$ in the following way where $\lambda_1$ is a partition of $r_\lambda$ : $(\lambda_1,\lambda_2)<(\mu_1,\mu_2)$ if $r_\lambda< r_\mu$ or ($r_\lambda = r_\mu$ and $\lambda_1< \mu_1$) or ($r_\lambda =r_\mu, \lambda_1=\mu_1$ and $\lambda_2 < \mu_2$).

\begin{lemme}
If $\lambda=(\lambda_1,\lambda_2)$ is a double-partition such that $\lambda \neq \lambda', \lambda \neq (\lambda_2,\lambda_1)$ and $\lambda \neq (\lambda_1',\lambda_2')$ then exactly one of those double-partitions verifies the property :

 $$ (*)  \lambda < \lambda' \mbox{ and } \lambda < (\lambda_1',\lambda_2').$$
\end{lemme}

\begin{proof}
Let $\lambda=(\lambda_1,\lambda_2)$ be a double-partition verifying the conditions in the lemma.
 Assume $\lambda > \lambda'$ and $\lambda < (\lambda_1',\lambda_2')$. Since $\lambda > \lambda'$, we have $r_\lambda \geq r_{\lambda'}$ and since $r_\lambda +r_\lambda' = n$, we get $r_\lambda \geq \frac{n}{2}$. Let's show that either ($\lambda' < \lambda$ and $\lambda' < (\lambda_2,\lambda_1)$) or ($(\lambda_2,\lambda_1) < (\lambda_1',\lambda_2')$ and $(\lambda_2,\lambda_1) < (\lambda_2',\lambda_1')$), i.e. either $\lambda'$ verifies $(*)$ or $(\lambda_2,\lambda_1)$ verifies $(*)$.
 Those two cases are indeed distinct because either $\lambda' < (\lambda_2,\lambda_1)$ or $(\lambda_2,\lambda_1) < \lambda'$. If $\lambda' < (\lambda_2,\lambda_1)$ then we are in the first case because we assumed $\lambda > \lambda'$. Let's now assume $\lambda' > (\lambda_2,\lambda_1)$, we must show $(\lambda_2,\lambda_1)< (\lambda_1',\lambda_2')$. This is obvious if $r_\lambda > \frac{n}{2}$. If $r_\lambda=\frac{n}{2}$ then $(\lambda_1,\lambda_2) =\lambda > \lambda'=(\lambda_2',\lambda_1')$ implies that $\lambda_1> \lambda_2'$ or ($\lambda_1=\lambda_2'$ and $\lambda_2 > \lambda_1'$) which is a contradiction so $\lambda_1>  \lambda_2'$ and since $\lambda_1 \neq \lambda_2$, this implies $\lambda_1'>\lambda_2$ by definition of our order on partitions of $r_\lambda$. This shows that $(\lambda_2,\lambda_1)< (\lambda_1',\lambda_2')$.

Assume $\lambda> \lambda'$ and $\lambda> (\lambda_1',\lambda_2')$. We then have that either $\lambda'$ verifies $(*)$ or $(\lambda_2,\lambda_1)$ verifies $(*)$ in exactly the same way as in the previous case.

Assume $\lambda < \lambda'$ and $\lambda > (\lambda_1',\lambda_2')$, let us show that $(\lambda_1',\lambda_2') < (\lambda_2,\lambda_1)$ and $(\lambda_1',\lambda_2')< (\lambda_1,\lambda_2)$, i.e. $(\lambda_1',\lambda_2')$ verifies $(*)$. It is enough to show the second inequality since we have the first one by assumption. This is obvious if $r_\lambda < \frac{n}{2}$. If $r_\lambda = \frac{n}{2}$ then $\lambda_1< \lambda_2'$ because $\lambda_1\neq \lambda_2'$ so $\lambda_1' < \lambda_2$ because $\lambda_1\neq \lambda_2$ and $(\lambda_1',\lambda_2') < (\lambda_2,\lambda_1)$.

Assume $\lambda < \lambda'$ and $\lambda < (\lambda_1',\lambda_2')$. To conclude the proof, it is enough to show that not one of $\lambda', (\lambda_1',\lambda_2')$ and $(\lambda_2,\lambda_1)$ verifies $(*)$ in this case. It is obvious for $\lambda'$ and $(\lambda_1',\lambda_2)$. If $r_\lambda < \frac{n}{2}$, it is also obvious for $(\lambda_2,\lambda_1)$ since $(\lambda_2,\lambda_1) > (\lambda_1',\lambda_2')$. If $r_\lambda = \frac{n}{2}$ then since $\lambda_1 < \lambda_2'$ and $\lambda_1 \neq \lambda_2'$, we have that $\lambda_2 > \lambda_1'$ so $(\lambda_2,\lambda_1) > (\lambda_1',\lambda_2')$.
\end{proof}

\bigskip

We are now able to state the main results for type $B$ which are a generalization of Theorem 1.1 of \cite{BMM}, the end of the proof will be in the next section. The main difference arises from the additional factorizations in the last cases for the field extensions.

We write $A_{1,n}=\{(\lambda_1,\emptyset),\lambda_1 \vdash n\}, A_{2,n} =\{(\emptyset,\lambda_2),\lambda_2 \vdash n\}, A_n = A_{1,n} \cup A_{2,n}$. $A\epsilon_n= \{(\lambda_1,\emptyset)\in A_{1,n},~\lambda_1 ~\mbox{not a hook}\}$, $ \epsilon_n=\{\lambda \vdash\vdash n, \lambda \notin A_n, \lambda~\mbox{not a hook}\}, \F_{\tilde{q}}=\F_p(\alpha)$. 
\begin{theo}\label{result1}
If $\F_q=\F_p(\alpha,\beta)=\F_p(\alpha+\alpha^{-1}, \beta+\beta^{-1})$ and $\F_p(\alpha)=\F_p(\alpha+\alpha^{-1})$, then the morphism : $\mathcal{A}_{B_n} \rightarrow H_{B_n,\alpha,\beta}^\times \simeq \underset{\lambda \vdash\vdash n}\prod GL(\lambda)$ factors through the epimorphism 
$$\Phi_n : \mathcal{A}_{B_n} \rightarrow SL_{n-1}(\tilde{q})\times \underset{(\lambda_1,\emptyset)\in A\epsilon_n, \lambda_1<\lambda_1'}\prod SL_{n_\lambda}(\tilde{q}) \times \underset{(\lambda_1,\emptyset)\in A\epsilon_n,\lambda_1=\lambda_1'}\prod OSP(\lambda)'\times$$
 $$SL_n(q)^2 \times \underset{\lambda\in \epsilon_n, \lambda < \lambda'}\prod SL_{n_\lambda}(q) \times \underset{\lambda\in \epsilon_n, \lambda=\lambda'}\prod OSP(\lambda)'.$$
where $OSP(\lambda)$ is the group of isometries of the bilinear form defined in 2.
\end{theo}

\bigskip

If $\F_q=\F_p(\alpha,\beta)=\F_p(\alpha+\alpha^{-1}, \beta+\beta^{-1})$  and $\F_p(\alpha)\neq\F_p(\alpha+\alpha^{-1})$ then we have the corresponding theorem where the first row of groups corresponds to the unitary case in \cite{BMM}.

\begin{theo}\label{result2}
If  $\F_q=\F_p(\alpha,\beta)=\F_p(\alpha+\alpha^{-1},\beta)=\F_p(\alpha,\beta+\beta^{-1})\neq \F_p(\alpha+\alpha^{-1}, \beta+\beta^{-1})$, then the morphism $\mathcal{A}_{B_n} \rightarrow H_{B_n,\alpha,\beta}^\times \simeq \underset{\lambda \vdash\vdash n}\prod GL(\lambda)$ factors through the epimorphism
$$\Phi_n : \mathcal{A}_{B_n} \rightarrow SU_{n-1}({\tilde{q}^{\frac{1}{2}}})\times \underset{(\lambda_1,\emptyset)\in A\epsilon_n, \lambda_1<\lambda_1'}\prod SU(\lambda) \times \underset{(\lambda_1,\emptyset)\in A\epsilon_n,\lambda_1=\lambda_1'}\prod \widetilde{OSP}(\lambda)'\times $$
 $$SU_n(q^{\frac{1}{2}})^2 \times \underset{\lambda\in \epsilon_n, \lambda < \lambda'}\prod SU(\lambda) \times \underset{\lambda\in \epsilon_n, \lambda=\lambda'}\prod \widetilde{OSP}(\lambda)'.$$
where $\widetilde{OSP}(\lambda)$ is the group of isometries of a bilinear form of the same type as the one in 2 but defined over $\F_{q^\frac{1}{2}}$.
\end{theo}

\bigskip

\bigskip

Assume that $\F_q=\F_p(\alpha,\beta)=\F_p(\alpha,\beta+\beta^{-1}) \neq \F_p(\alpha+\alpha^{-1},\beta) = \F_p(\alpha+\alpha^{-1},\beta+\beta^{-1})$, we then have by Proposition \ref{isomorphisme} and Lemmas \ref{Harinordoquy} and \ref{Ngwenya} the following theorem :

\begin{theo}\label{result3}
If $\F_q=\F_p(\alpha,\beta)=\F_p(\alpha,\beta+\beta^{-1}) \neq \F_p(\alpha+\alpha^{-1},\beta) = \F_p(\alpha+\alpha^{-1},\beta+\beta^{-1})$, then the morphism $\mathcal{A}_{B_n} \rightarrow H_{n,\alpha,\beta}^\times \simeq \underset{\lambda \vdash\vdash n}\prod GL(\lambda)$ factors through the epimorphism
$$\Phi_n : \mathcal{A}_{B_n} \rightarrow SU_{n-1}({\tilde{q}^{\frac{1}{2}}})\times \underset{(\lambda_1,\emptyset)\in A\epsilon_n, \lambda_1<\lambda_1'}\prod SU(\lambda) \times \underset{(\lambda_1,\emptyset)\in A\epsilon_n,\lambda_1=\lambda_1'}\prod \widetilde{OSP}(\lambda)'\times$$
 $$ SL_n(q) \times \underset{\lambda\in \epsilon_n, \lambda < \lambda', \lambda < (\lambda_1',\lambda_2'), \lambda\neq (\lambda_2,\lambda_1)}\prod SL_{n_\lambda}(q) \times \underset{\lambda \in \epsilon_n, \lambda < \lambda', \lambda=(\lambda_2,\lambda_1)}\prod SU(\lambda) \times \underset{\lambda \in \epsilon_n, \lambda< \lambda',\lambda=(\lambda_1',\lambda_2')}\prod SL_{n_\lambda}(q^\frac{1}{2})\times $$
 $$  \underset{\lambda\in \epsilon_n, \lambda=\lambda', \lambda<(\lambda_1',\lambda_2')}\prod OSP(\lambda)' \times \underset{\lambda \in \epsilon_n, \lambda=\lambda',\lambda=(\lambda_1',\lambda_2')}\prod \widetilde{OSP}(\lambda)' .$$
 \end{theo}

The corresponding statements for the cases $\F_q=\F_p(\alpha,\beta)=\F_p(\alpha+\alpha^{-1},\beta) \neq \F_p(\alpha,\beta+\beta^{-1}) = \F_p(\alpha+\alpha^{-1},\beta+\beta^{-1})$ and ($\F_p(\alpha) = \F_p(\alpha+\alpha^{-1})$ or $\F_p(\alpha) \neq \F_p(\alpha+\alpha^{-1})$) should be clear and are left to the reader. Their proof is deduced from the same propositions and lemmas.

 \subsection{Surjectivity of the morphism $\Phi_n$}
 
 In this section, we conclude the proof of the theorems in the previous section by showing that the morphism $\Phi_n$ is surjective.
 
 \subsubsection{First case : $\F_q=\F_p(\alpha,\beta)= \F_p(\alpha+\alpha^{-1},\beta+\beta^{-1}),\F_p(\alpha)=\F_p(\alpha+\alpha^{-1})$}
 
 In this subsection, we prove the surjectivity of the morphism in the easiest case and establish groundwork for the other cases. We first prove the result for $n\leq 4$ and then use induction to get the result for all $n$.

We recall Goursat's Lemma also used in \cite{BM} and \cite{BMM} :

\begin{lemme}[Goursat's Lemma]\label{Goursat}
Let $G_1$ and $G_2$ be two groups, $K\leq G_1 \times G_2$, and write $\pi_i: K \longrightarrow G_i$ the projection. Let $K_i=\pi_i(K)$ and $K^i = ker(\pi_{i'})$ or $(i,i')=(1,2)$. There exists an isomorphism $\varphi : K_1/K^1 \rightarrow K_2/K^2$ such that $K=\{(k_1,k_2) \in K_1 \times K_2, \varphi(k_1K^1) =k_2K^2\}$.
\end{lemme}

We first prove that if for any $\lambda\vdash\!\vdash n$, the composition of $R_\lambda$ with the projection on its corresponding quasi-simple factor is surjective then $\Phi_n$ is surjective and then prove by induction that each composition is indeed surjective. In order to get the image of the hook partitions it is enough to get the images inside the representations associated with the partitions $([1^{n-1}],[1])$ and $([1],[1^{n-1}])$ . We recall Wagner's theorem which can be found for example in \cite[II, Theorem 2.3]{MM}.
\begin{theo}\label{wag}
Let $\F_r$ be a finite field, $n\in \N, n\geq 3$ and $G \subset GL_n(r)$ a primitive group generated by pseudo-reflections of order greater than or equal to $3$. Then one of the following is true.
\begin{enumerate}
\item $SL_n(\tilde{r}) \subset G \subset GL_n(\tilde{r})$ for some $\tilde{r}$ dividing $r$.
\item $SU_n(\tilde{r}^\frac{1}{2}) \subset G \subset GU_n(\tilde{r}^\frac{1}{2})$ for some $\tilde{r}$ dividing $r$.
\item $n\leq 4$, the pseudo-reflections are of order $3$ and $G \simeq GU_n(2)$.
\end{enumerate}
\end{theo}
 
 \begin{prop}\label{lesgourgues}
Let $n\geq 3$ and $R_1$ (resp $R_2$)  be the representation  associated with the double-partition $([1^{n-1}],[1])$ (resp $([1],[1^{n-1}])$). We have $R_1(\mathcal{A}_{B_n})=R_2(\mathcal{A}_{B_n})=SL_n(q)$.
\end{prop}

\begin{proof}
Let $n\geq 3$, we use theorem \ref{wag}. The eigenvalues of $R_1(T)$ are $\beta$ with multiplicity $n-1$ and $-1$ with multiplicity $1$ and the eigenvalues of $R_1(S_i)$ are $\alpha$ with multiplicity $1$ and $-1$ with multiplicity $n-1$. The group $G=\langle\beta^{-1}R_1(T),-R_1(S_1),...,-R_1(S_{n-1})\rangle$ is generated by pseudo-reflections. To apply Wagner's Theorem (Theorem \ref{wag}), we must show our group is primitive. If $G$ was imprimitive, we could write $\F_q^n=V_1\oplus V_2 \oplus ... \oplus V_r$, where for all $i$ and for all $g\in G$, there exists a $j$ such that $g.V_i =V_j$. Since $R_1$ is irreducible, either $\beta^{-1}R_1(T) .V_1 \neq V_1$ or there exists $i\leq n-1$ such that $-R_1(S_i) .V_1 \neq V_1$. Assume there exists $i$ such that $-R_1(S_i).V_1 \neq V_1$. Up to reordering, we have $V_2=-R(S_i).V_1$. If $\dim(V_1)\geq 2$ then $H_{-R_1(S_i)}$ (the hyperplane fixed by $-R_1(S_i)$) has a non-empty intersection with $V_1$ so $V_1\cap V_2 \neq \emptyset$ which is a contradiction so $\dim(V_1)=1$. This reasoning is valid for any $V_i$ so they are all one-dimensional. Let $x\in V_1$ be a non-zero vector, it can be written in a unique way as $x =x_1+x_2$ with $x_1\in Ker(R_1(S_i)+\alpha)$ and $x_2\in H_{-R_1(S_i)}$. We then have that $-R_1(S_i)x = -\alpha x_1+x_2$ and $-R(S_i)(-R(S_i)x)= \alpha^2 x_1 +x_2=\alpha(x_1+x_2)+(1-\alpha)(-\alpha x_1+x_2)\in V_1\oplus V_2$. Since $\alpha\notin \{0,1\}$ this contradicts the fact that there exists $j$ such that $-R(S_i).V_2 =V_j$.
If $V_1\neq \beta^{-1}R_1(T) V_1 =V_2$, then if $x=x_1+x_2\neq 0$ with $x\in V_1,x_1\in Ker(R_1(T)+\beta) x_2\in H_{\beta^{-1}R_1(T)}$, we have that $\beta^{-1}R_1(T)x = -\beta^{-1}x_1+x_2$ and $\beta^{-1}R_1(T)(\beta^{-1}R_1(T)x) = \beta^{-2}x_1+x_2=\beta^{-1}(x_1+x_2)+(1-\beta^{-1})(-\beta^{-1}x_1+x_2)\in V_1\oplus V_2$. This is absurd because $\beta^{-1}\notin \{0,1\}$. This shows that $G$ is primitive and in the same way, $\tilde{G}=\langle -R_2(T),-R_2(S_1),...,-R_2(S_{n-1})\rangle$ is primitive and generated by pseudo-reflections of order greater than or equal to 3. By Theorem \ref{wag}, we have $SL_n(\tilde{q}) \subset G,\tilde{G} \subset GL_n(\tilde{q})$ or $SU_n(\tilde{q}^\frac{1}{2}) \subset G,\tilde{G} \subset GU_n(\tilde{q}^\frac{1}{2})$ for some $\tilde{q}$ dividing $q$. If we were in the unitary case then there would exist an automorphism $\epsilon$ of order 2 of $\F_{\tilde{q}^2}$ such that $\det(M)=\epsilon(\det(M))^{-1}$ for all $M$ in $G$ or $\tilde{G}$. We also have $\det(\beta^{-1}R_1(T)) = -\beta^{-1}, \det(-R_1(S_1))= \det(-R_2(S_1)) = -\alpha$ and $\det(-R_2(T))=  -\beta$ so if $G$ or $\tilde{G}$ is unitary then $\epsilon(\beta)=\beta^{-1}$ and $\epsilon(\alpha)=\alpha^{-1}$ so $\alpha+\alpha^{-1}$ and $\beta+\beta^{-1}$ would be in $\F_{\tilde{q}}$ which is a contradiction because $\tilde{q}^2$ divides $q$ and $\F_q=\F_p(\alpha+\alpha^{-1},\beta+\beta^{-1})$. This proves we have $SL_n(\tilde{q}) \subset G,\tilde{G} \subset GL_n(\tilde{q})$ for some $\tilde{q}$ dividing $q$ so using again the determinants, we have $\alpha$ and $\beta$ in $\F_{\tilde{q}}$ and so $\tilde{q}=q$. We have $SL_n(q)=[G,G]=[R_1(A_{B_n}),R_1(A_{B_n})]=R_1(\mathcal{A}_{B_n})$ and $SL_n(q)=[\tilde{G},\tilde{G}]=[R_2(A_{B_n}),R_2(A_{B_n})]=R_2(\mathcal{A}_{B_n})$ which concludes the proof.
\end{proof}

By \cite{MR}, $\mathcal{A}_{B_n}$ is perfect for $n\geq 5$ but not for $n\leq 4$ so those cases must be treated separately.

\begin{lemme}\label{platypus}
If $n\leq 4$ then $\Phi_n$ is surjective.
\end{lemme}

\begin{proof}
Double-partitions of $n=2$ are all one-dimensional except for $([1],[1])$ so we only need to show that $R_{[1],[1]}(\mathcal{A}_{B_2})=SL_2(q)$. We write $t=R_{[1],[1]}(T)=\begin{pmatrix} \beta & 0 \\ 0 & -1 \end{pmatrix}$ and $s=R_{[1],[1]}(S_1)= \frac{1}{\beta+1} \begin{pmatrix} \alpha -1 & \alpha +\beta\\ \alpha \beta+1 & \alpha \beta -\beta \end{pmatrix}$. First note that if $P = \begin{pmatrix} 1 & 1\\ \frac{\alpha\beta+1}{\alpha+\beta} & -1. \end{pmatrix}$ then $P^{-1}tP = \frac{1}{\alpha+1}\begin{pmatrix} \beta -1 & \alpha+\beta\\
\alpha\beta+1 & \beta\alpha -\alpha \end{pmatrix}$ and $P^{-1}sP = \begin{pmatrix}
\alpha & 0\\
0 & -1.
\end{pmatrix}$. This proves that the roles of $\alpha$ and $\beta$ are completely symmetrical in this case so up to conjugating by $P$, we can exchange the conditions on $\alpha$ and the conditions on $\beta$. We write $G=<t,s>$. We have $\det(t) = -\beta$ and $\det(s) = -\alpha$. Let $(u,v)\in \overline{\F_p}^2$ such that $u^2 = -\beta^{-1}$ and $v^2 = -\alpha^{-1}$, we set $\F_{q'} = \F_q(u,v)$. We then have $\overline{G}=\langle \overline{t},\overline{s}\rangle = \langle\overline{ut},\overline{vs}\rangle \subset PSL_2(q')$. We write $\mathfrak{S}_n$ the permutation group of $n$ elements and $\mathfrak{A}_n$ its derived subgroup. By Dickson's Theorem \cite[Chapter II, HauptSatz 8.27]{HUP}, we have that $\overline{G}$ is either abelain by abelian or isomorphic to $\mathfrak{S}_3, \mathfrak{A}_4, \mathfrak{A}_5, \mathfrak{S}_4, PSL_2(\tilde{q})$ or $PGL_2(\tilde{q})$ for a given $\tilde{q}$ greater than or equal to $4$ and dividing $q'$. 

If $\overline{ut}^r = 1$ in $PSL_2(q')$ then $((-u)^r)^2 = 1$ so $\beta^r = (-1)^r$ and by the condition on the order of $\alpha$, $\overline{G}$ cannot be isomorphic to $\mathfrak{S}_3, \mathfrak{S}_4, \mathfrak{A}_4$ or $\mathfrak{A}_5$.

We now exclude the case $\overline{G}$ abelian by abelian. If $\overline{G}$ is abelian by abelian, then $[\overline{G},\overline{G}]$ is abelian, i.e. $\overline{ab}= \overline{ba}$ for all $a,b \in [G,G]$ or equivalently $ab=\pm ba$ for all $a,b \in [G,G]$. We have that $(tst^{-1}s^{-1}) (s^{-1}tst^{-1})-(s^{-1}tst^{-1})(tst^{-1}s^{-1}) = $
$$ \begin{pmatrix}
-\frac{(\beta-1)(\alpha-1)^2(\alpha\beta+1)(\alpha+\beta)}{\beta\alpha^2(\beta+1)} & -\frac{(\alpha^2\beta^2+\alpha\beta^3-\alpha\beta^2-\alpha^2\beta+\alpha\beta+\beta^2-\alpha-\beta)(\alpha-1)(\alpha\beta+1)}{\beta\alpha^2(\beta+1)}\\
\frac{(\alpha^2\beta^2+\alpha\beta^3-\alpha\beta^2-\alpha^2\beta+\alpha\beta+\beta^2-\alpha-\beta)(\alpha-1)(\alpha+\beta)}{\beta^2\alpha^2(\beta+1)}
& \frac{(\beta-1)(\alpha-1)^2(\alpha\beta+1)(\alpha+\beta)}{\beta\alpha^2(\beta+1)}
\end{pmatrix}.$$
This matrix is non-zero because the diagonal coefficients are non-zero by the conditions on $\beta$. This means that if $[\overline{G},\overline{G}]$ is abelian than we have $(tst^{-1}s^{-1}) (s^{-1}tst^{-1})+(s^{-1}tst^{-1})(tst^{-1}s^{-1}) = 0$ but this matrix equals
$$\begin{pmatrix}
 \frac{{\alpha}^{4}\beta+{\alpha}^{3}{\beta}^{2}-2{\alpha}^{3}\beta-2{\alpha}^{2}{\beta}^{2}+{\alpha}^{3}+4{\alpha}^{2}\beta+\alpha{\beta}^{2}-2{\alpha}^{2}-2\alpha\beta+\alpha+\beta}{{\alpha}^{2}\beta} & - {\frac { \left( {\alpha}^{2}\beta+\alpha\,{\beta}^{2}-2\,\alpha\,\beta+\alpha+\beta \right)  \left( \alpha-1 \right)  \left( \alpha\,\beta+1 \right) }{{\alpha}^{2}\beta}}\\ 
-{\frac { ( {\alpha}^{2}\beta+\alpha\,{\beta}^{2}-2\,\alpha\,\beta+\alpha+\beta ) ( \alpha+\beta )  ( \alpha-1) }{{\alpha}^{2}{\beta}^{2}}} & \frac {{\alpha}^{4}\beta+{\alpha}^{3}{\beta}^{2}-2\,{\alpha}^{3}\beta-2\,{\alpha}^{2}{\beta}^{2}+{\alpha}^{3}+4\,{\alpha}^{2}\beta+\alpha\,{\beta}^{2}-2\,{\alpha}^{2}-2\,\alpha\,\beta+\alpha+\beta}{{\alpha}^{2}\beta }
\end{pmatrix}.$$
The non-diagonal coefficients are non-zero if $A= {\alpha}^{2}\beta+\alpha\,{\beta}^{2}-2\,\alpha\,\beta+\alpha+\beta $ is non-zero. If $A=0$ then, the bottom right coefficient of $(tst^{-1}s^{-1})(st^{-1}s^{-1}t)+(st^{-1}s^{-1}t)(tst^{-1}s^{-1})$ is equal to $-\frac{1}{(\beta+1)\alpha^2\beta^2}$ multiplied by
$${\alpha}^{4}{\beta}^{3}+{\alpha}^{3}{\beta}^{4}-{\alpha}^{4}{\beta}^{2}-3\,{\alpha}^{3}{\beta}^{3}-2\,{\alpha}^{2}{\beta}^{4}+5\,{\alpha}^{3}{\beta}^{2}+4\,{\alpha}^{2}{\beta}^{3}+$$
$$\alpha\,{\beta}^{4}-3\,{\alpha}^{3}\beta-8\,{\alpha}^{2}{\beta}^{2}-3\,\alpha\,{\beta}^{3}+4\,{\alpha}^{2}\beta+5\,\alpha\,{\beta}^{2}+{\beta}^{3}-2\,{\alpha}^{2}-3\,\alpha\,\beta-{\beta}^{2}=$$
$$ (\alpha^2\beta^2-\alpha^2\beta+2\alpha\beta-2\alpha)A+\beta((\beta-1)A-2\alpha^2(\beta^3+1)) = -2\beta\alpha^2(\beta^3+1).$$ This is non-zero by the condition on $\beta$. The diagonal coefficients of the difference of these two commutators are identical to the ones of the difference of the previous commutators so they're non-zero. This proves $\overline{G}$ isn't abelian by abelian and there exists $\tilde{q}$ greater than or equal to $4$ such that $[\overline{G},\overline{G}] \simeq PSL_2(\F_{\tilde{q}})$.

For $H$ a group and $A$ an $H$-module, we write $Z^2(H,A) = \{f :H\times H \rightarrow A, \forall x,y,z \in H, z.f(x,y)f(xy,z)=f(y,z)f(xy,z)\}$ the group of cocyles and  $B^2(H,A) = \{f :H\times H \rightarrow A,\exists t : H \rightarrow A, \forall x,y \in H, f(x,y)= t(y)t(xy)^{-1}y.t(x)\}$ the group of coboundaries. We write $M(H) = H^2(H,\C^\star) = H_2(H,\Z)$ its Schur multiplier. We have $[\overline{G},\overline{G}] = \overline{[G,G]} \subset PSL_2(q)$, this inclusion gives a projective representation of $SL_2(\tilde{q})$ with associated cocyle $c\in Z^2(SL_2(\F_{\tilde{q}}),\F_q^\star)$. We show that $H^2(SL_2(\F_{\tilde{q}}),\F_q^\star)$ is trivial which is equivalent to $Z^2(SL_2(\F_{\tilde{q}}),\F_q^\star) =B^2(SL_2(\F_{\tilde{q}}),\F_q^\star)$ so this cocycle is a coboundary. We write $H =SL_2(\tilde{q})$. By the Universal Coefficients Theorem  \cite[Theorem 3.2]{HAT}, we have the following exact sequence 
$$ 1 \rightarrow Ext(H_1(H,\Z),\F_q^\star)  \rightarrow H^2(H,\F_q^\star) \rightarrow Hom(H_2(H,\Z),\F_q^\star)\rightarrow 1.$$
We have $H_1(H,\Z) = H/[H,H]$ \cite{KAR} and $H_2(H,\Z) = M(H)$, since $\tilde{q} \geq 4$, $SL_2(\F_{\tilde{q}})$ is perfect and the exact sequence becomes
$$ 1 \rightarrow 1 \rightarrow H^2(H,\F_q^\star) \rightarrow Hom(M(H),\F_q^\star)\rightarrow 1.$$
By \cite[Theorem 7.1.1]{KAR}, if $\tilde{q}\notin \{4,9\}$ then the Schur multiplier $M(H)$ is trivial so this reduces to $ H^2(H,\F_q^\star)\simeq \{1\}$.

It remains to take care of the cases $\tilde{q}= 4$ and $\tilde{q}=9$. If $\tilde{q} =4$, we have $M(H) = \Z/2\Z$ and $char(\F_q) = 2$ so  $Hom(M(H),\F_q^\star) =1$. Indeed, every morphism $\varphi$ from $M(H)$ to $\F_q^\star$ verifies $1 = \varphi(2x)=\varphi(x)^2$ for all $x\in M(H)$ so $0=\varphi(x)^2 -1= (\varphi(x)-1)^2$  for all $x\in M(H)$ so $\varphi$ is trivial. If $\tilde{q} = 9$ then we have $M(H) = \Z/3\Z$ and $H^2(H,\F_q^\star)$ is trivial by the same reasoning as for $\tilde{q}=4$. In all cases, we can define a representation  $\rho$ of $SL_2(\tilde{q})$ in $SL_2(q)$.

By \cite{BN}, any representation $\sigma$ of $SL_2(q)$ in $GL_2(q)$ is up to conjugation of the form $\sigma(M)=\psi(M)$ where $\psi(M)$ is the matrix obtained from $M$ by applying $\psi \in Aut(\F_q))$ to all its coefficients. We have $\F_q = \F_{\tilde{q}}(w)$ for any $w$ generating the cyclic group $\F_q^{\star}$. There exists a homomorphism from $\F_q$ to $\F_{\tilde{q}}$ sending $1$ to $w$ and stabilizing $\F_{\tilde{q}}$. We define a representation $\tilde{\rho}$ of $SL_2(q)$ in $SL_2(q)$ such that $\tilde{\rho}(M) = \rho(M)$ for all $M$ in $SL_2(\tilde{q})$. We have $\rho(M) = \tilde{\rho}(M) = \psi(M)$ for all $M$ in $SL_2(\tilde{q})$. We have $[\overline{G},\overline{G}] \simeq PSL_2(\tilde{q})$ so $\psi([\overline{G},\overline{G}]) \simeq PSL_2(\tilde{q})$ is conjugate to $PSL_2(\tilde{q})$ in $GL_2(q')$. We have $\psi(tst^{-1}s^{-1}) \in \psi([\overline{G},\overline{G}])$ so its trace $2 - (\alpha + \alpha^{-1} + \beta + \beta^{-1})$ belongs to $\F_{\tilde{q}}$. This shows that $\alpha + \alpha^{-1} + \beta + \beta^{-1} \in \F_{\tilde{q}}$. We also have that the trace $T_1$ of $s^2t^{-1}s^{-2}t$ and the trace $T_2$ of $st^{-2}s^{-1}t^2$ are in $\F_{\tilde{q}}$. We have $T_1 = \frac{\alpha^4\beta+\alpha^3\beta^2-2\alpha^3\beta-2\alpha^2\beta^2+\alpha^3+4\alpha^2\beta+\alpha\beta^2-2\alpha^2-2\alpha\beta+\alpha+\beta}{\alpha^2\beta}$. We write $B = \alpha+\alpha^{-1}+\beta+\beta^{-1}$. We then have $T_1 = (\alpha+\alpha^{-1})B-2B+2$. $T_2$ has the expression of $T_1$ with $\alpha$ and $\beta$ switched so $T_2 = (\beta+\beta^{-1})B-2B+2$. Since $B, T_1$ and $T_2$ are in $\F_{\tilde{q}}$, we have $(\alpha+\alpha^{-1})B$ and $(\beta+\beta^{-1})B$ are in $\F_{\tilde{q}}$. We have $B = \alpha+\alpha^{-1}+\beta+ \beta^{-1}$ so $B =0$ would imply $\alpha\in \{-\beta,-\beta^{-1}\}$ which contradicts the assumptions on $\alpha$ and $\beta$. $B$ is non-zero so $\alpha+\alpha^{-1}$ and $\beta+\beta^{-1}$ are in $\F_{\tilde{q}}$ and so $\F_{\tilde{q}}=\F_q$. We conclude using  Lemma 2.1 of \cite{BM}.

\bigskip

\bigskip

The double-partitions of $n=3$ to consider are $([2,1],\emptyset),([1],[1^2])$ and $([1^2],[1])$. We want to show the image of $\Phi_3$ is equal to $SL_2(\tilde{q}) \times SL_3(q) \times SL_3(q)$. If we restrict ourselves to the image inside $SL_3(\F_q) \times SL_3(\F_q)$, we show that it is $SL_3(q)\times SL_3(q)$. By Proposition \ref{lesgourgues}, $SL_3(q)=R_{[1^2],[1]}(\mathcal{A}_{B_3})$ and $R_{[1],[1^2]}(\mathcal{A}_{B_3})=SL_3(q)$. We now use Goursat's Lemma : we write as in the lemma $K=R(\mathcal{A}_{B_3}), K_1=R_{[1^2],[1]}(\mathcal{A}_{B_3}), K_2= R_{[1],[1^2]}(\mathcal{A}_{B_3}) $ $\pi_1$ (resp $\pi_2$) the projection unto $SL_{3}(q)$(corresponding to $([1^2],[1])$) (resp $SL_{3}(q))$ (corresponding to $([1],[1^2])$)), $K^1=ker(\pi_2)$, $K^2=ker(\pi_1)$ and $\varphi$ the isomorphism given by Goursat's Lemma. We have $K=\{(x,y)=\in K_1\times K_2, \varphi(xK^1)=yK^2\}$. By the same reasoning as the one in Proposition 3.1. of \cite{BM}, either  $SL_3(\F_q)\times SL_3(\F_q) \subset K$ or $K_1/K^1$ is non-abelian and $\varphi$ is an isomorphism of $PSL_3(\F_q)$ and using the same notations, up to conjugation $R_2(b) = S^{\phi}(R_1(b))z(b)$ for all $b\in \mathcal{A}_{B_3}$ ($S=Id$ or $S=M\mapsto {}^t\!M^{-1}$).

Let us show that the second possibility is absurd by choosing the right elements in $\mathcal{A}_{B_3}$. For any element $b$ of $\mathcal{A}_{B_3}, \tr(R_{[1],[1^2]}(b)) = z(b)\tr(S^{\phi}(R_{[1^2],[1]}(b))$. We write $U=S_1S_2^{-1}, V=TS_1T^{-1}S_2^{-1}, W=S_2S_1S_2^{-2}$ and $X=S_2TS_1T^{-1}S_2^{-2}$, they are all elements of $\mathcal{A}_{B_3}$. By explicit computation, for both choices of $S$, we have :
$$\tr(R_{[1],[1^2]}(U))=\tr(R_{[1],[1^2]}(V))=\tr(R_{[1],[1^2]}(W))=\tr(R_{[1],[1^2]}(X))=-\frac{(\alpha-1)^2}{\alpha},$$
$$\tr(S(R_{[1^2],[1]}(U)))=\tr(S(R_{[1^2],[1]}(V)))=\tr(S(R_{[1^2],[1]}(W)))=\tr(S(R_{[1^2],[1]}(X)))=-\frac{(\alpha-1)^2}{\alpha}.$$
This shows that $z(U)=z(V)=z(W)=z(X)$ and
$$1=z(U)z(W)^{-1}=z(UW^{-1})=\frac{-\frac{(\alpha-1)^2}{\alpha}}{\phi(-\frac{(\alpha-1)^2}{\alpha})}.$$
This proves $\phi(\alpha+\alpha^{-1})=\alpha+\alpha^{-1}$. We also have
$$1=z(UV^{-1})= \frac{3-\alpha-\alpha^{-1}-\beta-\beta^{-1}}{\phi(3-\alpha-\alpha^{-1}-\beta-\beta^{-1})}.$$
Using $\phi(\alpha+\alpha^{-1})=\alpha+\alpha^{-1}$, we have $\phi(\beta+\beta^{-1})=\beta+\beta^{-1}$ so $\phi=1$. We deduce that
$$1=z(UX^{-1})=\frac{(\alpha-1)(\alpha\beta^2-2\alpha\beta+2\beta-1)}{\alpha\beta}\frac{-\alpha\beta}{(\alpha-1)(2\alpha\beta+\beta^2-\alpha-2\beta)}= \frac{1+2\alpha\beta-2\beta-\alpha\beta^2}{\beta^2+2\alpha\beta-2\beta-\alpha},$$
$$1-\alpha\beta^2=\beta^2-\alpha, (1-\beta^2)(1+\alpha)=0.$$
Since $\beta^2\neq 1$ and $\alpha^2\neq 1$, we get a contradiction. This shows that $SL_3(q)\times SL_3(q) = R(\mathcal{A}_{B_3})$.

We now set $G_1= SL_3(q) \times SL_3(q)$ and $G_2 = SL_2(\tilde{q})$, the image of $\Phi_3$ is a subgroup of $G_1 \times G_2$ for which the projections onto $G_1$ and $G_2$ are surjective. Using again Goursat's Lemma and the notation there, we have $K_1/K^1 \simeq K_2/K^2$. We have a surjective morphism $\psi$ from $K_1=G_1=SL_3(q) \times SL_3(q)$ to $K_2/K^2$ where $K_2=SL_2(\tilde{q})$. If $K_2/K^2$ was non-abelian then we would have $K_2/K^2 \simeq PSL_2(\tilde{q})$. If the restriction $\psi_1$ (resp $\psi_2$) of $\psi$ to $SL([1^2],[1])$ (resp $SL([1],[1^2])$ was not trivial then $\psi_1$ (resp $\psi_2$) would factor into an isomorphism from  $PSL_3(q)$ unto $PSL_2(\F_q)$ since the center of $SL_3(\F_q)$ would again be in the center of $\psi_1$ and $\psi_2$.This would lead to a contradiction so their image is trivial and $\psi$ is not surjective so the quotients are abelian. This shows that $K_1=[K_1,K_1]\subset K^1$ and $K_2=[K_2,K_2]\subset K^2$ then using Goursat's Lemma we conclude that the image of $\Phi_3$ is equal to $G_1 \times G_2$. This shows that $\Phi_3$ is  surjective.

\bigskip

\bigskip

The double-partitions of $4$ in our decomposition are $([1^4],\emptyset)$, $([2^2],\emptyset)$, $([2,1,1],\emptyset)$, $([1],[1^3])$, $([1^3],[1])$, $([1^2],[1^2])$ and $([2,1],[1])$ of respective dimensions $1$, $2$, $3$, $4$, $4$, $6$ and $8$ (We removed the hooks, $([3],[1])$ and $([1],[3])$). We know the restriction to the first five is surjective by \cite{BMM} and Proposition \ref{lesgourgues} so we only need to show that $R_{[1^2],[1^2]}(\mathcal{A}_{B_4})=SL_6(q)$ and $R_{([2,1],[1])}(\mathcal{A}_{B_4})=SL_8(q)$.

Let us first consider the double-partition $([1^2],[1^2])$. By the branching rule and the case $n=3$ above, we have 
$$SL_3(q)\times SL_3(q)=R_{[1^2],[1]}(\mathcal{A}_{B_3})\times R_{[1],[1^2]}(\mathcal{A}_{B_3})=R_{[1^2],[1^2]}(\mathcal{A}_{B_3}) \subset R_{[1^2],[1^2]}(\mathcal{A}_{B_4})\subset SL_6(q).$$

We can now use Theorem 3 from \cite{BM}.
\begin{theo}\label{LBJ}
Let $\F_r$ be a finite field and $\Gamma < GL_N(r)$ with $N\geq 5$ and $q>3$ such that
\begin{enumerate}
\item $\Gamma$ is absolutely irreducible,
\item $\Gamma$ contains $SL_a(r)$ in a natural representation with $a\geq \frac{N}{2}$.
\end{enumerate}
 If $N\neq 2a$ then $\Gamma$ contains $SL_N(r)$. Otherwise, either $\Gamma$ contains $SL_N(r)$, or $\Gamma$ is a subgroup of $GL_{\frac{N}{2}}(r) \wr \mathfrak{S}_2$.
\end{theo} 

We use this theorem on $R_{[1^2],[1^2]}(A_{B_4})$. To get the desired result, we only need to show that $R_{[1^2],[1^2]}(A_{B_4})$ cannot be a subgroup of  $GL_{3}(r) \wr \mathfrak{S}_2$. If it were true, then we would have $R_{[1^2],[1^2]}(\mathcal{A}_{B_4})\subset SL_3(q)\times SL_3(q)$ which would contradict the irreducibility shown in Lemma \ref{Lincoln}. This shows that we have $R_{[1^2],[1^2]}(\mathcal{A}_{B_4})=SL_6(q)$.

\bigskip

We now consider the double-partition $([2,1],[1])$. Again by the branching rule and the case $n=3$, we have that the restriction to $\mathcal{A}_{B_3}$ is $ SL_3(\F_q) \times SL_3(\F_q)  \times SL_2(\F_{\tilde{q}})$. We now use the fact each of these groups is generated by transvections and the fact that $\mathcal{A}_{B_4}$ is normally generated by $\mathcal{A}_{B_3}$. When the characteristic is different from $2$, we can use the following theorem (first theorem of \cite{SZ}).
\begin{theo}\label{transvections}
If $G$ is an irreducible subgroup of $GL_n(q)$ generated by transvections with $q=p^r, p> 3, n> 2$ then $G$ is conjugate inside $GL_n(q)$ to $SL_n(\tilde{q}), Sp_n(\tilde{q})$ or $SU_n(\tilde{q}^{\frac{1}{2}})$ for some $\tilde{q}$ dividing $q$.
\end{theo}

We write $G= R_{([2,1],[1])}(\mathcal{A}_{B_4})$, $H = R_{([2,1],[1])}(\mathcal{A}_{B_3}) = SL_3(q) \times SL_3(q) \times SL_2(\tilde{q})$ and we pick $t_1$ (resp $t_2$ (resp $t_3$)) a transvection of $SL_3(q) \times \{I_5\}$ (resp $\{I_3\} \times SL_3(q) \times \{I_2\}$ (resp $\{I_6\} \times SL_2(\tilde{q})$). We then have $H = \langle ht_ih^{-1}, h\in H, i \in \{1,2,3\}\rangle$ and so $G = \langle ghg^{-1}, h\in H, g\in G\rangle = \langle gt_ig^{-1}, g\in G, i\in \{1,2,3\}\rangle$ is generated by transvections and we can apply the theorem.  We also recall the following lemma \cite[Lemma 5.6]{BMM}.
\begin{lemme}\label{field}
For any prime $p$ and $m\geq 2$, the field generated over $\F_p$ by $\{\tr(g), g \in SL_m(q)\}$ is $\F_q$ and for all $m\geq 3$, the field generated over $\F_p$ by $\{\tr(g), g\in SU_m(q^{\frac{1}{2}})\}$ is $\F_q$.
\end{lemme}

By Proposition \ref{isomorphisme}, we know that  $R_{[2,1],[1]}(\mathcal{A}_{B_4})$ preserves no non-degenerate bilinear form. It also shows that it can preserve no non-degenerate hermitian form. Indeed, if it were to preserve a hermitian form then we would have $\tr(M)=\epsilon(\tr({}^t\!(M)^{-1}))$ for any $M$ in $G$ and we have $\diag([\alpha,\alpha^{-1},1,1,1,1,1,1])$ and $\diag([\beta,\beta^{-1},1,1,1,1,1,1])$ in $H \subset G$, so we would have $\epsilon(\alpha+\alpha^{-1})= \alpha+\alpha^{-1}$ and $\epsilon(\beta+\beta^{-1})=\beta+\beta^{-1}$. Since $\F_q=\F_p(\alpha,\beta) =\F_p(\alpha+\alpha^{-1},\beta+\beta^{-1})$, the automorphism $\epsilon$ of order $2$ would be trivial which is a contradiction. This proves $G$ is conjugated in $GL_8(q)$ to $SL_8(\tilde{q})$ for some $\tilde{q}$ dividing $q$. By Lemma \ref{field}, the field generated over $\F_p$ by the traces of the elements of $G$ is $\F_{\tilde{q}}$ so $\tilde{q}= q$ because $G$ contains $SL_3(q)$ in a natural representation so the field generated by its elements contains $\F_q$. This shows that when $p\neq 2$, $G = R_{([2,1],[1])}(\mathcal{A}_{B_4}) = SL_8(q)$.

Assume now $p=2$, we can use the following theorem \cite[Theorem 1]{P}.
\begin{theo}
Let $V$ be a $\F_q$-vector space of dimension $n\geq 4$ with $q$ even. If $G$ is an irreductible proper subgroup of $SL(V)=SL_n(q)$ generated by a set $D$ of transvections of $G$, then $D$ is a conjugacy class of odd transpositions of $G$.
\end{theo}

Assume $G = R_{([2,1],[1])}(\mathcal{A}_{B_4})$ is different from $SL_8(q)$. We again have that $G$ is generated by transvections and by applying the above theorem, those transvections are in a single conjugacy class of $G$. Since $O_p(G)$ is normal in $G$ and $V=\F_q^8$ is an irreducible $\F_qG$-module, we apply Clifford's Theorem \cite[Theorem 11.1]{C-R} and get that $Res^G_{O_p(G)}(V)$ is semisimple. Since $O_p(G)$ is a $p$-group, the unique irreducible $\F_qO_p(G)$-module is the trivial module so $O_p(G)$ acts trivially on $V$. It follows that $O_p(G)$ is trivial. We can thus apply Kantor's Theorem \cite[Theorem II]{K} :
\begin{theo}
Let $p$ be a prime and $q=p^l$ for some $l\in \N$. Assume $G$ is an irreducible subgroup of $SL_N(q)$ generated by a conjugacy class of transvections, such that $O_p(G) \leq [G,G] \cap Z(G)$. Then $G$ is one of the following subgroups.
\begin{enumerate}
\item $G=SL_n(q')$ or $G=Sp_N(q')$ in $SL_N(q')$ or $G = SU_N(q'^{\frac{1}{2}})$ in $SL_N(q')$, $q' | q$.
\item $G= O_N^{\pm}(q') < SL_n(q'), q'|q$.
\item $G = \mathfrak{S}_n < SL_N(2)$ where $N = n-d$ and $d = gcd(n,2)$.
\item $G = \mathfrak{S}_{2n}$ in $SL_{2n-1}(2)$ or in $SL_{2n}(2)$.
\item $G = SL_2(5) < SL_2(9)$.
\item $G = 3.P\Omega_6^{-,\pi} < SL_6(4)$.
\item $G = SU_4(2) < SL_5(4)$.
\item $G = A \rtimes S_N$ in $SL_N(2^i)$ where $A$ is a subgroup of diagonal matrices.
\end{enumerate}
\end{theo}

Since $\alpha$ is of order greater than $4$, we have $q \geq \tilde{q}=2^r> 8$. The group $G$ contains $H = SL_3(q) \times SL_3(q) \times SL_2(\tilde{q})$, so cases $3$ to $7$ are excluded. If we were in Case $8$ then $G$ would have at mose $(q-1)\frac{10\times 9}{2}=45(q-1)$ transvections (see proof of Theorem 1.3. page 661 of \cite{BM}). $SL_3(q)$ has $\frac{(q^3-1)(q^2-1)}{q-1} = (q-1)(q^2+q+1)(q+1)$ transvections and  $(q^2+q+1)(q+1) \geq  847(q-1)> 45(q-1)$. For the same reasons as when $p\neq 2$, $G$ is neither unitary nor symplectic nor orthogonal. The only remaining possibility is $G = SL_8(q)$ which is a contradiction since we assumed $G \neq SL_8(q)$. This proves that $G=SL_8(q)$.

\bigskip

The restriction to each double-partition of $4$ is thus surjective, it remains to show $\Phi_4$ is surjective using Goursat's Lemma (Lemma \ref{Goursat}). This means we have to show the image is $SL_2(\tilde{q}) \times SL_3(\tilde{q}) \times SL_4(q) \times SL_4(q) \times SL_8(q)$.

By Theorem 1.2. of \cite{BM}, the restriction to $SL_2(\tilde{q}) \times SL_3(\tilde{q})$ is surjective. We write $G_1$ the image of this restriction and $G_2 = R_{([1],[1^3])}(\mathcal{A}_{B_4}) = SL_4(q)$. Let $K$ be the image $\mathcal{A}_{B_4}$ in $G_1 \times G_2$, using the corresponding notations in Goursat's Lemma, we have $K_1 = G_1$, $K_2 = G_2$ and $K_1/K^1 \simeq K_2/K^2$. If these quotients are abelian then the proof of $K = G_1 \times G_2$ is straightforward using Goursat's Lemma. Since the only non-abelian decomposition factor of $G_2$ is $PSL_4(q)$ and the only non-abelian decomposition factors of $G_1$ are $PSL_2(\tilde{q})$ and $PSL_3(\tilde{q})$, we have a contradiction if these quotients are non-abelian. Write now $\tilde{K} = R_{([1],[3])}(\mathcal{A}_{B_4})=SL_4(q)$ and let us consider the image $J$ of $\mathcal{A}_{B_4}$ inside $K\times \tilde{K}$. 
Using again Goursat's Lemma, this time with $K_1 = K, K_2=\tilde{K}$, we have $K_1/K^1\simeq K_2/K^2$. If the quotients are abelian then $J=K\times \tilde{K}$. If the quotients are non-abelian then there is an isomorphism $S^\phi$ from $PSL_4(q)$ to $PSL_4(q)$, where the first one corresponds to $\overline{R_{([1],[1^3])}(\mathcal{A}_{B_4})}$ and the second one to $\overline{R_{([1],[3])}(\mathcal{A}_{B_4})}$. This implies that there exists a character $z$ from $\mathcal{A}_{B_4}$ to $\F_q^{\star}$ such that up to conjugation, for every $h\in \mathcal{H}_4$, we have $R_{[1],[1^3]}(h) = S^{\phi}(R_{([1],[3])}(h))z(h)$. The isomorphism $S^\phi$ is of the form \cite[Section 3.3.4]{W} $M\mapsto \phi(M)$ or $M\mapsto \phi({}^t\!(M^{-1}))$ where $\phi$ is a field automorphism of $\F_q$. We would then have  that for all $h\in \mathcal{A}_{B_4}, \tr(R_{([1^3],[1])}(h)) = \phi(\tr(S(R_{([1],[3])}(h)))z(h)$.

Writing $U = S_1S_2^{-1}, V = TS_1T^{-1}S_2^{-1}, X= S_2TS_1T^{-1}S_2^{-2},P=S_3S_2S_3^{-2}, Q= TS_1T^{-1}S_3^{-1}$, $R_1=R_{([1^3],[1])}$ and $R_2=R_{([1],[3])}$, we have $3-\alpha-\alpha^{-1}=\tr(R_1(P)) =\tr(R_1(PQ^{-1})) = \tr(R_2(P)) = \tr(R_2(PQ^{-1}) = \tr({}^t\!(R_2(PQ^{-1})^{-1}))=\tr({}^t\!(R_2(PQ^{-1})^{-1})
$ so $z(PQ^{-1})=z(P) = \frac{3-\alpha-\alpha^{-1}}{\Phi(3-\alpha-\alpha^{-1})}$.
This shows that $z(Q) = z(P)z(PQ^{-1})^{-1}=1$. We also have $\tr(R_1(Q)) = \tr(R_2(Q))=\tr({}^t\!(R_2(Q)^{-1}))= 2-\alpha-\alpha^{-1}$ so $1=z(Q)= \frac{2-\alpha-\alpha^{-1}}{\Phi(2-\alpha-\alpha^{-1})}$ and $\Phi(\alpha+\alpha^{-1}) = \alpha+\alpha^{-1}$. We have $\tr(R_1(U)) = \tr(R_1(V))= \tr(R_1(X)) = \tr(R_2(U)) = \tr(R_2(V))=\tr(R_2(X)) = \tr({}^t\!(R_2(X)^{-1}))= \tr({}^t\!(R_2(V)^{-1}))=\tr({}^t\!(R_2(U)^{-1}) = 3-\alpha-\alpha^{-1}$ so $z(U) = z(V) = z(X) =1$. $\tr(R_1(UV^{-1})) = \tr(R_2(UV^{-1})) = \tr({}^t\!(R_2(UV^{-1})^{-1})) = 4-\alpha-\alpha^{-1}-\beta-\beta^{-1}$. It follows that $z(UV^{-1}) = z(U)z(V)^{-1}= 1 = \frac{4-(\alpha+\alpha^{-1})-(\beta+\beta^{-1})}{\Phi(4-(\alpha+\alpha^{-1})-(\beta+\beta^{-1}))}$ so $\Phi(\beta+\beta^{-1}) =\beta+\beta^{-1}$. Since $\F_q=\F_p(\alpha+\alpha^{-1},\beta+\beta^{-1})$, we have $\Phi=I_d$. $\tr(R_1(UX^{-1})) = -\frac{2\alpha^2\beta+\alpha\beta^2-\alpha^2-5\alpha\beta-\beta^2+\alpha+2\beta}{\alpha\beta}$, $\tr(R_2(UX^{-1}))=\tr({}^t\!(R_2(UX^{-1})^{-1}))) = \frac{\alpha^2\beta^2-2\alpha^2\beta-\alpha\beta^2+5\alpha\beta-\alpha-2\beta+1}{\alpha\beta}$. Since $z(UX^{-1}) = z(U)z(X)^{-1}=1$ and $\Phi = I_d$, it follows that
$$\alpha^2\beta^2-2\alpha^2\beta-\alpha\beta^2+5\alpha\beta-\alpha-2\beta+1=-2\alpha^2\beta-\alpha\beta^2+\alpha^2+5\alpha\beta+\beta^2-\alpha-2\beta.$$
This shows that $\alpha^2\beta^2+1=\alpha^2+\beta^2$ so $(\alpha^2-1)(\beta^2-1)=0$. This contradicts the conditions on $\alpha$ and $\beta$. This contradiction shows that $J=K\times \tilde{K}$.

We conclude using Goursat's Lemma with $R_{([1^2],[1^2])}(\mathcal{A}_{B_4})=SL_6(q)$ then with $R_{[2,1],[1]}(\mathcal{A}_{B_4}) =SL_8(q)$. 
\end{proof}

We now show that if the representation associated with each double-partition is surjective, then $\Phi_n$ is surjective.

\begin{lemme}\label{gougou}
If $n\geq 5, \F_p(\alpha,\beta)=\F_p(\alpha+\alpha^{-1},\beta+\beta^{-1})$ and the composition of $\Phi_n$ and the projection upon each quasi-simple group associated with each double-partition is surjective, then $\Phi_n$ is surjective.
\end{lemme}

\begin{proof}
Let $n\geq 5$, we know by \cite[Theorem 1.1]{BMM} that the restriction to double-partitions with an empty component is surjective. We first show that we can add the hook partitions. We then show by induction on the double-partitions using the order we picked that $\Phi_n$ is surjective.

\bigskip

We write $G_{0,0}=SL_{n-1}(\tilde{q})\times \underset{(\lambda_1,\emptyset)\in A\epsilon_n, \lambda_1<\lambda_1'}\prod SL_{n_\lambda}(\tilde{q}) \times \underset{(\lambda_1,\emptyset)\in A\epsilon_n,\lambda_1=\lambda_1'}\prod OSP'(\lambda)$ where $OSP'(\lambda)$ is the derived subgroup of the group of isometries of the $\F_{\tilde{q}}$-bilinear form defined in \cite{BMM} which identifies to the one defined in this article. We then have by Theorem 1.1. of \cite{BMM} that the image of $\mathcal{A}_{B_n}$ inside $G_{0,0}$ is $G_{0,0}$.  We have $G_{0,1}= R_{[1],[n-1]}(\mathcal{A}_{B_4}) = SL_n(q)$ by Proposition \ref{lesgourgues}. We use Goursat's Lemma to show that the image of $\mathcal{A}_{B_n}$ inside $G_{0,0}\times G_{0,1}$ is equal to $G_{0,0} \times G_{0,1}$. Using the notations in Goursat's Lemma, we have $K_1=G_{0,0}, K_2=G_{0,1}$ and $K_1/K^1\simeq K_2/K^2$. If the quotients are abelian then we are done since the groups we consider are perfect. We assume that they are non-abelian and show there is a contradiction. The only non-abelian decomposition factor of  $K_2$ is $PSL_n(q)$ so since the finite classical simple groups are non-isomorphic as long as $n\geq 4$ and $q\geq 4$ \cite[Section 1.2]{W} , there would exist a decomposition factor of $K_1$ corresponding to a double-partition $\lambda$ of $n$ with its right component empty such that $PG(\lambda)=\overline{R_{\lambda}(\mathcal{A}_{B_n})} \simeq PSL_n(q)=\overline{R_{([1],[n-1])}(\mathcal{A}_{B_n})}$. Therefore, up to conjugation \cite[Section 3.3.4]{W}, we have that $R_{\lambda}(h) = S^\Phi(R_{([1],[n-1])}(h))z(h)$ for all $h\in \mathcal{A}_{B_n}$ with $z :\mathcal{A}_{B_n}\rightarrow \F_q^\star$, $\Phi$ an automorphism of $\F_q$ and $S$  being either the identity or the transpose of the inverse. Since $n\geq 5$, we have $\mathcal{A}_{B_n}$ perfect \cite[Corollary]{MR} so $z$ is trivial. By Lemma \ref{abel} and since the abelianization of  $A_{B_n}$ is the group $\{\overline{T},\overline{S_1}\} \simeq \Z^2$, we have $R_{\lambda}(h) = S^\Phi(R_{([1],[n-1])}(h))u^{\ell_1(h)}v^{\ell_2(h)}$. Since $\lambda$ has its right component empty, the only eigenvalue of $R_{\lambda}(T)$ is $\beta$. On the other hand, the eigenvalues of $S^\phi(R_{([1],[n-1])}(T))v$ are equal to $\{v\Phi(\beta), -v\}$ or $\{v\Phi(\beta^{-1}),-v\}$, so we would have $-v=v\Phi(\beta)$ or $-v=v\Phi(\beta^{-1})$ which is not possible because we have $\beta\neq-1$. This contradiction shows that the image is equal to $G_{0,0}\times G_{0,1}$.

Assume now $G_{0,2} = G_{0,0}\times G_{0,1}$ and $G_{0,3}=R_{([1],[1^n-1])}(\mathcal{A}_{B_n}) =SL_n(q)$ and consider the image of $\mathcal{A}_{B_n}$ inside $G_0=G_{0,2}\times G_{0,3}$. We use Goursat's Lemma with $K_1= G_{0,2}$ and $K_2=G_{0,3}$. In the same way as before, it is sufficient to show that the quotients $K_1/K^1\simeq K_2/K^2$ are abelian. The sets of eigenvalues of $S^{\Phi}(R_{([1],[1^{n-1}])}(T)$ are again $\{\Phi(\beta),-1\}$ or $\{\Phi(\beta)^{-1},-1\}$. If the quotients were non-abelian, we would have $R_{([1],[n-1])}(h) =S^{\phi}(R_{([1],[1^{n-1}])}(h))z(h)$ for all $h\in \mathcal{A}_{B_n}$ with $S$, $
\phi$ and $z$ as before. We have $z$ trivial since $\mathcal{A}_{B_n}$ is perfect so $R_{([1],[n-1])}(h) =S^{\phi}(R_{([1],[1^{n-1}])}(h))u^{\ell_1(h)}v^{\ell_2(h)}$. Let us show $\Phi$ is trivial. We have $R_{([1],[n-1])}(T)$ and $R_{([1],[1^{n-1}])}(T)$ both have for eigenvalues $-1$ with multiplicity $n-1$ and $\beta$ with multiplicity $1$. This shows that either $\beta=v\Phi(\beta)$ and $-1=-v$ or  $\beta=v\Phi(\beta)^{-1}$ and $-1=-v$. In both cases, $v=1$ and $\Phi(\beta+\beta^{-1})=\beta+\beta^{-1}$. The eigenvalues of $R_{([1],[n-1])}(S_1)$ are $-1$ with multiplicity $1$ and $-\alpha$ with multiplicity $n-1$ and the eigenvalues of $R_{([1],[1^{n-1}])}(S_1)$ are $-1$ with multiplicity $n-1$ and $\alpha$ with multiplicity $1$ so we have either $-1=u\Phi(\alpha)$ and $\alpha=-u$ or $-1=u\Phi(\alpha^{-1})$ and $\alpha = -u$. In both cases $u=-\alpha$ and $\Phi(\alpha+\alpha^{-1})=\alpha+\alpha^{-1}$. We have $\Phi$ trivial so $\F_q=\F_p(\alpha+\alpha^{-1},\beta+\beta^{-1})$. This would imply $R_{([1],[n-1])|\mathcal{A}_{B_n}} \simeq S(R_{([1],[1^{n-1}])|\mathcal{A}_{B_n}})$ but $([1],[1^{n-1}])\notin \{([1],[n-1]),([1],[n-1])'\}$ when $n >2$. By Proposition \ref{isomorphisme}, this is absurd. This shows the image of $\mathcal{A}_{B_n}$ in $G_0$ is equal to $G_0$.

For $\lambda_0 \in \epsilon_n=\{\lambda \vdash\vdash n, \lambda \notin A_n, \lambda~\mbox{not a hook}\}$, we set
$$G_{\lambda_0} = SL_{n-1}(\F_{\tilde{q}})\times \underset{(\lambda_1,\emptyset)\in A\epsilon_n, \lambda_1<\lambda_1'}\prod SL_{n_\lambda}(\tilde{q}) \times \underset{(\lambda_1,\emptyset)\in A\epsilon_n,\lambda_1=\lambda_1'}\prod OSP'(\lambda)\times$$ $$SL_n(\F_q)^2\times \underset{\lambda\in \epsilon_n, \lambda < min(\lambda',\lambda_0)}\prod SL_{n_\lambda}(q) \times \underset{\lambda\in \epsilon_n, \lambda=\lambda'< \lambda_0}\prod OSP'(\lambda).$$
where $OSP(\lambda)$ is the group of isometries of the bilinear form defined before Proposition \ref{bilin}.

For the minimal element $\lambda_0$ of $\epsilon_n$, we just showed the composition $\Phi_n$ with the projection onto $G_{\lambda_0}=G_0$ is surjective. Let us show by induction (numbering the double-partitions of $n$ with the order defined previously) that for all $\lambda_0$, the composition of $\Phi_n$ with the projection onto $G_{\lambda_0}$ is surjective.

Let $\lambda_0\in \epsilon_n$. Assume the composition is surjective onto $G_{\lambda_0}$ and let us show that the composition onto $G_{\lambda_0+1}= G_{\lambda_0}\times G(\lambda_0)$ is surjective where $G(\lambda_0)= SL_N(q)$ if $\lambda_0\neq\lambda_0'$ and $G(\lambda_0) =OSP'(\lambda_0)\in \{SP_N(q),\Omega_N^{+}(q)\}$ if $\lambda_0=\lambda_0'$. We use Goursat's Lemma with $K_1=G_{\lambda_0}$ and $K_2=G(\lambda_0)$ on the image $\Phi_n$ in $K_1\times K_2$. As before, it is sufficient to show that the quotients $K_1/K^1 \simeq K_2/K^2$ are abelian. Assume they are non-abelian. The only non-abelian Jordan-Hölder factor of $G(\lambda_0)$ is $PG(\lambda_0)$, so there exists $\lambda$ less than $\lambda_0$ such that up to conjugation (see \cite{W} 3.3.4., 3.5.5. and 3.7.5) $\overline{R_{\lambda}(h)}=S^{\Phi}(R_{\lambda_0}(h)z(h)$ for all $h\in \mathcal{A}_{B_n}$ (There is no triality involved since if $n\geq 5$, $\lambda=\lambda'$ and $\lambda\in \epsilon_n$ then $\dim(V_\lambda) > \dim(V_{([2,1],[2,1])})=80>8$). By the same arguments as in the induction initialization, we have that $\lambda$ has no empty components. Since $n\geq 5$, $\mathcal{A}_{B_n}$ is perfect and $z$ is trivial. We then have $R_{\lambda_0|\mathcal{A}_{B_n}}\simeq S^{\Phi}(R_{\lambda|\mathcal{A}_{B_n}})$. Let us show that $\Phi$ is trivial. By Lemma \ref{abel}, there exists $u,v\in \F_q^\star$ such that up to conjugation, for all $h\in A_{B_n}$, we have $R_ {\lambda_0}(h)= S^\Phi(R_\lambda(h))u^{\ell_1(h)}v^{\ell_2(h)}$. Comparing eigenvalues of $T$, we get either $\{\beta,-1\}=\{v\Phi(\beta),-v\}$ or $\{\beta,-1\}=\{v\Phi(\beta^{-1}),-v\}$. In the first case, either  $v=1$ and $\Phi(\beta)=\beta$ or $v=-\beta$ and $-1=v\Phi(\beta)$ so $\Phi(\beta+\beta^{-1})=\beta+\beta^{-1}$.
In the second case either $v=1$ and $\Phi(\beta^{-1})=\beta$ or $v=-\beta$ and $v\Phi(\beta^{-1}) = -1$ so $\Phi(\beta+\beta^{-1})=\beta+\beta^{-1}$. In the same way using $S_1$, we show $\Phi(\alpha+\alpha^{-1})= \alpha+\alpha^{-1}$. This shows that $\Phi$ is trivial because $\F_q=\F_p(\alpha+\alpha^{-1},\beta+\beta^{-1})$. We then have $R_{\lambda_0|\mathcal{A}_{B_n}}\simeq S(R_{\lambda|\mathcal{A}_{B_n}})$ which contradicts Proposition \ref{isomorphisme} since $\lambda< \lambda_0\leq \lambda_0'$.
\end{proof}

\bigskip

To get that $\Phi_n$ is surjective, it now only remains to show that what we assumed in Lemma \ref{gougou} is true.

\begin{theo}\label{hortillonage}
If $n\geq 5$ then for all $\lambda\vdash\vdash n$ double-partitions in our decomposition, we have $R_\lambda(\mathcal{A}_{B_n})= G(\lambda)$ where $G(\lambda)$ is the group in the following list.
\begin{enumerate}
\item $SL_{n-1}(\tilde{q})$ if $\lambda =([n-1,1],\emptyset)$.
\item $SL_N(\tilde{q})$ if $\lambda = (\lambda_1,\emptyset), \lambda_1< \lambda_1'$.
\item $SP_N(\tilde{q})$ if $\lambda = (\lambda_1,\emptyset), \lambda_1=\lambda_1'$ and ( $p=2$ or ($p\geq 3$ and $\nu(\lambda_1) = -1$)).
\item $\Omega_N^{+}(\tilde{q})$ if $\lambda=(\lambda_1,\emptyset), \lambda_1=\lambda_1'$, $p\geq 3$ and $\nu(\lambda_1)=1$.
\item $SL_n(q)$ if $\lambda \in \{([1],[n-1]),([1],[1^{n-1}])\}$.
\item $SL_N(q)$ if $\lambda \in \epsilon_n, \lambda< \lambda'$.
\item $SP_N(q)$ if $\lambda=\lambda'$ and ($p=2$ or $\tilde{\nu}(\lambda)=-1)$.
\item $\Omega_N^{+}(q)$ if $\lambda = \lambda', p\geq 3$ and $\tilde{\nu}(\lambda)= 1$.
\end{enumerate}
\end{theo}

\begin{proof}
Let $n\geq 5$. By \cite{BM} (Theorem 1.1.), it is sufficient to show it for $\lambda \in \epsilon_n$.

The result is true for $n=4$, so we can use induction and assume $\Phi_{n-1}$ is surjective.

The main result we use to show this theorem is a theorem by Guralnick and Saxl \cite{GS}. Recall that the proof of \cite{GS} uses the classification of finite simple groups.

\begin{theo}[Gulralnick-Saxl]\label{CGFS}
Let $V$ be a finite-dimensional vector space of dimension $d > 8, d\neq 10$ over an algebraically closed field $\overline{\F_p}$ of characteristic $p>0$. Let $G$ be a primitive tensor-indecomposable finite irreducible subgroup of $GL(V)$. We write $v_G(V)$ the minimal dimension of $[\beta g,V] = (\beta g -1)V$, for $g\in G$ and $\beta \in \overline{\F_p}$ such that $\beta g \neq 1$. We then have either $v_G(V) > max(2,\frac{\sqrt{d}}{2})$ or one of the following assertions.
\begin{enumerate}
\item G is a classical group in a natural representation. 
\item G is the alternating or the symmetric group of degree $c$ and $V$ is the permutation module of dimension $c-1$ or $c-2$.
\end{enumerate}
\end{theo}

The first thing to do is to take care of the double-partitions such that $n_\lambda>8$ and $n_\lambda\neq 10$. For $n=5$, the double partitions to consider are $([1^3],[1^2])$, $([1^2],[1^3])$, $([1],[2,2])$, $([1],[2,1^2])$, $([1],[3,1])$, $([2],[2,1])$ and $([1^2],[2,1])$ of respective dimensions $10$, $10$, $10$, $15$, $15$, $20$ and $20$. For $n= 6$, they are $([1^2],[1^4])$, $([1^4],[1^2])$, $([1^3],[1^3])$, $([1],[4,1])$, $([1],[2,1^3])$, $([1],[3,2])$, $([1],[2^2,1])$, $([1],[3,1,1])$, $([2],[2^2])$, $([1^2],[2^2])$, $([2],[3,1])$, $([1^2],[3,1])$, \break$([2],[2,1^2])$, $([1^2],[2,1^2])$, $([3],[2,1])$, $([1^3],[2,1])$ and $([2,1],[2,1])$ of respective dimensions $15$, $15$, $20$, $24$, $24$, $30$, $30$, $36$, $30$, $30$, $45$, $45$, $45$, $45$, $40$, $40$ and $80$. We can now note that if $n=6$, we have $n_\lambda \geq 15$ so by the branching rule, if $n\geq 6$ and $\lambda$ is a double-partition of $\epsilon_n$ then $n_\lambda \geq 15$. The only double-partitions $\lambda$ such that $n_\lambda\leq 8$ or $n_\lambda=10$ are $([1^3],[1^2]), ([1^2],[1^3])$ and $([1],[2,2])$ which are of dimension $10$. By Lemma \ref{platypus} and the branching rule, we have that $R_{[1],[2,2]}(\mathcal{A}_{B_4}) = SL_8(q) \times SL_2(\tilde{q})$, $R_{[1^3],[1^2]}(\mathcal{A}_{B_4})=SL_4(q)\times SL_6(q)$ and $R_{[1^3],[1^2]}(\mathcal{A}_{B_4})=SL_4(q)\times SL_6(q)$. By Theorem \ref{LBJ}, we have that $R_{([1],[2,2])}(\mathcal{A}_{B_5})=R_{[1^3],[1^2]}(\mathcal{A}_{B_5})=R_{[1^2],[1^3]}(\mathcal{A}_{B_5}) \simeq SL_{10}(q)$.

We now show in the same way as in \cite{BMM} (Part 5), that the other assumptions of Theorem \ref{CGFS} are verified. In order to do this, we use the following results shown in \cite{BMM}.

\begin{lemme}\label{tens1}
If $d\geq 6$ and $G\leq GL_d(q)$ contains an element conjugated to $\diag(\xi  ,\xi^{-1} ,1,1,...)$ with $\xi^2\neq 1$, then $G$ is tensor-indecomposable.
\end{lemme}

\begin{lemme}\label{tens2}
If $d\geq 16$ and $G\leq GL_d(q)$ contains an element of order prime to $p$ conjugated to a\\ $\diag(\xi,\xi,\xi^{-1}, \xi^{-1},1,1,..,1)$ with $\xi^2\neq 1$, then $G$ is tensor-indecomposable except possibly if  $G \leq G_1\otimes G_2$ with $G_1\leq GL_2(q)$.
\end{lemme}

For a block matrix with blocks $B_1,\dots,B_r$, we write $\diag(B_1,\dots,B_r)$.

\begin{lemme}
If $G$ contains a natural $SL_2(q)$ and $q \geq 8$ or $G$ contains a twisted diagonal embedding of  $SL_2(q)$ $(G \supset \{\diag(M,{}^t\!(M^{-1}),I_{N-4}), M\in SL_2(q)\})$, then case $2$ of Theorem \ref{CGFS} is excluded.
\end{lemme}

By the proof of the imprimitivity of $G$ in \cite{BMM}, it is sufficient to show that $\mathcal{A}_{B_n}$ is normally generated by $\mathcal{A}_{B_{n-1}}$ and that $G$ contains either a transvection or an element of Jordan form $\diag(I_2+E_{1,2},I_2+E_{1,2},I_{N-4})$ to get that $G$ is imprimitive.

In order to show that we are in case $1$ of Theorem \ref{CGFS}, we must show that for $n\geq 5$, that we have $q\geq 8$ and that for any double-partition $\lambda$ of $n$, $G=R_\lambda(\mathcal{A}_{B_n})$ contains either a natural $SL_2(q)$ and  $n_\lambda> 6$ or contains a twisted diagonal embedding of $SL_2(q)$ and $n_\lambda > 16$. We must also prove that $\mathcal{A}_{B_n}$ is normally generated $\mathcal{A}_{B_{n-1}}$ and the exceptional case of Lemma \ref{tens2} is impossible when $n_\lambda > 16$, $G$ contains a twisted diagonal embedding of $SL_2(q)$ but no natural $SL_2(q)$ in an obvious way.

Let $n\geq 5$, assume the lemma is true for all $m\leq n-1$. By Lemmas \ref{gougou} and \ref{platypus}, we have $\Phi_m$ surjective for all $m\leq n-1$.  By assumption, $\alpha$ is of order strictly greater than $5$ and not belonging to $\{1,2,3,4,5,6,8,10\}$. This implies that $\alpha$ is of order at least $7$ and that $q\geq 8$. If $\lambda$ has at most two columns then since $\lambda\in \epsilon_n, \lambda$ contains a natural $SL_2(q)$. Assume now $\lambda\in \epsilon_n$ has at least three rows or three columns.

Assume that for all $\mu \subset \lambda$ containing $([2,1],[1])$ or $([1],[2,1])$, we have $\mu' \subset \lambda$. We then have that $\lambda = \lambda'$ and $n$ is even. Since $n$ is even, we have $\mu \neq \mu'$ for any double-partition $\mu\subset \lambda$. Since $\lambda \in \epsilon_n$ and contains strictly more than two rows and two columns, there exists $\mu\subset \lambda$ containing $([1],[2,1])$ or $([2,1],[1])$. Since $\Phi_{m-1}$ is surjective and $\mu'\subset \lambda$, we have a twisted diagonal embedding of $SL_{n_\mu}(q)$ in $G= R_\lambda(\mathcal{A}_{B_n})$ and since $n_\mu \geq 8 \geq 2$, we have a twisted diagonal embedding of $SL_2(q)$. Otherwise there exists $\mu\subset \lambda$ containing $([2,1],[1])$ or $([1],[2,1])$ such that $\mu' \subset \lambda$. Since $\Phi_{m-1}$ is surjective, we get that $\lambda$ contains a natural $SL_{n_\mu}(q)$ and so contains a natural $SL_2(q)$. For double-partitions which are not of dimension strictly greater than $16$, i.e. $([1],[2,1^2])$ and $([1],[3,1])$, we are in the second case.

We now show that $\mathcal{A}_{B_n}$ is normally generated by $
\mathcal{A}_{B_{n-1}}$ for $n\geq 5$. By \cite[Lemma 2.1]{BMM}, we have that $\mathcal{A}_{A_n}$ is normally generated by $\mathcal{A}_{A_{n-1}}$ for $n\geq 4$. Since $T$ commutes with $S_i$ for all $i \geq 2$, we have the same result for $\mathcal{A}_{B_n}$ for $n\geq 4$.

It now only remains to show that the exception of Lemma \ref{tens2} is impossible when there is no obvious natural $SL_2(q)$ in $G$. In order to do this, we show a proposition analogous to Proposition 2.4. of \cite{BMM}.

\begin{prop}
Let $K$ be a field, if $n\geq 7$ and $\varphi : \mathcal{A}_{B_n} \rightarrow PSL_2(K)$ is a group morphism then $\varphi =1$.
\end{prop}

\begin{proof}
Let $K$ be a field, $n\geq 7$ and $\varphi$ such a morphism. The restriction $\varphi$ to $\mathcal{A}_{A_n} \leq \mathcal{A}_{B_n}$ is trivial by Proposition 2.4. of \cite{BMM}. By Theorem 3.9. of \cite{MR}, $\mathcal{A}_{B_n}$ is generated by $p_0=S_{n-2}S_{n-1}^{-1}, p_1=S_{n-1}S_{n-2}S_{n-1}^{-2}, q_3=S_{n-3}S_{n-1}^{-1}, b = S_{n-2}S_{n-1}^{-1}S_{n-3}S_{n-2}^{-1}, r_l=T^lS_1T^{-l}S_{n-1}, q_i = S_{n-i}S_{n-1}^{-1}, l\in \Z, 4\leq i \leq n-2$ and the following relations :
\begin{enumerate}
\item For $4\leq j\leq n-2, p_0q_j =q_jp_1$ and $p_1q_j=q_jp_0^{-1}p_1$.
\item For $l\in \Z, p_0r_l = r_lp_1$ and $p_1r_l=r_lp_0^{-1}p_1$.
\item For $3\leq i <j\leq n-2, \vert i-j\vert \geq 2, q_iq_j=q_jq_i$.
\item For $3\leq i \leq n-3, q_ir_l=r_lq_i$.
\item $p_0q_3p_0^{-1}=b, p_0bp_0^{-1}=b^2q_3^{-1}b$
\item $p_1q_3p_1^{-1} = q_3^{-1}b, p_1bp_1^{-1}=(q_3^{-1}b)^3q_3^{-2}b$.
\item For $3\leq i \leq n-3, q_iq_{i+1}q_i = q_{i+1}q_iq_{i+1}$.
\item For $l\in \Z, q_{n-2}r_lq_{n-2} = r_lq_{n-2}r_l$.
\item For $l\in \Z, r_lr_{l+1}=r_{l+1}r_{l+2}$.
\end{enumerate}

By \cite[Proposition 2.4]{BMM}, the images of all the generators except for $(r_l)_{l\in \Z}$ are trivial. By the eighth relation, we get that the images of the $r_l$ are also trivial and the desired result follows.
\end{proof}

This shows that if $n\geq 7$ and $G \leq G_1\otimes G_2$ with $G_1\leq GL_2(q)$, then $G \subset SL_{\frac{N}{2}}(q) \times SL_{\frac{N}{2}}(q)$. This contradicts the irreducibility. Since we need $n\geq 7$ to apply this reasoning, we must consider separately the cases where $n\in \{5,6\}$ and $G$ does not contain a natural $SL_2(q)$. Looking at all the cases enumerated previously, the only one to consider is $\lambda=([2,1],[2,1])$. Up to conjugation, we have $H=R_{[2,1],[2,1]}(\mathcal{A}_{B_5}) = \{\diag(M ,{}^t\!(M^{-1}), N ,{}^t\!(N^{-1})), M,N\in SL_{20}(q)\} \simeq SL_{20}(q) \times SL_{20}(q)$.

Assume that $G=R_{([2,1],[2,1])}(\mathcal{A}_{B_6})  \subset G_1\otimes G_2$ and that $G_1\subset GL_2(q)$. We then have a morphism $\theta$ from $G$ to $SL_2(q)$ since $\mathcal{A}_{B_n}$ is perfect for $n\geq 5$. If we consider the restriction of $\theta$ to $H$, its kernel is a subgroup of $H$ and its image is a subgroup of $SL_2(q)$. Since $PSL_{20}(q)$ is the only non-abelian composition factor of $H$, we have that if the image is non-abelian then there exists a subgroup of $SL_2(q)$ isomorphic to $PSL_{20}(q)$. This is absurd so the image is abelian and the kernel contains the derived subgroup of $H$ which is equal to $H$ since $H$ is perfect. In the same way, for all $g\in G$, the restriction of $\theta$ to $gHg^{-1}$ is trivial and since $H$ normally generates $G$, $\theta$ is trivial which contradicts the irreducibility of $G$ in the same way as in the proof of the previous proposition.

We have thus shown that we are in the first case of Theorem \ref{CGFS}. By the same reasoning as in \cite[page 16]{BMM}, we have in all cases that $q'=q$. If $\lambda=\lambda'$, we have $G \subset G(\lambda)$ by Proposition \ref{transpose} so $G =G(\lambda)$. If $\lambda \neq \lambda'$, $G$ preserves no bilinear form since $R_\lambda$ is not isomorphic to $R_\lambda^\star$.  If $G$ preserves a hermitian form then there exists an automorphism $\Phi$ of order $2$ of $\F_q$ such that $M$ is conjugated to $\Phi({}^t\!(M)^{-1})$ for all $M\in G$. Since $G$ contains a natural $SL_2(q)$, we then have $\tr(\diag(\alpha,\alpha^{-1},1,1,...,1)) = \Phi(\tr(\diag(\alpha^{-1}, \alpha,1,...,1)))$ and $\tr( \diag(\beta,\beta^{-1},1,1,...,1)) = \Phi(\tr(\diag(\beta^{-1}, \beta,1,...,1)))$ so $\Phi(\alpha+\alpha^{-1}) =\alpha+\alpha^{-1}$ and $\Phi(\beta+\beta^{-1})= \beta+\beta^{-1}$. This implies that $\Phi=I_d$ because $\F_q=\F_p(\alpha+\alpha^{-1}, \beta+\beta^{-1})$. This is absurd and we conclude that $G =SL_{n_\lambda}(q)$.
\end{proof}

By Theorem \ref{hortillonage}, Lemma \ref{platypus} and Lemma \ref{gougou}, we have that for all $n, \Phi_n$ is surjective.

\subsubsection{Cases 2-3}

We have shown the surjectivity of $\Phi_n$ for the first of the six possible field extension configurations described at the beginning of \ref{lalala}. The proof in cases $2$ and $3$ only requires small changes to the one in the first case, but the new factorizations appearing in cases $4$ to $6$ require more work, especially for the low dimensional representations. We treat in this subsection cases $2$ and $3$ emphasizing on the differences with the first case.

In case 2, i.e., $\F_q=\F_p(\alpha,\beta) = \F_p(\alpha+\alpha^{-1},\beta+\beta^{-1})$ and $\F_p(\alpha) \neq \F_p(\alpha+\alpha^{-1})$, the same arguments as the ones in case 1 work at every step of the proof. Indeed, $SU_2(\tilde{q}^{\frac{1}{2}})$ is also generated by a conjugacy class of transvections. Since $\tilde{q}$ is a square and $\alpha$ is order not diving $8$ by assumption, we have that $\tilde{q}\geq 16$ and $\tilde{q}^{\frac{1}{2}}> 3$. We also still have that $SL_8(q)\times SU_2(\tilde{q}^{\frac{1}{2}})$ contains $SL_8(q) \times \{I_2\}$ so all the arguments works in the same way. This shows that in case $2, \Phi_n$ is surjective for all $n$.

In case $3$, i.e., $\F_q=\F_p(\alpha,\beta)=\F_p(\alpha+\alpha^{-1},\beta)= \F_p(\alpha,\beta+\beta^{-1})\neq \F_p(\alpha+\alpha^{-1},\beta+\beta^{-1})$, all representations are unitary. The main differences occur in the proof that when $n=4$, the direct product of two $SU_4(q^\frac{1}{2})$ is in the image, and in the conclusion of the proof of this version of Theorem \ref{hortillonage}.
\begin{theo}
If $n\geq 5$, then for all $\lambda\vdash\vdash n$ in our decomposition, $R_\lambda(\mathcal{A}_{B_n})= G(\lambda)$ where $G(\lambda)$ is the corresponding group in the following list.
\begin{enumerate}
\item $SU_{n-1}(\tilde{q}^{\frac{1}{2}})$ if $\lambda =([n-1,1],\emptyset)$.
\item $SU_N(\tilde{q}^{\frac{1}{2}})$ if $\lambda = (\lambda_1,\emptyset), \lambda_1< \lambda_1'$.
\item $SP_N(\tilde{q}^{\frac{1}{2}})$ if $\lambda = (\lambda_1,\emptyset), \lambda_1=\lambda_1'$ and ( $p=2$ or ($p\geq 3$ and $\nu(\lambda_1) = -1$)).
\item $\Omega_N^{+}(\tilde{q}^\frac{1}{2})$ if $\lambda=(\lambda_1,\emptyset), \lambda_1=\lambda_1'$, $p\geq 3$ and $\nu(\lambda_1)=1$.
\item $SU_n(q^\frac{1}{2})$ if $\lambda \in \{([1],[n-1]),([1],[1^{n-1}])\}$.
\item $SU_N(q^\frac{1}{2})$ if $\lambda \in \epsilon_n, \lambda< \lambda'$.
\item $SP_N(q^\frac{1}{2})$ if $\lambda=\lambda'$ and ($p=2$ or ($p\geq 3$ and $\tilde{\nu}(\lambda)=-1$.
\item $\Omega_N^{+}(q^\frac{1}{2})$ if $\lambda = \lambda', p\geq 3$ and $\tilde{\nu}(\lambda)= 1$.
\end{enumerate}
\end{theo}

\begin{proof}
 We recall Proposition $4.1.$ of \cite{BMM}.
\begin{prop}\label{coolprop}
Let $q=u^2$, $\varphi$ be a non-degenerate bilinear form over $\F_q^N$, $\psi$ a non-degenerate hermitian form over $\F_q^N$. If $G\subset OSP_N(\varphi) \cap U_N(\psi)$ is absolutely irreducible, then there exists $x\in GL_N(q)$ and a non-degenerate bilinear form $\varphi'$ over $\F_u^N$ such that $^x G \subset OSP(\varphi')$ and $\varphi'$ is of the same type as $\varphi$.
\end{prop}

When $n=4$, the proof that $\Phi_4$ is surjective is the same up to the point where we prove $\Phi$ is trivial using $\Phi(\alpha+\alpha^{-1})=\alpha+\alpha^{-1}$ and $\Phi(\beta+\beta^{-1})=\beta+\beta^{-1}$. In case $3$, $\Phi$ could also be equal to the automorphism $\epsilon$ of order $2$ of $\F_q$. It is thus necessary to show that the following is absurd :
$$\frac{\alpha^2\beta^2-2\alpha^2\beta-\alpha\beta^2+5\alpha\beta-\alpha-2\beta+1}{\alpha\beta}=\epsilon(\frac{-2\alpha^2\beta-\alpha\beta^2+\alpha^2+5\alpha\beta+\beta^2-\alpha-2\beta}{\alpha\beta}).$$
This would imply
\begin{eqnarray*}
\alpha^2\beta^2-2\alpha^2\beta-\alpha\beta^2+5\alpha\beta-\alpha-2\beta+1 &  = &\frac{-2\alpha^{-2}\beta^{-1}-\alpha^{-1}\beta^{-2}+\alpha^{-2}+5\alpha^{-1}\beta^{-1}+\beta^{-2}-\alpha^{-1}-2\beta^{-1}}{\alpha^{-2}\beta^{-2}}\\
& = & -2\beta-\alpha+\beta^2+5\alpha\beta+\alpha^2-\alpha\beta^2-2\alpha^2\beta.
\end{eqnarray*}
This is absurd because it is the same equality we proved to be impossible in the first case.

We now adapt the end of the proof of the corresponding version of Theorem \ref{hortillonage}.  By \cite[page 18]{BMM},  we are in case $1$ of Theorem \ref{CGFS}. If $\lambda \neq \lambda'$, $G$ contains a natural $SU_3(q^\frac{1}{2})$ so $q=q'$ by Lemma \ref{field}. Since $G\subset SU_{n_\lambda}(q^\frac{1}{2})$ and $G$ preserves no bilinear form by Proposition \ref{isomorphisme}, we have $G=SU_{n_\lambda}(q^{\frac{1}{2}})$. If $\lambda=\lambda'$, we use Proposition \ref{coolprop} to get that $G\subset OSP(q^\frac{1}{2})$. By Lemma \ref{field}, we have that $\F_{q'}$ contains $\{x+\epsilon(x), x\in \F_q\}$. This implies that $q'=q^{\frac{1}{2}}$ because $\F_{q'}$ contains $\alpha+\alpha^{-1}$ and $\beta+\beta^{-1}$ and $q'$ divides $q^\frac{1}{2}$. We conclude that $G=OSP'(q^\frac{1}{2})$.
\end{proof}

\subsubsection{Cases 4-5-6}

In this subsection, we finish the proof for type $B$ by considering the last three cases for the field extensions. In these cases more factorizations appear and this complicates greatly the proof for small $n$. We will use the tables of maximal subgroups of finite classical groups in low dimension from \cite{BHRC}. This gives interesting techniques to determine if a certain subgroup $G$ of a classical group is the group itself, when given information on the subgroups of $G$.

In these cases, we can still use various arguments from the first case, but except for Proposition \ref{lesgourgues} which remains true in all these cases, all the low-dimensional cases must be done again. It is not necessary to use new arguments for Lemma \ref{gougou}. We start by studying the case $n\leq 4$.

\begin{lemme}
In cases $4,5$ and $6$, we have $\Phi_n$ surjective for $n\leq 4$.
\end{lemme}

\begin{proof}
For $n=2$, we have, using the same arguments as in the first case and  Lemma \ref{Ngwenya}, that $\mathrm{Im}(\Phi_2)=SL_2(q^\frac{1}{2})$.

For $n=3$, we have by the factorizations in Proposition \ref{isomorphisme} that in all cases $\Phi_3$ is surjective.

The only case left to consider is $n=4$ and the double-partitions $ ([1^2],[1^2])$ and $([1],[2,1])$ of respective dimensions $6$ and $8$. We have to prove that $R_{([1^2],[1^2])}(\mathcal{A}_{B_4})=SL_6(q^\frac{1}{2})$ or $SU_6(q^{\frac{1}{2}})$ and $R_{([2,1],[1])}(\mathcal{A}_{B_4}) =SL_8(q^\frac{1}{2})$ or $SU_8(q^\frac{1}{2})$ depending on the case.

We start by $G=R_{([1^2],[1^2])}(\mathcal{H}_4)$ in case $5$ or $6$, where we have $\F_q = \F_p(\alpha,\beta) = \F_p(\alpha+\alpha^{-1},\beta)\neq  \F_p(\alpha,\beta+\beta^{-1})= \F_p(\alpha+\alpha^{-1},\beta+\beta^{-1})$. We then have $H=R_{([1^2],[1^2])}(\mathcal{A}_{B_3}) \simeq SL_3(\F_q)$. Since $([1^2],[1^2])=([1^2],[1^2])$, by Proposition \ref{isomorphisme} (4.c) and Lemma \ref{Ngwenya}, up to conjugation, we have $G\subset SL_6(q^\frac{1}{2})$.

We use the classification of maximal subgroups $SL_6(q^{\frac{1}{2}})$ \cite[Tables 8.24 and 8.25]{BHRC} . Using the fact that $H$ is a subgroup of $G$, we exclude the possibility that $G$ is included in all but two of these groups, using the divisibility of the cardinals that would ensue.  We start by considering the sporadic maximal subgroups in table 8.25 and get the cardinals of these groups using the atlas \cite{CCNPW}. Since $q$ is a square and $\alpha$ is of order greater than $4$, we have $q\geq 9$. This implies that $\vert SL_3(q) \vert = q^3 (q^2-1)(q^3-1) \geq 9^3*(9^2-1)(9^3-1) = 42456960$ which is greater than all the cardinals in the list (the last one is of cardinal $6q^\frac{1}{2}(q-1)(q^3-1)$ and $6 < q^\frac{3}{2}(q+1)(q^\frac{3}{2}+1)$). We now  look at the list in table 8.24 of the 18 geometric maximal subgroups of $SL_6(q^{\frac{1}{2}})$. In cases $1,2,3,4,5,6,7$ and $13$, the cardinal of the maximal subgroup divides  $q^\frac{15}{2}(q^\frac{5}{2}-1)(q^2-1)(q^\frac{3}{2}-1)^2(q-1)^2(q^\frac{1}{2}-1)$. This implies that it is sufficient to show that $\vert SL_3(q)\vert=q^3(q^3-1)(q^2-1)$ does not divide this quantity to exclude these cases.  It can be true only if $q^3-1$ divides $(q^\frac{5}{2}-1)(q^\frac{3}{2}-1)^2(q-1)^2(q^\frac{1}{2}-1)$.  The Euclidean remainder of those two quantities seen as polynomials in $q^\frac{1}{2}$ is $4q^\frac{5}{2}+2q^2-2q^\frac{3}{2}-4q-2q^\frac{1}{2}+2$. Therefore, if $q^3-1$ divides the first quantity than it divides the remainder which is positive so it is less than or equal to it.  We have $\epsilon(\alpha)=\alpha^{-1}=\alpha^{q^\frac{1}{2}}$ so $\alpha^{q^\frac{1}{2}+1}=1$. Since $\alpha$ is of order strictly greater than $6$ by assumption, we have that $q^\frac{1}{2}\geq 6$ and $4q^\frac{5}{2}+2q^2-q-2q^\frac{3}{2}-4q-2q^\frac{1}{2}+2\leq 4q^\frac{5}{2}+2q^2+2\leq 4q^\frac{5}{2}+3q^2\leq 5q^\frac{5}{2}<q^3-1$. This gives us the desired contradiction.

Using the same type of arguments, we exclude all the remaining cases except for cases $11$ and $17$ which are $SL_3(q).(q^\frac{1}{2}+1).2$ and $(q^\frac{1}{2}-1,3)\times SP_6(q^\frac{1}{2})$. We start by excluding case $11$.

We know $H=R_{([1^2],[1^2])}(\mathcal{A}_{B_3})\simeq SL_3(q)$ normally generates $G=R_{([1^2],[1^2])}(\mathcal{A}_{B_4})\subset P^{-1}SL_6(\F_{q^\frac{1}{2}})P$ for a certain matrix $P$ in $GL_6(q)$. Assume $PGP^{-1}$ is a subgroup of  $M=SL_3(q).(q^\frac{1}{2}+1).2$, since $SL_3(q)$ is perfect, $PHP^{-1}$ is perfect and the image of $PHP^{-1}$ in the quotient $\Z/2\Z$ of $M$ is trivial. $H$ is thus included in $SL_3(q).(q^\frac{1}{2}+1)$. Using the same argument, we have that $PHP^{-1}$ is included in the $SL_3(q)$ appearing in the expression of $M$ so $PHP^{-1}$ is equal to that $SL_3(q)$. For all $g\in PGP^{-1}$, we can apply the same reasoning to $gPHP^{-1}g^{-1}= SL_3(q)=PHP^{-1}$ so $PGP^{-1}=SL_3(q)=PHP^{-1}$ because $H$ normally generates $G$ so $H=G$. This leads to a contradiction because $G$ is irreducible and $H$ is not.

It only remains to show $PGP^{-1}$ cannot be included in the maximal subgroup $M= (q^\frac{1}{2}-1,3)\times SP_6(q^\frac{1}{2})$ of $SL_6(q^\frac{1}{2})$. Assume it is the case. By the same arguments as in case $11$, we have that $PGP^{-1} \subset \{1\} \times SP_6(q^\frac{1}{2}) \simeq SP_6(q^\frac{1}{2})$. Looking at the cardinals in Tables $8.28$ and $8.29$ of \cite{BHRC}, we see than none of them is divisible by $\vert SL_3(q)\vert$. This implies that $SL_3(q)\simeq SP_6(q^\frac{1}{2})$ which is absurd. This excludes case $17$ and shows that $PGP^{-1}$ is equal to  $SL_6(q^{\frac{1}{2}})$. This concludes the study of double-partition $([1^2],[1^2])$ in the field cases $5$ and $6$.

Assume now we are in case $4$, i.e., $\F_q=\F_p(\alpha,\beta)= \F_p(\alpha,\beta+\beta^{-1}) \neq \F_p(\alpha+\alpha^{-1},\beta)= \F_p(\alpha+\alpha^{-1},\beta+\beta^{-1})$. We then have by Lemma \ref{Harinordoquy} and Proposition \ref{isomorphisme} that there exists a matrix $P$ such that $PGP^{-1} \subset SU_6(q^{\frac{1}{2}})$ and $H\simeq SL_3(q)$, writing again $G=R_{([1^2],[1^2])}(\mathcal{A}_{B_4})$ and $H = R_{([1^2],[1^2])}(\mathcal{A}_{B_3})$. The goal this time is to show that $PGP^{-1}=SU_6(q^{\frac{1}{2}})$. Using Tables $8.26$ and $8.27$ of \cite{BHRC}, cardinality arguments and the fact that groups in class $\mathcal{C}^1$ are not irreducible and so cannot contain $PGP^{-1}$ as a subgroup, we get that if $G\neq SU_6(q^\frac{1}{2})$, then $G$ is included in $SL_3(q).(q^\frac{1}{2}-1).2$. By the same argument as before, this is impossible. It follows that $PGP^{-1}=SU_6(q^\frac{1}{2})$.

The only double-partition remaining for $n\leq 4$ now is $\lambda=([2,1],[1])$ which affords a representation of dimension 8 and verifies $\lambda=(\lambda_1',\lambda_2')$.

We start by case $4$, i.e., $\F_q=\F_p(\alpha,\beta)=\F_p(\alpha,\beta+\beta^{-1})\neq \F_p(\alpha+\alpha^{-1},\beta)=\F_p(\alpha+\alpha^{-1},\beta+\beta^{-1})$ and so $\F_{\tilde{q}}=\F_p(\alpha+\alpha^{-1})\neq \F_p(\alpha)$. We then have by Goursat's Lemma and the result for $n=3$ that $H=R_{([2,1],[1])}(\mathcal{A}_{B_3})\simeq SL_3(q)\times SU_2(\tilde{q}^\frac{1}{2}) \subset G=R_{([2,1],[1])}(\mathcal{A}_{B_4})$. By Proposition \ref{isomorphisme}, we know that there exists $P\in GL_8(q)$ such that for all $h\in \mathcal{H}_4$,   $PR_{([2,1],[1])}(h)P^{-1}= \epsilon(R_{([2,1],[1])})(h)$. By Lemma \ref{Ngwenya}, this implies that there exists $S\in GL_8(q)$ such that $S^{-1}R_{([2,1],[1])}(\mathcal{A}_{B_4})S\subset GL_8(q^\frac{1}{2})$ with $\gamma^{-1} P=\epsilon(S)S^{-1}$ and $\epsilon(P)P=\epsilon(\gamma)\gamma$. 

We can use the arguments used previously to see that our group is primitive, tensor-indecomposable,  preserves no symmetric, skew-symmetric or hermitian form over $\F_q^{\frac{1}{2}}$ and cannot be included in $GL_8(q')$ for $q'<q^\frac{1}{2}$.  Using \cite[Tables 8.44 and 8.45]{BHRC} and the order of our subgroup $H$ of $G$, if $S^{-1}GS\neq SL_8(q^\frac{1}{2})$ then we have $S^{-1}GS \subset (((q^\frac{1}{2}-1,4)(q^\frac{1}{2}+1))\circ SL_4(q)).\frac{(q-1,4)}{(q^\frac{1}{2}-1,4)}.2$, which is in class $\mathcal{C}^3$ of $SL_8(q^\frac{1}{2})$. An element of class $\mathcal{C}^3$ cannot contain a transvection of $SL_8(q^\frac{1}{2})$. This contradicts the fact that $H$ is included in $G$ because $H$ contains the transvections of $SU_2(\tilde{q}^\frac{1}{2})$. This shows that the only possibility left is $S^{-1}GS=SL_8(q^\frac{1}{2})$.
 
We now consider cases $5$ and $6$ where our representation is now unitary by Proposition  \ref{isomorphisme}. In both cases, there exists a matrix $P$ such that $PGP^{-1}\subset SU_8(q^\frac{1}{2})$ with $G=R_{([2,1],[1])}(\mathcal{A}_{B_4})$ and we have $H=R_{([2,1],[1])}(\mathcal{A}_{B_3})\simeq SL_3(q)\times SU_2(\tilde{q})$ (resp $SL_2(\tilde{q}))$ in case $5$ (resp case $6$). We again have that $G$ is a primitive tensor-indecomposable group preserving no symmetric or skew-symmetric form over $\F_{q^\frac{1}{2}}$. This implies that $G$ is included in no maximal subgroup of type $\mathcal{C}^1$ or $\mathcal{C}^2$. Looking at the tables $8.46$ and $8.47$ of \cite{BHRC} and using the information above and cardinality arguments, we get that $PGP^{-1}$ must be equal to $SU_8(q^\frac{1}{2})$. This concludes the proof of the lemma.
\end{proof}
 
We must now show that we can use Theorem \ref{CGFS}. The factorizations of $\lambda=(\lambda_1,\lambda_2)$ by $(\lambda_1',\lambda_2')$ and by $(\lambda_2,\lambda_1)$ change the arguments for the natural $SL_2(q)$ and twisted diagonal embeddings of $SL_3(q)$. Let $\lambda = (\lambda_1,\lambda_2)$ be a double-partition of $n\geq 5$.
 
We then have five different cases.
 
\begin{enumerate}

 \item $\lambda\neq \lambda', \lambda\neq (\lambda_2,\lambda_1)$ and $\lambda\neq (\lambda_1',\lambda_2')$. Let us show $R_{\lambda}(\mathcal{A}_{B_n})$ contains a natural $SL_3(q)$. It is sufficient to show there exists $\mu \subset \lambda$ such that $\mu'\not\subset \lambda, (\mu_2,\mu_1) \not\subset \lambda$ and $(\mu_1',\mu_2') \not\subset \lambda$.
 
 We write $\lambda_1$ partition of $n_1$ and $\lambda_2$ partition of $n_2$ with $n=n_1+n_2\geq 5$. We only consider double-partitions with no empty component. This implies that $n_1$ and $n_2$ are greater than or equal to $1$. Since the roles of $\lambda_1$ and $\lambda_2$ are symmetrical for this, we can assume without loss of generality $n_1\geq n_2$.
 \begin{enumerate}
 \item $n_2=1$, we then have that $\lambda_2=\lambda_2'$, so $\lambda_1\neq \lambda_1'$. There exists $\mu_1\subset \lambda_1$ such that $\mu_1'\not\subset \lambda_1$. We then have that $\mu=(\mu_1,\lambda_2)\subset \lambda$, but $\mu'\not\subset \lambda$ and $(\lambda_2,\mu_1) \not\subset \lambda$ because $n_1-1\geq 4>1$ and $(\mu_1',\lambda_2') \not\subset \lambda$, because $\mu_1'\not\subset \lambda_1$.
 
 \item $n_1 > n_2= 2$ and $\lambda_1\neq \lambda_1'$. We set $\mu=(\lambda_1,[1])$, we have $\mu'$ and $([1],\lambda_1)\not\subset \lambda$ because $n_1>n_2$ and $(\lambda_1',[1])\not\subset\lambda$ because $\lambda_1'\neq \lambda_1$.
 
 \item $n_1>n_2=2$ and $\lambda_1=\lambda_1'$. If for all $\mu_1\subset \lambda_1$, $\mu_1\subset \lambda_2$ or $\mu_1'\subset \lambda_2$ then $n_1=3$ and $\lambda_1=[2,1]$ which implies that either $([2],[1^2])\subset \lambda$ or $([1^2],[2])\subset \lambda$. By Proposition \ref{patate}, $R_{\lambda}(\mathcal{A}_{B_n})$ contains a natural $SL_3(q)$.
 
 \item $n_1>n_2\geq 3$ and $\lambda_2\neq \lambda_2'$. There exists $\mu_2\subset \lambda_2$ such that $\mu_2'\not\subset\lambda_2$. We then set $\mu=(\lambda_1,\mu_2)$. We have that $(\mu_2,\lambda_1) \not\subset \lambda, (\mu_2',\lambda_1') \not\subset \lambda$ because $n_1>n_2$ and $(\lambda_1',\mu_2') \not\subset \lambda$ because $\mu_2'\not\subset \lambda_2$.
 
 \item $n_1>n_2\geq 3$ and $\lambda_2=\lambda_2'$, so $\lambda_1\neq \lambda_1'$. We know that there exists $\mu_1 \subset \lambda_1$ such that $\mu_1'\not\subset \lambda_1$. If $(\lambda_2,\mu_1)\subset \lambda$ or $(\lambda_2',\mu_1') \subset \lambda$ then $\mu_1=\lambda_2$ or $\mu_1'=\lambda_2$. We have that $\lambda_2=\lambda_2'$ so this contradicts $\mu_1'\not\subset \lambda_1$. This shows that $\mu_1\neq \mu_1'$. We have that $(\mu_1',\lambda_2') \not\subset \lambda$ because $\mu_1'\not\subset \lambda_1$.
 
 \item  $n_1=n_2\geq 3$. We then have that $\lambda_1\neq \lambda_1'$ or $\lambda_2\neq \lambda_2'$. If $\lambda_1\neq \lambda_1'$, we pick $\mu_1\subset \lambda_1$ such that $\mu_1'\not\subset \lambda_1$ and set $\mu=(\mu_1,\lambda_2)$, by the assumption on $\mu_1$,  $(\mu_1',\lambda_2)\not\subset \lambda, (\lambda_2,\mu_1)\not\subset \lambda$ and $(\lambda_2',\mu_1')\not\subset \lambda$ because $\lambda_2\neq \lambda_1$ and $\lambda_2'\neq \lambda_1$. If $\lambda_2\neq \lambda_2'$, we pick $\mu_2\subset \lambda_2$ such that $\mu_2'\not\subset \lambda_2$ and $\mu=(\lambda_1,\mu_2)$ verifies the required property.
 
 \end{enumerate} 
 
 \bigskip
 
 \item  $\lambda=(\lambda_1',\lambda_2'), \lambda\neq (\lambda_2,\lambda_1)$ and $\lambda\neq \lambda'$. We then have that $\mu_1'\subset \lambda$ for all $\mu_1\subset \lambda_1$ and that for all $\mu_2\subset \lambda_2$, $\mu_2'\subset \lambda_2$. We also have that $n_1+n_2\geq 5$.

\begin{enumerate}

\item $n_1\geq n_2=1$. Let $\mu_1\subset \lambda_1$, we set $\mu= (\mu_1, \lambda_2)$. We have that  $(\lambda_2,\mu_1)\not\subset \lambda$ and $\mu'\not\subset \lambda$ because $n_1-1\geq 3>1$ but $(\mu_1',\lambda_2') \subset \lambda$.
 
 \item $n_1>n_2\geq 2$. We pick $\mu_2\subset \lambda_2$ and set $\mu=(\lambda_1,\mu_2)$. We have that $(\mu_2,\lambda_1)\not\subset \lambda$ and $\mu'\not\subset \lambda$ because $n_1>n_2$ but $(\lambda_1',\mu_2') \subset \lambda$.
 
 \item $n_1=n_2\geq 2$. We pick $\mu_1\subset \lambda_1$ and set $\mu=(\mu_1,\lambda_2)$. We have that $(\lambda_2,\mu_1)\not\subset \lambda$ because  $\lambda_2\neq \lambda_1$ and $\mu'\not\subset \lambda$ because $\lambda_2'\neq \lambda_1$ but $(\mu_1',\mu_2')\subset \lambda$.
 
  In case $4$ for the fields, i.e., $\F_q=\F_p(\alpha,\beta)=\F_p(\alpha,\beta+\beta^{-1})\neq \F_p(\alpha+\alpha^{-1},\beta)$,  if $\mu\neq (\mu_1', \mu_2')$ then $R_{\lambda}(\mathcal{A}_{B_n})$ contains up to conjugation $\{\diag(
 M,\epsilon(M),I_{n_\lambda-6}), M\in SL_3(q)\}$, and a natural $SL_3(q^{\frac{1}{2}})$ if $\mu=(\mu_1',\mu_2')$ (it is possible this is the case for all $\mu\subset \lambda$ if we have square partitions).
 
 In cases $5$ and $6$ for the fields, i.e., $\F_q=\F_p(\alpha,\beta)=\F_p(\alpha+\alpha^{-1},\beta)\neq \F_p(\alpha,\beta+\beta^{-1})$,  if $\mu\neq (\mu_1', \mu_2')$ then $R_{\lambda}(\mathcal{A}_{B_n})$ contains up to conjugation $\{\diag(
 M,{}^t\!\epsilon(M^{-1}),I_{n_\lambda-6}), M\in SL_3(q)\}$, and a natural $SU_3(q^{\frac{1}{2}})$ if $\mu=(\mu_1',\mu_2')$.
 
 \end{enumerate}

 \item $\lambda=(\lambda_2,\lambda_1)\neq \lambda'$. We then have $n_1=n_2\geq 3$. $\lambda_1=\lambda_2\neq \lambda_1'$ because $\lambda\neq \lambda'$. We can then pick $\mu_1\subset \lambda_1$ such that $\mu_1'\not\subset \lambda_1'$ and $\mu=(\mu_1,\lambda_2)$, we have $(\mu_1',\lambda_2')\not\subset \lambda$ and $\mu'\not\subset \lambda$ because $\lambda_2'\neq \lambda_1=\lambda_2$ and $(\lambda_2,\mu_1) \subset \lambda$ but $\mu\neq (\lambda_2,\mu_1)$ because $n_1-1<n_1=n_2$.
 
 In case $4$ for the fields, $R_{\lambda}(\mathcal{A}_{B_n})$ contains up to conjugation
  $\{\diag(M,{}^t\!\epsilon(M^{-1}),I_{n_\lambda-6}),M\in SL_3(q)\}$.
                          
  In cases $5$ and $6$ for the fields, $R_{\lambda}(\mathcal{A}_{B_n})$ contains up to conjugation $\{\diag(M,\epsilon(M),I_{n_\lambda-6}), M\in SL_3(q)\}$.

\item $\lambda=\lambda'\neq (\lambda_1',\lambda_2')=(\lambda_2,\lambda_1)$, we have $n_1=n_2\geq 3$ and there exists $\mu_1\subset \lambda_1$ such that $\mu_1\not\subset \lambda_2$ because $\lambda_1\neq \lambda_2$. We have $\mu'\subset \lambda$ because $\mu_1'\subset\lambda_1'=\lambda_2$, $(\lambda_2,\mu_1) \not\subset \lambda$ car $\lambda_2\neq \lambda_1$ and $(\mu_1',\lambda_2') \not\subset \lambda$ because $\lambda_2'\neq \lambda_2=\lambda_1'$. We have $\mu\neq \mu'$ because $\lambda_2'\neq \mu_1$. $R_\lambda(\mathcal{A}_{B_n})$ contains up to conjugation 
 $\{\diag(M ,{}^t\!(M^{-1}),I_{n_\lambda-6}), M\in SL_3(q)\}$.

\item  $\lambda =\lambda'=(\lambda_2,\lambda_1)=(\lambda_1',\lambda_2')$. We then have $n_1=n_2\geq 3$. If $\lambda_1$ and $\lambda_2$ are square partitions, then for all $\mu \subset \lambda$, we have that $\mu=(\mu_1',\mu_2')\neq \mu'= (\mu_2,\mu_1)$, because $n_1=n_2> n_1-1=n_2-1$.
  
  In case $4$ for the fields, $R_\lambda(\mathcal{A}_{B_n})$ contains up to conjugation
  $\{\diag(M ,{}^t\!(M^{-1}), I_{n_\lambda-6} ), M\in SU_3(q^{\frac{1}{2}})\}$.
    
    In cases $5$ and $6$ for the fields, $R_\lambda(\mathcal{A}_{B_n})$ contains up to conjugation   $\{\diag(M,{}^t\!(M^{-1}),I_{n_\lambda-6}), M\in SL_3(q^{\frac{1}{2}})\}$.
    
    If $\lambda_1$ or $\lambda_2$ is a square partition, then there exists $\mu\subset \lambda$ such that $\mu\neq \mu'$, $\mu\neq (\mu_2,\mu_1)$ and $\mu\neq (\mu_1',\mu_2')$. This implies that $R_\lambda(\mathcal{A}_{B_n})$ contains up to conjugation $\{\diag(M,{}^t\!(M^{-1}),\epsilon(M),\epsilon({}^t\!(M^{-1})),I_{n_\lambda-12}), M\in SL_3(q)\}$.
      
\end{enumerate}

We use the notations in Theorem \ref{CGFS}. In all of the above cases except for the last one, there exists $g$ in $R_{\lambda}(\mathcal{A}_{B_n})$ such that $[g,V]=2$. This implies that $v_G(V)=2$ and $v_G(V) \leq \max(2,\frac{\sqrt{d}}{2})$. In the last case, we have in the same way an element $g$ such that $[g,V]=4$. We also have in that case that $\lambda=\lambda'=(\lambda_1,\lambda_2)$ and $n\geq 6$,.This implies that $\lambda$ contains $([2,1],[2,1])$, which is of dimension $\binom {6} {3}\times 2\times 2=80$. It follows that $d\geq 80$ and $\frac{\sqrt{d}}{2}\geq \frac{\sqrt{80}}{2}> 4$. This shows that we still have $v_G(V) \leq \max(2,\frac{\sqrt{d}}{2})$.

\bigskip

It remains to check that all the assumptions of the theorem are again verified and the classical group we get is the one we want.

The first step is to take care separately of double-partitions $\lambda$ such that $n_\lambda\leq 10$. If $n_\lambda=10$ then by the conditions in Theorem \ref{CGFS}, we can assume $p\neq 2$. The second step is to verify that the remaining double-partitions are tensor-indecomposable. The third step is to verify that they are imprimitive in the monomial case. The fourth step is to verify that they are imprimitive in the non-monomial case. The fifth step is to check that we are not in 2. of Theorem \ref{CGFS}. The sixth and last step is to verify that we have the desired classical groups in each of the above cases.

\bigskip

First step. For $n=5$, it is enough to consider $([2,2],[1])$, $([2,1,1],[1])$, $([2,1],[2])$ and $([1^3],[1^2])$, for which the respective $n_\lambda$ is $10$, $15$, $20$ and $10$. We must show that $R_{([1^3],[1^2])}(\mathcal{A}_{B_5})=SL_{10}(q)$ and $R_{([2,2],[1])}(\mathcal{A}_{B_5})=SL_{10}(q^\frac{1}{2})$ in case $4$ for the fields and $R_{([2,2],[1])}(\mathcal{A}_{B_5})=SU_{10}(q^\frac{1}{2})$ in cases $5$ and $6$ on the fields. The other double-partitions are of dimensions greater than $10$. We know that $G=R_{([1^3],[1^2])}(\mathcal{A}_{B_5})$ contains $R_{([1^3],[1^2])}(\mathcal{A}_{B_4}) =SL_4(q) \times SL_6(q^\frac{1}{2})$ and it is normally generated by this group, which is generated by transvections. Since $p\neq 2$, Theorem \ref{transvections} implies that $G$ is conjugated in $GL_{10}(q)$ to $SL_{10}(q'), SP_{10}(q')$ or $SU_{10}(q'^{\frac{1}{2}})$ for some $q'$ dividing $q$. Lemma \ref{field} implies that $q'=q$. The groups $SP_{10}(q)$ and $SU_{10}(q^\frac{1}{2})$ are excluded by Proposition \ref{isomorphisme}, because $R_{([1^3],[1^2])}$ is not isomorphic to its dual representation or its dual representation composed with the automorphism of order $2$ of $\F_q$. This shows that $G=SL_{10}(q)$. 
In case $4$, we know that $G=R_{([2,2],[1])}(\mathcal{A}_{B_5})$ is conjugated to a subgroup of $SL_{10}(q^{\frac{1}{2}})$ by Proposition \ref{isomorphisme} and Lemma \ref{Ngwenya} and that $G$ contains $R_{([2,2],[1])}(\mathcal{A}_{B_4}) = SL_8(q^\frac{1}{2})\times SL_2(\tilde{q})$. It follows that it contains a natural $SL_8(q^\frac{1}{2})$ and we can apply Theorem \ref{LBJ} to get that $G=SL_{10}(q^\frac{1}{2})$.  
In cases $5$ and $6$, $G=R_{([2,2],[1])}(\mathcal{A}_{B_5})$ is conjugated to a subgroup of $SU_{10}(q^\frac{1}{2})$ by Proposition \ref{isomorphisme} and Lemma \ref{Harinordoquy} and contains $R_{([2,2],[1])}(\mathcal{A}_{B_4})$, which contains a natural $SU_{8}(q^\frac{1}{2})$ in both cases. By Theorem $1.4$ of \cite{BM}, we have indeed $G=SU_{10}(q^\frac{1}{2})$.

\bigskip

Second step. We now show that those representations are tensor-indecomposable. Since $([2,1,1],[1])$ contains a natural $SL_3(q)$, doubles-partitions with at most two rows or at most two columns are tensor-indecomposable by Lemmas \ref{tens1} and \ref{tens2}. By the enumeration of the different cases, those lemmas cover all double-partitions of $n$ except if $\lambda=(\lambda_1,\lambda_2)=\lambda'=(\lambda_2,\lambda_1)=(\lambda_1',\lambda_2')$ and neither $\lambda_1$ nor $\lambda_2$ contains a sub-partition $\mu$ such that $\mu=\mu'$. In such a case,  $\lambda$ contains $([2,1],[2,1])$ which is of dimension $80$ and we can use the following lemma.

\begin{lemme}
If $d\geq 80$ and $G\subset GL_d(q)$ contains an element of order coprime to $p$ and conjugated in $GL_d(q)$ to the diagonal matrix $(\xi,\xi,\xi,\xi,\xi^{-1},\xi^{-1},\xi^{-1},\xi^{-1},1,...,1)$ with $\xi^2\neq 1$, then $G$ is tensor-indecomposable, except possibly if $G\subset G_1\otimes G_2$ with $G_1\subset GL_a(q)$, $a\in \{2,4\}$.
\end{lemme}

\begin{proof}

Let $g=P\diag(\xi I_4,\xi^{-1} I_4,I_{d-8})P^{-1}$. Assume that $g=g_1\otimes g_2$ with $g_1\in GL_a(\F_q)$, $g_2\in GL_b(\F_q)$ with $3\leq a\leq b$ and $ab=d$. We have that $b\geq \sqrt{d}$ so $b\geq 9$ because $d\geq 80$. We write $\lambda_1,...,\lambda_a$ the eigenvalues of $g_1$ and $\mu_1,...,\mu_b$ the eigenvalues of  $g_2$. We then have that $\forall i \in [\![ 1,a]\!]$, $\forall j\in [\![1,b]\!]$, $\lambda_i\mu_j\in \{1,\xi,\xi^{-1}\}$. The numbers $\xi$ and $\xi^{-1}$ only appear $4$ times each. This implies the number of couples $(\lambda_1\mu_i,\lambda_2\mu_i)\in \{(1,\xi),(\xi,1),(\xi,\xi^{-1})\}$ is less than or equal to $4$ as is the number of couples $(\lambda_1\mu_i,\lambda_2\mu_i)\in \{(1,\xi^{-1}),(\xi^{-1},1),(\xi^{-1},\xi)\}$. For any $i\in [\![1,a]\!]$, the inequality $\lambda_1\mu_i\neq \lambda_2\mu_i$ implies that $(\lambda_1\mu_i,\lambda_2\mu_i)\in \{(1,\xi),(\xi,1),(1,\xi^{-1}),(\xi^{-1},1),(\xi,\xi^{-1}),(\xi^{-1},\xi)\}$. It follows that there are at most $8$ couples $(\lambda_1\mu_i, \lambda_2\mu_i)$ such that $\lambda_1\mu_i\neq \lambda_2\mu_i$. Since $b\geq 9$, this implies that there exists $i\in [\![1,a]\!]$ such that $\lambda_1\mu_i=\lambda_2\mu_i$. It follows that $\lambda_1=\lambda_2$. In the same way, we have that $\lambda_1=\lambda_j$ for all $j \in [\![1,a]\!]$.  Up to reordering, we can assume $\lambda_1\mu_1=\xi$. We then have $\lambda_2\mu_1=\lambda_3\mu_1=\xi$. Since there are exactly $4$ ways $\xi$ appears as a $\lambda_i\mu_j$, we have that $a=4$.

 By the assumptions on $\lambda$, $H=R_{\lambda}(\mathcal{A}_{B_{n-1}})$ is a direct product of groups isomorphic to  some $SL_m(q)$ with $m\geq n_{([2,1],[2])}=20$. If $G=R_{\lambda}(\mathcal{A}_{B_n})$ is not tensor-indecomposable then $G\subset SL_2(q) \otimes SL_{\frac{d}{2}}(q)$ or $G\subset SL_4(q) \otimes SL_{\frac{d}{4}}(q)$. We then have a morphism from $G$ into $SL_2(q)$ or $SL_4(q)$. If we consider the restriction of this morphism to $H$, its kernel is a normal subgroup of $H$. The only non-abelian decomposition factors of $H$ are $PSL_m(q)$ with $m\geq 20$ so if the image is non-abelian, then there exists a subgroup of $SL_2(q)$ or a subgroup of $SL_4(q)$ isomorphic to some $PSL_m(q)$. This leads to a contradiction because $m\geq 20$. It follows that the image is abelian and since $H$ is perfect, the kernel is equal to $H$.  Since $H$ normally generates $G$, the morphism is trivial on $G$ which contradicts the irreducibility of $G$. 
 \end{proof}
 
 \bigskip
 
 Third step. In the monomial case, the only additional case to consider is the same one as in the second step. Looking at the corresponding proof in \cite[page 14]{BMM}, we get that $(p-1)r\leq 4$ with $q=p^r$. We know that $q$ is a square, $n\geq 6$, $\alpha$ is of order greater than $n$ and $\epsilon(\alpha)\in \{\alpha,\alpha^{-1}\}$ so $\alpha^{q^\frac{1}{2}-1}=1$ or $\alpha^{q^\frac{1}{2}+1}=1$. In both cases $q^\frac{1}{2}+1> 6$, and so $q^\frac{1}{2}\geq 6$ and $q\geq 36$. The condition $(p-1)r\leq 4$ implies that $q\leq max(5^1,4^1,3^2,2^4)= 16$ so we have a contradiction.
 
 \bigskip
 
 Fourth step. We know that there exists a matrix $t$ of order $p$ such as the one in \cite[page 14]{BMM} or with Jordan form $\diag(I_2+E_{1,2},I_2+E_{1,2},I_2+E_{1,2},I_2+E_{1,2},I_{n_\lambda-8})$.
 
 If $p\neq 2$, we can use the same arguments as in page $15$ of \cite{BMM} because we still have $(t-1)^2=0$.
 
 Assume now that $p=2$. Assume that $G\subset H \wr \mathfrak{S}_m = (H_1\times H_2 \times \dots \times H_m) \rtimes \mathfrak{S}_m$ with $H_1, \dots H_m$ the $m$-copies of $GL_{N/m}(q)$ permuted by $\mathfrak{S}_m$, that $V=U_1\oplus U_2 \oplus \dots \oplus U_m$ is the direct sum corresponding to the wreath product and that $t\notin H_1\times \dots \times H_m$.  Assume $t\notin H_1\times \dots \times H_m$. Up to reordering, we can assume $tU_1=U_2$.  If $\dim(U_i)\geq 5$ then we can consider linearly independent vectors $v_1,v_2,v_3,v_4,v_5$ in $U_1$ and by completing the family of vectors $(v_1,tv_1,v_2,tv_2,v_3,tv_3,v_4,tv_4,v_5,tv_5)$ which are linearly independent because $tU_1=U_2\neq U_1$, we get a basis upon which $t$ acts as a matrix of the form $M_2\oplus M_2\oplus M_2 \oplus M_2 \oplus M_2 \oplus X$ for a certain $X$ with $M_2=\begin{pmatrix}
 0 & 1\\
 1& 0
 \end{pmatrix}$. This implies that the rank of $t-1$ is greater than or equal than $5$ which is a contradiction.
 
 We can thus assume that $\dim(U_i) \leq 4$. Note that $G=R_\lambda(\mathcal{A}_{B_n})$ and $\mathcal{A}_{B_n}$ is perfect for $n\geq 5$ \cite{MR} so $G$ is perfect. If $G\subset (H_1\times H_2\times \dots \times H_m) \rtimes \mathfrak{S}_m$, we get $G\subset (H_1\times H_2 \times \dots \times H_m) \rtimes \mathfrak{A}_m$ because $[\mathfrak{S}_m,\mathfrak{S}_m]\subset \mathfrak{A}_m$. 
 
  If $t$ is a transvection then by the same reasoning as above on the dimensions of $U_i$, we are in the monomial case which was done in the third step. 
 
  If $t$ is of rank $2$, then either we are in the monomial case or $\dim(U_i)=2$. The monomial case is done, so it is sufficient to prove that $\dim(U_i)=2$ leads to a contradiction. We take $t_1$ and $t_2$ two such elements of rank $2$. Assume $\dim(U_i)=2$, since we have $t(U_1)=U_2$ and $t_1(U_2)=t_1^2(U_1)=U_1$. If $(u_a,u_b)$ are linearly independent then $(t_1u_a-u_a,t_1u_b-u_b)$ is a basis of $\mathrm{Im}(t_1-1)$,  which is of dimension $2$ and included in $U_1\oplus U_2$ for all $i\notin\{1,2\}$, $t_i(U_i)=U_i$. It follows that the projection of $t_1$ upon $\mathfrak{S}_m$ from the semi-direct product is a transposition. This is a contradiction because the projection of $G$ upon $\mathfrak{S}_m$ is included in $\mathfrak{A}_m$.
 
 If $t$ is of rank $4$ and $R_\lambda(\mathcal{A}_ {B_{n-1}})$ does not contain in an obvious way any transvections or elements $t$ of rank $2$, then $G$ contains up to conjugation $\{\diag(M,{}^t\!(M^{-1}),\epsilon(M),{}^t\!\epsilon(M^{-1}),I_{n_\lambda-8}), M\in SL_2(q)\}$.
We consider two elements $t_1$ and $t_2$ of rank $4$. If $\dim(U_1)=4$ then if $(u_1,u_2,u_3,u_4)$ is a basis of $U_1$, $(u_1-t_1u_1,u_2-t_1u_2,u_3-t_1u_3,u_4-t_1u_4)$ is a basis of $\mathrm{Im}(t_1-1)$, which is of dimension $4$. It follows that the projection of $t_1$ upon $S_m$ is a transposition, which is absurd.

 If $\dim(U_1)=3$, then if $(u_1,u_2,u_3)$ is a basis of $U_1$, we have that $Vect\{t_1u_1-u_1,t_1u_2-u_2,t_1u_3-u_3\} \subset \mathrm{Im}(t_1-1)$. If there exists $i\notin \{1,2\}$ such that $t_i(U_i)\neq U_i$ then in the same way as before, there would exist a subspace of dimension $6$ of $\mathrm{Im}(t_1-1)$, which is of dimension $4$. This shows that the projection of $t_1$ upon $\mathfrak{S}_m$ is a transposition, which is absurd.
 
 If $\dim(U_i)=2$, then we can take $4$ distinct non-zero elements $a_1,a_2,a_3,a_4$ of $\F_q$, which is possible because $q^{\frac{1}{2}}\geq 6$. We know that $G$ contains up to conjugation the elements $t_j$ for $j\in \{1,2,3,4\}$ with $t_j=\diag(I_2+a_j E_{1,2},I_2+a_j E_{1,2},I_2+\epsilon(a_j) E_{1,2}, I_2+\epsilon(a_j)E_{1,2},I_{n_\lambda-8})$. We have that
$\mathrm{Im}(t_j-1)$ is independent of $j$. We also have that $t_1(U_1)=U_2$ and $t_1(U_2)=t_1^2(U_1)=U_1$. Since $\mathrm{Im}(t_1-1)\cap U_1\oplus U_2$ is then of dimension $2$ and the projection of $t_1$ upon $\mathfrak{S}_m$ is not a transposition, there exists $i\notin\{1,2\}$ such that $t_1(U_i)\neq U_i$. Up to reordering, we can assume $t_1(U_3)=U_4$ and $t_1(U_4)=t_1^2(U_3)=U_3$. This shows that for all $j\in \{1,2,3,4\}, \mathrm{Im}(t_j-1)=\mathrm{Im}(t_1-1)\subset U_1\oplus U_2 \oplus U_3 \oplus U_4$. Since each $t_j$ is of order $2$, it follows writing $\pi$ the projection of $G$ upon $\mathfrak{S}_m$ that we have $\{\pi(t_1),\pi(t_2),\pi(t_3),\pi(t_4)\}\subset \{I_d,(12)(34),(13)(24),(14)(23)\}$.

Let us show that $\pi(t_j)=I_d$ for all $j$. They are all conjugated in $G$. Since $H_1\times H_2\times \dots \times H_m$ is a normal subgroup of $(H_1\times H_2\times \dots \times H_m)\rtimes \mathfrak{S}_m$, it is sufficient to show it for one of them. Assume it is wrong for all of them. We then have $\{\pi(t_1),\pi(t_2),\pi(t_3),\pi(t_4)\}\subset \{(12)(34),(13)(24),(14)(23)\}$. Therefore, there exists a pair $(i,j), i\neq j$ such that $\pi(t_i)=\pi(t_j)$ and so $\pi(t_it_j)=I_d$. But the matrix of $t_it_j$ in the basis we chose is $\diag(I_2+(a_i+a_j) E_{1,2},I_2+(a_i+a_j) E_{1,2},I_2+\epsilon(a_i+a_j) E_{1,2},I_2+\epsilon(a_i+a_j)E_{1,2},I_{n_\lambda-8})$.
We have $a_i+a_j\neq 0$ because $p=2$ and the elements $a_l$ are pairwise distinct. It follows that $t_it_j$ is conjugated each $t_l$ so we have a contradiction. This shows that for all $j\in \{1,4\}, \pi(t_j)=I_d$. It follows that $t_j\in H_1\times H_2\times \dots \times H_m$ which is normal in $(H_1\times H_2\times \dots \times H_m)\rtimes \mathfrak{S}_m$. Since $G$ is normally generated by $R_\lambda(\mathcal{A}_{B_{n-1}})$, which is normally generated by elements of the form $t_j$, we have that $G\subset H_1\times H_2\times \dots \times H_m$, which contradicts the irreducibility of $G$. This is absurd and it follows that $G$ is an imprimitive group.
 
 \bigskip
 
Fifth step. If $G$ contains a natural $SL_2(q^\frac{1}{2})$ or a natural $SU_2(q^\frac{1}{2})$ then we can apply the same arguments as in \cite[page 13]{BMM}. If $G$ contains a twisted diagonal embedding or a twisted diagonal embedding composed with the automorphism of order $2$ of $\F_q$ of $SL_3(q)$, then we can apply the arguments of \cite[page 14]{BMM}. If we are not in any of the above cases, then $\lambda=\lambda'=(\lambda_2,\lambda_1)$, so $n\geq 6$ and we are in one of the following cases.
 \begin{enumerate}
 \item $R_\lambda(\mathcal{A}_{B_n})$ contains up to conjugation   $\{\diag(M,{}^t\!(M^{-1}),I_{n_\lambda-6} ), M\in SU_3(q^{\frac{1}{2}})\}$.
 \item $R_\lambda(\mathcal{A}_{B_n})$ contains up to conjugation   $\{\diag(M,{}^t\!(M^{-1}),I_{n_\lambda-6}), M\in SL_3(q^{\frac{1}{2}})\}$.
 \item $R_\lambda(\mathcal{A}_{B_n})$ contains up to conjugation $\{\diag(M,{}^t\!(M^{-1}),\epsilon(M),{}^t\!\epsilon(M^{-1}),I_{n_\lambda-12}), M\in SL_3(q)\}$.
 \end{enumerate}
 
In the first two cases, we have an element $g$ conjugated to $\diag(\xi,\xi,\xi^{-1},\xi^{-1},1,\dots ,1)$ with $\xi$ of order $q^\frac{1}{2}-1$ but the order of $\alpha$ is less than or equal to $q^{\frac{1}{2}}+1$ in both cases. If $g$ is an element of $\mathfrak{S}_{n_\lambda}$ such that $[g,V]= 4$, then we have ($g=\sigma_1\sigma_2\sigma_3\sigma_4$ product of 4 disjoint transpositions and $g$ is of order $2$) or ($g$ is the product of 2 disjoint 3-cycles and $g$ is of order $3$) or ($g$ is a $5$-cycle and $g$ is of order $5$) or ($g$ is the disjoint product of  $2$ transpositions and a $3$-cycle and $g$ is of order $6$). Since $n_\lambda \geq 6$ and the order of $\alpha$ is greater than $n$, $q^\frac{1}{2}+1>7$ so $q^\frac{1}{2}-1>5$ which contradicts all the cases except for the last one. In the last case, we have that $n_\lambda\geq 7$ by the decomposition of $g$. Since $\lambda=\lambda'$, $n_\lambda$ is even and $q^\frac{1}{2}-1= q^\frac{1}{2}+1-2>n_\lambda-2> 6$, which contradicts the last case.

In the third case, we have an element $g$ conjugated to $\diag(\xi,\xi,\xi,\xi,\xi^{-1},\xi^{-1},\xi^{-1},\xi^{-1},1,\dots,1)$ which is of order $o(g)=q-1$ but $q^\frac{1}{2}+1>7$ so $q> 36$. Since $q$ is an even power of a prime number, it follows that $q>49$ and $q-1\geq 49$. We have $[g,V]= 8$ so by considering the decomposition into disjoint cycles of $g$ and using the fact that the rank of $\sigma-1$ of a cycle $\sigma$ is equal to the length of the cycle minus 1, we get $o(g)\in \{\lcm(\{n_i+1\}_{i\in I}), \underset{i\in I}\sum n_i=8, n_i\in \N^\star\}$. It follows that $o(g) \leq 30 <49\leq q-1=o(g)$ which is a contradiction. 

\bigskip

Sixth step. We have shown that $G=R_\lambda(\mathcal{A}_{B_n})$ is a classical group in a natural representation, the last step is to show that we have the following theorem.
\begin{theo}
If $n\geq 5$, then for all double-partition $\lambda\vdash\vdash n$ in our decomposition, $R_\lambda(\mathcal{A}_{B_n})=G(\lambda)$, where $G(\lambda)$ is given by the following list.
\begin{enumerate}
\item When $\F_q=\F_p(\alpha,\beta)=\F_p(\alpha,\beta+\beta^{-1})\neq\F_p(\alpha+\alpha^{-1},\beta)$, and $\F_{\tilde{q}}=\F_p(\alpha)\neq \F_p(\alpha+\alpha^{-1})$
\begin{enumerate}
\item $SU_{n-1}(\tilde{q}^\frac{1}{2})$ if $\lambda =([n-1,1],\emptyset)$.
\item $SU_{n_\lambda}(\tilde{q}^\frac{1}{2})$ if $\lambda = (\lambda_1,\emptyset), \lambda_1< \lambda_1'$.
\item $SP_{n_\lambda}(\tilde{q}^\frac{1}{2})$ if $\lambda = (\lambda_1,\emptyset), \lambda_1=\lambda_1'$ and ( $p=2$ or ($p\geq 3$ and $\nu(\lambda_1) = -1$)).
\item $\Omega_N^{+}(\tilde{q}^\frac{1}{2})$ if $\lambda=(\lambda_1,\emptyset), \lambda_1=\lambda_1'$, $p\geq 3$ and $\nu(\lambda_1)=1$.
\item $SL_n(\F_q)$ if $\lambda=([1],[n-1])$.
\item  $SL_{n_\lambda}(\F_q)$ if $\lambda\neq \lambda',\lambda\neq(\lambda_1',\lambda_2')$ and $\lambda\neq (\lambda_2,\lambda_1)$.
\item $SU_{n_\lambda}(q^\frac{1}{2})$, if $\lambda=(\lambda_2,\lambda_1) \neq \lambda'$.
\item $SL_{n_\lambda}(\F_{q^\frac{1}{2}})$,if $\lambda=(\lambda_1',\lambda_2') \neq \lambda'$.
\item $SP_{n_\lambda}(q)$, if $\lambda=\lambda'\neq (\lambda_2,\lambda_1)$ and ($p=2$ or $\nu(\lambda)=-1$).
\item $\Omega^+_{n_\lambda}(q)$,if $\lambda=\lambda'\neq (\lambda_2,\lambda_1)$, $p\neq 2$ and $\nu(\lambda)=1$.
\item $SP_{n_\lambda}(q^\frac{1}{2})$, if $\lambda=\lambda'=(\lambda_2,\lambda_1)$ and ($p=2$ or $\nu(\lambda)=-1$).
\item $\Omega_{n_\lambda}^+(q^\frac{1}{2})$ if $\lambda=\lambda'= (\lambda_2,\lambda_1)$, $p\neq 2$ and $\nu(\lambda)=1$.
\end{enumerate}
\item When $\F_q=\F_p(\alpha,\beta)=\F_p(\alpha+\alpha^{-1},\beta)\neq \F_p(\alpha,\beta+\beta^{-1})$,
\begin{enumerate}
\item when $\F_{\tilde{q}}=\F_p(\alpha)=\F_p(\alpha+\alpha^{-1})$,
\begin{enumerate}
\item $SL_{n-1}(\F_{\tilde{q}})$ if $\lambda =([n-1,1],\emptyset)$.
\item $SL_{n_\lambda}(\F_{\tilde{q}})$ if $\lambda = (\lambda_1,\emptyset), \lambda_1< \lambda_1'$.
\item $SP_{n_\lambda}(\tilde{q})$ if $\lambda = (\lambda_1,\emptyset), \lambda_1=\lambda_1'$ and ($p=2$ or ($p\geq 3$ and $\nu(\lambda_1) = -1$)).
\item $\Omega_N^{+}(\tilde{q})$ if $\lambda=(\lambda_1,\emptyset), \lambda_1=\lambda_1'$, $p\geq 3$ and $\nu(\lambda_1)=1$.
\end{enumerate}
\item when $\F_{\tilde{q}}=\F_p(\alpha)\neq \F_p(\alpha+\alpha^{-1})$,
\begin{enumerate}
\item $SU_{n-1}(\tilde{q}^\frac{1}{2})$ if $\lambda =([n-1,1],\emptyset)$.
\item $SU_{n_\lambda}(\tilde{q}^\frac{1}{2})$ if $\lambda = (\lambda_1,\emptyset), \lambda_1< \lambda_1'$.
\item $SP_{n_\lambda}(\tilde{q}^\frac{1}{2})$ if $\lambda = (\lambda_1,\emptyset), \lambda_1=\lambda_1'$ and ( $p=2$ or ($p\geq 3$ and $\nu(\lambda_1) = -1$)).
\item $\Omega_N^{+}(\tilde{q}^\frac{1}{2})$ if $\lambda=(\lambda_1,\emptyset), \lambda_1=\lambda_1'$, $p\geq 3$ and $\nu(\lambda_1)=1$.
\end{enumerate}
\item $SL_n(\F_q)$ if $\lambda=([1],[n-1])$,
\item $SL_{n_\lambda}(\F_q)$ if $\lambda\neq \lambda',\lambda\neq(\lambda_1',\lambda_2')$ and $\lambda\neq (\lambda_2,\lambda_1)$, 
\item $SL_{n_\lambda}(\F_{q^\frac{1}{2}})$ if $\lambda=(\lambda_2,\lambda_1) \neq \lambda'$,
\item  $SU_{n_\lambda}(q^\frac{1}{2})$ if $\lambda=(\lambda_1',\lambda_2') \neq \lambda'$,
\item  $SP_{n_\lambda}(q)$ if $\lambda=\lambda'\neq (\lambda_2,\lambda_1)$ and ($p=2$ or $\nu(\lambda)=-1$),
\item $\Omega^+_{n_\lambda}(q)$ if $\lambda=\lambda'\neq (\lambda_2,\lambda_1)$, $p\neq 2$ and $\nu(\lambda)=1$, 
\item  $SP_{n_\lambda}(q^\frac{1}{2})$ if $\lambda=\lambda'=(\lambda_2,\lambda_1)$ and ($p=2$ or $\nu(\lambda)=-1$),
\item  $\Omega_{n_\lambda}^+(q^\frac{1}{2})$ if $\lambda=\lambda'= (\lambda_2,\lambda_1)$, $p\neq 2$ and $\nu(\lambda)=1$.
\end{enumerate}
\end{enumerate}
\end{theo}

\begin{proof}

It is sufficient to prove the result for double-partitions with no empty components which are not hooks. We know by Theorem \ref{CGFS} and the previous steps that $G(\lambda)$ is a classical group in a natural representation. The proof uses Proposition \ref{isomorphisme} and the separation of the cases made before the enumeration of the six steps. We write $\F_{q'}$ the field over which our classical group is defined. In all cases $G(\lambda) \subset SL_n(q)$ so $q'$ divides $q$.

Assume $\F_q=\F_p(\alpha,\beta)=\F_p(\alpha,\beta+\beta^{-1})\neq \F_p(\alpha+\alpha^{-1},\beta).$
\begin{enumerate}
\item If $\lambda\neq  \lambda', \lambda\neq (\lambda_1',\lambda_2')$ and $\lambda\neq (\lambda_2,\lambda_1)$, then $G(\lambda)$ contains a natural $SL_2(q)$. By Lemma \ref{field}, we have that $q'=q$. By Proposition \ref{isomorphisme}, $G(\lambda)$ preserves no hermitian or bilinear form, so $G(\lambda)=SL_{n_\lambda}(q)$.
\item If $\lambda=(\lambda_2,\lambda_1)\neq \lambda'$, then by Proposition \ref{isomorphisme} and Lemma \ref{Harinordoquy}, we have up to conjugation $G(\lambda) \subset SU_{n_\lambda}(q^\frac{1}{2})$. Up to conjugation, $G(\lambda)$ contains $\{\diag(M ,{}^t\!\epsilon(M^{-1}),I_{n_\lambda-6}), M\in SL_3(q)\}$, so $G(\lambda)$ contains $\{\diag(M,M,I_{n_\lambda-6}),M\in SU_3(q^\frac{1}{2})\}$.
                          
Let $\varphi$ be the natural representation of $SU_3(q^\frac{1}{2})$ in $GL_3(\overline{\F_p})$ and $\rho$ the diagonal representation of $SU_3(q^\frac{1}{2})$ in $GL_{n_\lambda}(\overline{\F_p})$, given by the above subgroup of $G(\lambda)$. 

We have $\rho\simeq \varphi \oplus \varphi \oplus \mathbf{1}^{n_\lambda-6}$ with $\mathbf{1}$ the trivial representation. Let $\sigma$ be a generator of $\Gal(\F_q/\F_{q'})$. Since $G(\lambda)$ is a classical group over $\F_{q'}$, we have that $\rho\simeq \rho^\sigma$ so $\varphi \simeq \varphi^{\sigma}$. It follows that for every $M\in SU_3(q^\frac{1}{2})$, we have $\sigma(\tr(M))=\tr(M)$. By Lemma \ref{field}, we have that $\F_q=\F_q^{\Gal(\F_q/\F_{q'})}$ and so $q'=q$. By Proposition \ref{isomorphisme}, $G(\lambda)$ preserves no bilinear form so $G(\lambda)=SU_{n_\lambda}(q^\frac{1}{2})$.
\item If $\lambda=(\lambda_1',\lambda_2')\neq \lambda'$, then by Proposition \ref{isomorphisme} and Lemma \ref{Ngwenya}, up to conjugation, we have that $G(\lambda)\subset SL_{n_\lambda}(q^\frac{1}{2})$. The group $G(\lambda)$ contains either a natural $SL_3(q^\frac{1}{2})$ or a group of the form $\{\diag( M,\epsilon(M),I_{n_\lambda-6}), M\in SL_3(q)\}$.
 
 If $G(\lambda)$ contains a natural $SL_3(q^\frac{1}{2})$ then by Lemma \ref{field}, we have $q'=q^\frac{1}{2}$. We know by Proposition \ref{isomorphisme} that $G(\lambda)$ preserves no symmetric or skew-symmetric bilinear form . If we had $G(\lambda) \subset SU_{n_\lambda}(q^\frac{1}{4})$, then the natural $SL_3(q^\frac{1}{2})$ in $G(\lambda)$ would inject itself in some $SU_3(q^\frac{1}{4})$.  This is absurd because of their orders so we have $G(\lambda)=SL_{n_\lambda}(q^\frac{1}{2})$.
 
 If $G$ contains up to conjugation a group of the form  $\{\diag(M,\epsilon(M),I_{n_\lambda-6}), M\in SL_3(q)\}$ then it contains $\{\diag(M,M ,I_{n_\lambda-6}), M\in SL_3(q^{\frac{1}{2}})\}$.  Let $\varphi$ be the natural representation of $SL_3(q^\frac{1}{2})$ in $GL_3(\overline{\F_p})$ and $\rho$ the diagonal representation of $SL_3(q^\frac{1}{2})$ in $GL_{n_\lambda}(\overline{\F_p})$ given by the above subgroup of $G(\lambda)$. We then have $\rho \simeq \varphi \oplus \varphi \oplus \mathbf{1}^{n_\lambda-6}$. Let $\sigma$ be a generator of $\Gal(\F_{q^\frac{1}{2}}/\F_{q'})$. We have $\rho \simeq \rho^\sigma$ so $\varphi\simeq \varphi^\sigma$ and by Lemma \ref{field}, we have that $\F_{q^\frac{1}{2}}=\F_{q^\frac{1}{2}}^{\Gal(\F_{q^\frac{1}{2}}/\F_{q'})}$ so $q'=q^\frac{1}{2}$. We cannot have $G(\lambda)=SU_{n_\lambda}(q^\frac{1}{4})$ because $SL_3(q)$ would inject itself in $SU_6(q^\frac{1}{4})$ and we know that $\frac{\vert SU_{6}(q^\frac{1}{4})\vert}{q^\frac{15}{4}} < \frac{\vert SL_3(q)\vert}{q^3}$. By Proposition \ref{isomorphisme}, $G(\lambda)$ cannot preserve any symmetric or skew-symmetric bilinear form so $G(\lambda)=SL_{n_\lambda}(q^\frac{1}{2})$.
 
\item Case $4$ is analogous to Case $3$.
     
\item If $\lambda=\lambda'=(\lambda_2,\lambda_1)$, then by Proposition \ref{isomorphisme} and Proposition \ref{coolprop}, $G(\lambda)$ preserves a bilinear form of the type given by Proposition \ref{bilin} defined over $\F_{q^\frac{1}{2}}$. This shows that $q'\leq q^\frac{1}{2}$ and it is enough to show that $q'=q^\frac{1}{2}$ to conclude the proof.

If $\lambda_1$ and $\lambda_2$ are square partitions then $G(\lambda)$ contains up to conjugation the group\break $\{\diag(M,{}^t\!(M^{-1}),I_{n_\lambda-6}), M\in SU_3(q^{\frac{1}{2}})\}$.
    
Let $\varphi$ be the natural representation of $SU_3(q^\frac{1}{2})$ in $GL_3(\overline{\F_p})$, and $\rho$ the twisted diagonal representation of $SU_3(q^\frac{1}{2})$ in $GL_{n_\lambda}(\overline{\F_p})$ given by the above subgroup of $G(\lambda)$. We have $\rho\simeq \varphi \oplus \varphi^\star \oplus \mathbf{1}^{n_\lambda-6}$. Let $\sigma$ be a generator of $\Gal(\F_{q^\frac{1}{2}}/\F_{q'})$. Since $G(\lambda)$ is a classical group over $\F_{q'}$, we have $\rho\simeq \rho^\sigma$. It follows that $\varphi \simeq \varphi^{\sigma}$ or $\varphi\simeq (\varphi^\star)^\sigma$. The first possibility implies that $\F_{q^\frac{1}{2}}=\F_{q^\frac{1}{2}}^{\Gal(\F_{q^\frac{1}{2}}/\F_{q'})}$, so $q'=q^\frac{1}{2}$. The second possibility implies that $\varphi \simeq \varphi^{\sigma^2}$, so $q'=q^\frac{1}{2}$ or $\sigma$ is of order $2$ and $SU_3(q^\frac{1}{2})$ injects into $SU_3(q^\frac{1}{4})$, which is a contradiction. In both cases, we have $q'=q^\frac{1}{2}$ and the desired result follows.

If $\lambda_1$ or $\lambda_2$ is not a square partition, then $G(\lambda)$ contains up to conjugation the group \break$\{\diag(M, {}^t\!(M^{-1}),\epsilon(M),{}^t\!\epsilon(M^{-1}),I_{n_\lambda-12}), M\in SL_3(q)\}$, and so contains its subgroup \break$\{\diag(M,{}^t\!(M^{-1}),M ,{}^t\!(M^{-1}),I_{n_\lambda-12}), M\in SL_3(q^\frac{1}{2})\}$.  Let $\varphi$ be the natural representation of $SL_3(q^\frac{1}{2})$ in $GL_3(\overline{\F_p})$ and $\rho$ be the representation of $SL_3(q^\frac{1}{2})$ in $GL_{n_\lambda}(\overline{\F_p})$ given by the above subgroup $G(\lambda)$. We have $\rho\simeq \varphi \oplus \varphi \oplus \varphi^\star\oplus \varphi^\star \oplus \mathbf{1}^{n_\lambda-6}$ . Let $\sigma$ be a generator of $\Gal(\F_{q^\frac{1}{2}}/\F_{q'})$. Since $G(\lambda)$ is a classical group defined over $\F_{q'}$, we have that $\rho\simeq \rho^\sigma$. It follows that $\varphi \simeq \varphi^{\sigma}$ or $\varphi\simeq (\varphi^\star)^\sigma$. By the same arguments as before, we have $q'=q^\frac{1}{2}$ or $SL_3(q^\frac{1}{2})$ injects itself in $SU_3(q^\frac{1}{4})$, which is not possible. This proves $q'=q^\frac{1}{2}$ and concludes the case $\F_q=\F_p(\alpha,\beta)=\F_p(\alpha,\beta+\beta^{-1})\neq \F_p(\alpha+\alpha^{-1},\beta).$ 
\end{enumerate}
 
 If $\F_q=\F_p(\alpha,\beta)=\F_p(\alpha+\alpha^{-1},\beta)\neq \F_p(\alpha,\beta+\beta^{-1})$, then all the arguments are the same up to permutation of the different cases.
\end{proof}

\section{Type D}

In this section, we determine the image of the Artin group of type $D$ inside its asssociated finite Iwahori-Hecke algebra. In the first two subsections, in an analogous way to what we did in type $B$, we establish preliminary results which permit us to state Theorems \ref{result4} and \ref{result5} which are the main theorems of this section. In the last two subsections, we complete the proofs of these theorems.

\subsection{Definition of the model}

Let $n\geq 4$, $p$ a prime different from $2$, $\alpha\in \overline{\F_p}$ of order greater than $2n$. We set in this section $\F_q=\F_p(\alpha)$. As in \cite{G-P}, we take for the Iwahori-Hecke algebra of type $D$ $\mathcal{H}_{D_n,\alpha}$ the sub-algebra of $\mathcal{H}_{B_n,\alpha,1}$ generated by $U=TS_1T,S_1,\dots S_{n-1}$. In this case $\sigma: V_{\lambda_1,\lambda_2} \rightarrow V_{\lambda_2,\lambda_1}, (\T_1,\T_2)\mapsto (\T_2,\T_1)$ is an isomorphism of $\mathcal{H}_{D_n,\alpha}$-modules. It is a standard fact (see \cite{phD} for a proof) that under those conditions on $\alpha$, the simple modules of $\mathcal{H}_{D_n,\alpha}$ are indexed by double-partitions $(\lambda_1,\lambda_2), \lambda_1>\lambda_2$ and the modules $V_{\lambda,\lambda}$ split into two irreducible sub-modules $V_{\lambda,\lambda,+}$ and $V_{\lambda,\lambda,-}$ of the same dimension. The branching rule for type $D$ is more complicated so we recall it in the following proposition. (a proof in a more general setting can be found in \cite{Branch}):

\begin{prop}\label{branch}
Let $n\geq 5$ and $(\lambda,\mu)\vdash\!\vdash n, \lambda>\mu$. We then have:
\begin{enumerate}
\item If $n_\lambda > n_\mu+1$, then $V_{\lambda,\mu|\mathcal{H}_{D_{n-1},\alpha}}=\underset{(\tilde{\lambda},\tilde{\mu})\subset (\lambda,\mu)}\bigoplus V_{\tilde{\lambda},\tilde{\mu}}$.
\item If $n_\lambda=n_\mu+1$ and $\mu \not\subset \lambda$, then
$$V_{\lambda,\mu|\mathcal{H}_{D_{n-1},\alpha}}=(\underset{\tilde{\mu}\subset \mu}\bigoplus V_{\lambda,\tilde{\mu}}) \oplus (\underset{\tilde{\lambda}> \mu}{\underset{\tilde{\lambda}\subset \lambda}\bigoplus}V_{\tilde{\lambda},\mu})\oplus( \underset{\tilde{\lambda}< \mu}{\underset{\tilde{\lambda}\subset \lambda}\bigoplus}V_{\mu,\tilde{\lambda}}).$$
\item If $n_\lambda=n_\mu+1$ and $\mu \subset \lambda$, then
$$V_{\lambda,\mu|\mathcal{H}_{D_{n-1},\alpha}}=(\underset{\tilde{\mu}\subset \mu}\bigoplus  V_{\lambda,\tilde{\mu}}) \oplus (\underset{\tilde{\lambda}> \mu}{\underset{\tilde{\lambda}\subset \lambda}\bigoplus}V_{\tilde{\lambda},\mu})\oplus (\underset{\tilde{\lambda}< \mu}{\underset{\tilde{\lambda}\subset \lambda}\bigoplus}V_{\mu,\tilde{\lambda}})\oplus V_{\mu,\mu,+}\oplus V_{\mu,\mu,-}.$$
\item If $n_\lambda=n_\mu$ and $\lambda>\mu$, then $V_{\lambda,\mu|\mathcal{H}_{D_{n-1},\alpha}}=(\underset{\tilde{\mu}\subset \mu}\bigoplus V_{\lambda,\tilde{\mu}})\oplus (\underset{\tilde{\lambda}\subset \lambda}\bigoplus V_{\mu,\tilde{\lambda}}).$
\item If $\lambda=\mu$, then $V_{\lambda,\lambda,+|\mathcal{H}_{D_{n-1},\alpha}}=V_{\lambda,\lambda,-|\mathcal{H}_{D_{n-1},\alpha}}=\underset{\tilde{\mu} \subset \mu}\bigoplus V_{\lambda,\tilde{\mu}}.$
\end{enumerate}
\end{prop}

We keep the same weight on double-tableaux as for type B. Let $\lambda=(\lambda_1,\lambda_2)\vdash\!\vdash n$ and $\T=(\T_1,\T_2)\in \lambda$. We define $\varphi(\T)$ to be $\T'$ if $\mu'>\lambda'$ and $\sigma(\T')$ otherwise. We define a new $\tilde{\nu}(\lambda)$ to be $\nu(\lambda_1)\nu(\lambda_2)(-1)^{n_{\lambda_1}(n-n_{\lambda_1})}$ if $\lambda_2' \geq \lambda_1'$ and $\tilde{\nu}(\lambda)=\nu(\lambda_1)\nu(\lambda_2)$ otherwise. We define the bilinear form $(\T|\tilde{\T})=\omega(\T)\delta_{\varphi(\T),\tilde{\T}}$.

\begin{prop}\label{bilin2}
For any pair of standard double-tableaux $(\T, \tilde{\T})$, we have the following properties.
\begin{enumerate}
\item $(S_i.\T|S_i.\tilde{\T}) = (-\alpha)(\T|\tilde{\T})$ and $(U.\T|U.\tilde{\T}) = (-\alpha)(\T|\tilde{\T})$.

\item
For all $d\in \mathcal{D}_n$, we have that $(d.\T|d.\tilde{\T}) = (\T|\tilde{\T})$.

Those relations stay true if we substitute one or two of the standard double-tableaux by the elements $\sigma(\T)-\T$ and $\sigma(\T)+\T$, which form bases for $V_{\lambda,+}$ and $V_{\lambda,-}$ for double-partitions $\lambda$ of the form $\lambda=(\lambda_1,\lambda_1)$.
\item
The restriction of $(.,.)$ to $V_\lambda$ if $\lambda =\varphi(\lambda)\neq (\lambda_2,\lambda_1)$ and to $V_\lambda\oplus V_{\varphi(\lambda)}$ if $\lambda\notin \{\varphi(\lambda),(\lambda_2,\lambda_1)\}$ is non-degenerate with $\varphi(\lambda)=\lambda'$ if $n_\lambda=n_\mu$ and $\mu'> \lambda'$, and $\varphi(\lambda)=(\lambda_1',\lambda_2')$ otherwise.

If $\lambda=\varphi(\lambda)\neq (\lambda_2,\lambda_1)$ then $(.,.)$ is symmetric on $V_\lambda$ if $\tilde{\nu}(\lambda)=1$ and skew-symmetric otherwise.

Moreover, its Witt index is positive.
\item If $n\equiv 0 ~(\bmod 4)$ and $\lambda=(\lambda_1,\lambda_1)$, then the restriction of $(.,.)$ to $V_{\lambda,+}$ and $V_{\lambda,-}$ if $\lambda=\lambda'$ and to $V_{\lambda,+}\oplus V_{\lambda,-}$ if $\lambda\neq \lambda'$, is non-degenerate. 

If $\lambda=\lambda'$ then $(.,.)$ is symmetric on $V_{\lambda,+}$ and $V_{\lambda,-}$ if $\tilde{\mu}(\lambda)=1$ and skew-symmetric otherwise. 
Moreover, its Witt index is positive.

\item If $n\equiv 2 ~(\bmod 4)$ and $\lambda=(\lambda_1,\lambda_1)$ then the restriction of $(.,.)$ to $V_{\lambda,+}\oplus V_{\lambda',-}$ is non-degenerate.
\end{enumerate}
\end{prop}

\begin{proof}

For $1.$ and $2.$, the proof is exactly the same as for Proposition \ref{bilin} by noting $m_i(\sigma(\T)) =m_i(\T)$. The extension to elements of the bases of $V_{\lambda,+}$ and $V_{\lambda,-}$ follows from the bilinearity of $(.,.)$.

For $3.$, the proof also remains the same because $\tilde{\nu}(\lambda)=\omega(\T)\omega(\varphi(\T))$. This is true because when $\varphi(\T)=\T'$, $\tilde{\nu}(\lambda)$ does not change from the one in type $B$ and when $\varphi(\T)=\sigma(\T')$, $\tilde{\nu}(\lambda)$ is multiplied by $(-1)^{n_{\lambda_1}(n-n_{\lambda_1})}=\omega(\T)\omega(\sigma(\T))$.

$4.$ We assume $n \equiv 0 ~(\bmod 4)$. If $\lambda=(\lambda_1,\lambda_1)\vdash\!\vdash n$ and $\T\in \lambda$ then $\omega(\sigma(\T))=(-1)^{n_{\lambda_1}(n-n_{\lambda_1)})}\omega(\T)=(-1)^{(\frac{n}{2})^2}\omega(\T)=\omega(\T).$ 

For any standard double-tableaux $\T,\tilde{\T}$, we have that $(\T|\tilde{\T})=\omega(\T)\delta_{\T,\varphi(\T)}$. Since $\lambda=(\lambda_1,\lambda_1)$, we have $\varphi(\lambda)=\lambda'$ and for all $\T\in \lambda$ and all $\T\in \lambda'$, we have $\varphi(\T)=\T'$.

Let $\lambda=(\lambda_1,\lambda_1)$ and $\tilde{\lambda}=(\tilde{\lambda}_1,\tilde{\lambda}_1)$ be double-partitions of $n$. If $\T\in \lambda$ and $\tilde{\T}\in \tilde{\lambda}$ then we have :

\begin{eqnarray*}
(\T+\sigma(\T)|\tilde{\T}+\sigma(\tilde{\T})) & = & (\T|\tilde{\T}) + (\T|\sigma(\tilde{\T}))+ (\sigma(\T)|\tilde{\T})) + (\sigma(\T)|\sigma(\tilde{\T}))  \\
 & = & \omega(\T)(\delta_{\T,\tilde{\T}'}+\delta_{\T,\sigma(\tilde{\T})'})+\omega(\sigma(\T))(\delta_{\sigma(\T),\tilde{\T}'}+\delta_{\sigma(\T),\sigma(\tilde{\T})'})\\
 & = & (\delta_{\T,\tilde{\T}'}+\delta_{\T,\sigma(\tilde{\T})'})(\omega(\T)+\omega(\sigma(\T))\\
 & = & \delta_{\T+\sigma(\T),\tilde{\T}'+\sigma(\tilde{\T}')}(\omega(\T)+\omega(\sigma(\T)))\\
 & = & 2\omega(\T)\delta_{\T'+\sigma(\T)',\tilde{\T}+\sigma(\tilde{\T})}.
\end{eqnarray*}
In the same way, we have that $(\T+\sigma(\T)|\tilde{\T}-\sigma(\tilde{\T}))=(\T-\sigma(\T)|\tilde{\T}+\sigma(\tilde{\T}))=0$ and\\ $(\T-\sigma(\T)|\tilde{\T}-\sigma(\tilde{\T}))=2\omega(\T)\delta_{\T'-\sigma(\T)',\tilde{\T}-\sigma(\tilde{\T})}$. The result follows.

$5.$ Assume $n\equiv 2 ~(\bmod ~ 4)$. If $\lambda=(\lambda_1,\lambda_1)\vdash\!\vdash n$ and $\T\in \lambda$, then $\omega(\sigma(\T))=(-1)^{n_{\lambda_1}(n-n_{\lambda_1)})}\omega(\T)=(-1)^{(\frac{n}{2})^2}\omega(\T)=-\omega(\T)$. It follows that if $\tilde{\lambda}=(\tilde{\lambda}_1,\tilde{\lambda}_1)\vdash\!\vdash n$ and $\tilde{\T}\in \tilde{\lambda}$, then $(\T+\sigma(\T)|\tilde{\T}+\sigma(\tilde{\T}))=(\T-\sigma(\T)|\tilde{\T}-\sigma(\tilde{\T}))=0$, $(\T+\sigma(\T)|\tilde{\T}-\sigma(\tilde{\T}))=2\omega(\T)\delta_{\T'-\sigma(\T)',\tilde{\T}-\sigma(\tilde{\T})}$ and $(\T-\sigma(\T)|\tilde{\T}+\sigma(\tilde{\T}))=2\omega(\T)\delta_{\T'+\sigma(\T)',\tilde{\T}+\sigma(\tilde{\T})}$. The result follows. 
\end{proof}

\subsection{Factorization of the image of the Artin group inside the finite Hecke algebra}

In this subsection, we find the different factorizations between the irreducible representations of $\mathcal{A}_{D_n}$. Most of the factorization results are summarized in Proposition \ref{isomorphisme2}. We then state the main results for type $D$ in Theorems \ref{result4} and \ref{result5}.

We define the linear map $\mathcal{L}$ from $V$ to $V$ which sends $\T$ to $\mathcal{L}(\T)=\omega(\T)\varphi(\T)$.

\begin{prop}\label{transpose2}
Let $r\in [\![1,n-1]\!]$ and $\T$ a standard double-tableau, we then have
$$\mathcal{L}S_r\mathcal{L}^{-1}(\T)=(-\alpha){}^t\!(S_r^{-1})(\T), \mathcal{L}U \mathcal{L}^{-1}=(-\alpha){}^t\!(U^{-1})(\T).$$
Let $\lambda=(\lambda_1,\lambda_2)\vdash\!\vdash n$. We have the following propositions.
\begin{enumerate}
\item If $\lambda\notin\{\varphi(\lambda),(\lambda_2,\lambda_1)\}$, then $\mathcal{L}$ stabilizes $V_{\lambda}\oplus V_{\varphi(\lambda)}$ and switches $V_{\lambda}$ and $V_{\varphi(\lambda)}$.
\item If $\lambda=\varphi(\lambda)\neq (\lambda_2,\lambda_1)$, then $\mathcal{L}$ stabilizes $V_\lambda$.
\item If $n \equiv 0~ (\bmod 4)$ and $\lambda=(\lambda_1,\lambda_1)\neq (\lambda_1',\lambda_1')$, then $\mathcal{L}$ stabilizes $V_{\lambda,+}\oplus V_{\lambda',+}$ (resp $V_{\lambda,-} \oplus V_{\lambda',-}$) and switches $V_{\lambda,+}$ and $V_{\lambda',+}$ (resp $V_{\lambda,-}$ and $V_{\lambda',-}$).
\item If $n \equiv 0~ (\bmod4)$ and $\lambda=(\lambda_1,\lambda_1)=(\lambda_1',\lambda_1')$, then $\mathcal{L}$ stabilizes $V_{\lambda,+}$ and $V_{\lambda,-}$.
\item If $n\equiv 2~ (\bmod4)$ and $\lambda=(\lambda_1,\lambda_1)$, then $\mathcal{L}$ stabilizes $V_{\lambda,+}\oplus V_{\lambda',-}$ and switches $V_{\lambda,+}$ and $V_{\lambda',-}$.
\end{enumerate}
\end{prop}

This follows directly from Proposition \ref{bilin2} by writing the matrix of the bilinear form and the matrix of $\mathcal{L}$. \hfill$\square$.

\begin{prop}\label{chaise}
For $r\in [\![1,n-1]\!]$, we write $\lambda_r$ the double-partition of $n$ defined by $\lambda_r=([r],[1^{n-r}])$ if $r\geq \frac{n}{2}$ and $\lambda_r=([1^{n-r}],r)$ if $r< \frac{n}{2}$. Taking the same notations as in Proposition \ref{patate} of type $B$, for all $d\in A_{D_n}$, we have $\eta_{1,r}(d)=\eta_{2,r}(d)$ because the length in $T$ of such an element is even and we have $\beta=1$. We define the character $\eta_r$ of $A_{D_n}$ by $\eta_r(d)=\eta_{1,r}(d)=\eta_{2,r}(d)$.

We then have by Proposition \ref{patate} that $R_{\lambda_r}\simeq (\Lambda^rR_{\lambda_1})\otimes \eta_{1,r}$ for all $r\in [\![1,n-1]\!]$.
\end{prop}

Assume $\F_q=\F_p(\alpha)\neq \F_p(\alpha+\alpha^{-1})$. We write $\epsilon$ the unique automorphism of $\F_q$ of order $2$. We have that $\epsilon(\alpha)=\alpha^{-1}$. We then define for every standard double-tableau the same way as for type $B$, \break$d(\T)=\tilde{d}(\T_1)\tilde{d}(\T_2)\underset{i<j}{\underset{i\in \T_1,j\in \T_2}\prod}\frac{2+\alpha^{a_{i,j}-1}+\alpha^{1-a_{i,j}}}{\alpha+\alpha^{-1}+\alpha^{a_{i,j}}+\alpha^{-a_{i,j}}}$ and the associated hermitian form $\langle .,.\rangle$ defined by \break$\langle \T,\tilde{\T}\rangle = d(\T)\delta_{\T,\tilde{\T}}$. We write $\Lambda$ for the set of all double-partitions $\lambda=(\lambda_1,\lambda_2)$ of $n$ such that $\lambda_1\geq \lambda_2$. 

\begin{prop}\label{unitary2}
For all $d\in A_{D_n}$, $\lambda\in \Lambda$ and $\T\in \lambda$, we have $\langle d.\T,d.\tilde{\T}\rangle=\langle \T,\tilde{\T}\rangle$. This shows that $A_{D_n}$ acts in a unitary way on those irreducible modules.  
\end{prop}

\begin{proof}
The proof of the first statement follows from Proposition \ref{unitary} and the second follows from the expression of the bases of $V_{\lambda,\pm}$ and the $\Z$-bilinearity of the hermitian form. 
\end{proof}

The following proposition summarizes the results in this section :

\begin{prop}\label{isomorphisme2}
Let $\lambda$, $\mu$, $\gamma$ and $\delta$ be doubles-partitions of $n$ such that $\dim(V_\lambda)>1$, $\lambda_1>\lambda_2, \mu_1>\mu_2$, $\gamma_1=\gamma_2$ and $\delta_1=\delta_2$. We have the following properties.
\begin{enumerate}
\item The restrictions of $R_\lambda$, $R_{\gamma,+}$ and $R_{\gamma,-}$ to $\mathcal{A}_{D_n}$ are absolutely irreducible.
\item $R_{\lambda|\mathcal{A}_{D_n}}\simeq R_{\mu|\mathcal{A}_{D_n}} \Leftrightarrow \lambda=\mu$.
\item $R_{\lambda|\mathcal{A}_{D_n}}\not\simeq R_{\gamma,\pm|\mathcal{A}_{D_n}}$.
\item $R_{\gamma,\pm|\mathcal{A}_{D_n}} \simeq R_{\delta,\pm|\mathcal{A}_{D_n}} \Leftrightarrow \gamma= \delta$.
\item $R_{\gamma,\pm|\mathcal{A}_{D_n}} \not\simeq R_{\delta,\mp|\mathcal{A}_{D_n}}$.
\item $R_{\lambda|\mathcal{A}_{D_n}}\simeq R_{\mu|\mathcal{A}_{D_n}}^\star \Leftrightarrow \mu =\varphi(\lambda)$.
\item $R_{\lambda|\mathcal{A}_{D_n}}\not\simeq R_{\gamma,\pm}^\star$.
\item If $n \equiv 0~(\bmod 4)$, then 
\begin{enumerate}
\item $R_{\gamma,\pm|\mathcal{A}_{D_n}} \simeq R_{\delta,\pm|\mathcal{A}_{D_n}}^\star \Leftrightarrow \gamma= \varphi(\delta)$.
\item $R_{\gamma,\pm|\mathcal{A}_{D_n}} \not\simeq R_{\delta,\mp|\mathcal{A}_{D_n}}^\star$.
\end{enumerate}
\item If $n\equiv 2~(\bmod 4)$, then
\begin{enumerate}
\item $R_{\gamma,\pm|\mathcal{A}_{D_n}} \not\simeq R_{\delta,\pm|\mathcal{A}_{D_n}}^\star$.
\item $R_{\gamma,\pm|\mathcal{A}_{D_n}} \simeq R_{\delta,\mp|\mathcal{A}_{D_n}}^\star \Leftrightarrow \gamma=\varphi(\delta)$.
\end{enumerate}
\item If $\F_q=\F_p(\alpha)\neq \F_p(\alpha+\alpha^{-1})$, then 
\begin{enumerate}
\item $R_{\lambda|\mathcal{A}_{D_n}} \simeq \epsilon\circ R_{\mu|\mathcal{A}_{D_n}}^\star \Leftrightarrow \lambda=\mu$.
\item $R_{\lambda|\mathcal{A}_{D_n}}\not\simeq \epsilon \circ R_{\gamma,\pm}$.
\item $R_{\gamma,\pm|\mathcal{A}_{D_n}} \simeq \epsilon \circ R_{\delta,\pm|\mathcal{A}_{D_n}}^\star \Leftrightarrow \gamma= \delta$.
\item $R_{\gamma,\pm|\mathcal{A}_{D_n}} \not\simeq \epsilon\circ R_{\delta,\mp|\mathcal{A}_{D_n}}^\star$.
\end{enumerate}
\end{enumerate}
\end{prop}

\begin{proof}
$1.$ is shown in the same way as Lemma $3.4.$ of  \cite{BMM}, because $A_{D_n}$ is generated by $\mathcal{A}_{D_n}$ ans $A_{D_{n-1}}$.

Using Propositions \ref{unitary2} and \ref{transpose2}, it is sufficient to show $2,3,4$ and $5$ to conclude the proof. In the same way as for type $B$, we need to use Lemma \ref{abel}. If $R_{\lambda|\mathcal{A}_{D_n}}\simeq R_{\mu|\mathcal{A}_{D_n}}$ then there exists a character $\eta : A_{D_n} \rightarrow \F_q^\star$ such that $R_\lambda \simeq R_\mu \otimes \eta$. Since $A_{D_n}/\mathcal{A}_{D_n}=<\overline{S_1}>$, there exists $u\in \F_q^\star$ such that for all $d\in A_{D_n}, \eta(d)=u^{l(d)}$. We have $R_\lambda(S_1)=uR_\mu(S_1)$. By considering the eigenvalues, we have that $\{\alpha,-1\}=\{u\alpha,-u\}$. Therefore $u=1$ or $\alpha^2=1$. By the conditions on $\alpha$,  $u=1$ and $R_\lambda\simeq R_\mu$ so $\lambda=\mu$. Since the set of eigenvalues is of $R_{\gamma,\pm}(S_1)$ is also $\{\alpha,-1\}$, the rest of the proof follows.
\end{proof}

We now give a theorem for double-partitions with an empty component and then results for hook partitions.

\begin{theo}\label{empty}
Let $\lambda=(\lambda_1,\emptyset)\vdash\!\vdash n$ with $\lambda_1$ not a hook and $G=R_\lambda(\mathcal{A}_{D_n})$. We then have the following.
\begin{enumerate}
\item If $\F_q=\F_p(\alpha)=\F_p(\alpha+\alpha^{-1})$, then
\begin{enumerate}
\item if $\lambda_1\neq \lambda_1'$, then $G=SL_{n_{\lambda}}(q)$,
\item if $\lambda_1=\lambda_1'$ and ($p=2$ or $p\geq 3$ and $\tilde{\nu}(\lambda)=-1$), then $G=SP_{n_\lambda}(q)$,
\item if $\lambda_1=\lambda_1'$, $p\geq3$ and $\tilde{\nu}(\lambda)=1$, then $G=\Omega^+_{n_\lambda}(q)$.
\end{enumerate}
\item If $\F_q=\F_p(\alpha)\neq\F_p(\alpha+\alpha^{-1})$, then
\begin{enumerate}
\item if $\lambda_1\neq \lambda_1'$ then $G=SU_{n_{\lambda}}(q^\frac{1}{2})$,
\item if $\lambda_1=\lambda_1'$ and $\tilde{\nu}(\lambda)=-1$ then $G=SP_{n_\lambda}(q^\frac{1}{2})$,
\item if $\lambda_1=\lambda_1'$ and $\tilde{\nu}(\lambda)=1$ then $G=\Omega^+_{n_\lambda}(q^\frac{1}{2})$.
\end{enumerate}
\end{enumerate}
\end{theo}

\begin{proof}
The restriction of $R_\lambda$ to $\mathcal{A}_{D_n}$ is the same as the representation $R_{\lambda_1}$ in type $A$. Since $\tilde{\nu}(\lambda)=\nu(\lambda_1)$, the result follows directly from \cite[Theorem 1.1]{BMM} after noting that $R_\lambda(\mathcal{A}_{A_n})\subset R_{\lambda}(\mathcal{A}_{D_n})$ and that we have the corresponding inclusions by Proposition \ref{bilin2}.
\end{proof}

\begin{prop}\label{timberwolves}
If $\F_q=\F_p(\alpha)=\F_p(\alpha+\alpha^{-1})$, then $R_{([1^{n-1}],[1])}(\mathcal{A}_{D_n})= SL_n(q)$ and if $\F_q=\F_p(\alpha)\neq \F_p(\alpha+\alpha^{-1})$ then $R_{([1^{n-1}],[1])}(\mathcal{A}_{D_n})=SU_n(q^\frac{1}{2})$.
\end{prop}

\begin{proof}
The proof is the same one as the proof of Proposition \ref{lesgourgues}. 
\end{proof}

We write again $A_{1,n}=\{(\lambda_1,\emptyset),\lambda_1 \vdash n\}, A_{2,n} =\{(\emptyset,\lambda_2),\lambda_2 \vdash n\}, A_n = A_{1,n} \cup A_{2,n}$ and\\
$\epsilon_n=\{\lambda \vdash\!\vdash n, \lambda~\mbox{not a hook}\}$

\begin{theo}\label{result4}
If $\F_q=\F_p(\alpha)=\F_p(\alpha+\alpha^{-1})$ and $n$ is odd, then the morphism from $\mathcal{A}_{D_n}$ to $\mathcal{H}_{D_n,\alpha}^\times \simeq \underset{\lambda_1>\lambda_2}{\underset{\lambda \vdash\vdash n}\prod}GL_{n_\lambda}(q)$ factorizes through the epimorphism
$$\Phi_n: \mathcal{A}_{D_n} \rightarrow SL_{n-1}(q) \times SL_n(q) \times \underset{\lambda_1>\lambda_2}{\underset{\lambda\in \epsilon_n,\lambda>\varphi(\lambda)}\prod} SL_{n_\lambda}(q)\times \underset{n_\lambda> n_\mu}{\underset{\lambda\in \epsilon_n,\lambda=\varphi(\lambda)}\prod}OSP(\lambda)'.$$ 
If $\F_q=\F_p(\alpha)=\F_p(\alpha+\alpha^{-1})$ and $n \equiv 0~ (\bmod4)$, then the morphism from $\mathcal{A}_{D_n}$ to $\mathcal{H}_{D_n,\alpha}^\times \simeq \underset{\lambda_1>\lambda_2}{\underset{\lambda \vdash\vdash n}\prod}GL_{n_\lambda}(\F_q) \times \underset{\lambda=(\lambda_1,\lambda_1)\vdash n}\prod GL_{n_{\lambda,+}}(q)\times GL_{n_{\lambda,-}}(q)$ factorizes through the epimorphism
$$\Phi_n: \mathcal{A}_{D_n} \rightarrow SL_{n-1}(q) \times SL_n(q) \times \underset{\lambda_1>\lambda_2}{\underset{\lambda\in \epsilon_n,\lambda>\varphi(\lambda)}\prod} SL_{n_\lambda}(q)\times \underset{\lambda_1> \lambda_2}{\underset{\lambda\in \epsilon_n,\lambda=\varphi(\lambda)}\prod}OSP(\lambda)'\times$$ $$\underset{\lambda>\varphi(\lambda)}{\underset{\lambda=(\lambda_1,\lambda_1)\in \epsilon_n}\prod}SL_{\frac{n_\lambda}{2}}(q)^2\times \underset{\lambda=\varphi(\lambda)}{\underset{\lambda=(\lambda_1,\lambda_1)\in \epsilon_n}\prod}OSP(\lambda,+)'^2.$$ 
If $\F_q=\F_p(\alpha)=\F_p(\alpha+\alpha^{-1})$ and $n \equiv 2~ (\bmod4)$ then the morphism from $\mathcal{A}_{D_n}$ to $\mathcal{H}_{D_n,\alpha}^\times \simeq \underset{\lambda_1>\lambda_2}{\underset{\lambda \vdash\vdash n}\prod}GL_{n_\lambda}(\F_q) \times \underset{\lambda=(\lambda_1,\lambda_1)\vdash n}\prod GL_{n_{\lambda,+}}(q)\times GL_{n_{\lambda,-}}(q)$ factorizes through the epimorphism
$$\Phi_n: \mathcal{A}_{D_n} \rightarrow SL_{n-1}(q) \times SL_n(q) \times \underset{\lambda_1>\lambda_2}{\underset{\lambda\in \epsilon_n,\lambda>\varphi(\lambda)}\prod} SL_{n_\lambda}(q)\times \underset{\lambda_1> \lambda_2}{\underset{\lambda\in \epsilon_n,\lambda=\varphi(\lambda)}\prod}OSP(\lambda)'\times$$ $$\underset{\lambda>\varphi(\lambda)}{\underset{\lambda=(\lambda_1,\lambda_1)\in \epsilon_n}\prod}SL_{\frac{n_\lambda}{2}}(q)^2\times \underset{\lambda=\varphi(\lambda)}{\underset{\lambda=(\lambda_1,\lambda_1)\in \epsilon_n}\prod}SL_{\frac{n_\lambda}{2}}(q).$$ 
In all of the above, $OSP(\lambda)$ is the group of isometries of the bilinear form defined in Proposition \ref{bilin2}.
\end{theo}

In the unitary case, we have an analogous result :

\begin{theo}\label{result5}
If $\F_q=\F_p(\alpha)\neq \F_p(\alpha+\alpha^{-1})$ and $n$ is odd, then the morphism from $\mathcal{A}_{D_n}$ to $\mathcal{H}_{D_n,\alpha}^\times \simeq \underset{\lambda_1>\lambda_2}{\underset{\lambda \vdash\vdash n}\prod}GL_{n_\lambda}(q)$ factorizes through the morphism
$$\Phi_n: \mathcal{A}_{D_n} \rightarrow SU_{n-1}(q^\frac{1}{2}) \times SU_n(q^\frac{1}{2}) \times \underset{\lambda_1>\lambda_2}{\underset{\lambda\in \epsilon_n,\lambda>\varphi(\lambda)}\prod} SU_{n_\lambda}(q^\frac{1}{2})\times \underset{n_\lambda> n_\mu}{\underset{\lambda\in \epsilon_n,\lambda=\varphi(\lambda)}\prod}\widetilde{OSP}(\lambda)'.$$ 
If $\F_q=\F_p(\alpha)=\F_p(\alpha+\alpha^{-1})$ and $n \equiv 0~ (\bmod4)$, then the morphism from $\mathcal{A}_{D_n}$ to $\mathcal{H}_{D_n,\alpha}^\times \simeq \underset{\lambda_1>\lambda_2}{\underset{\lambda \vdash\vdash n}\prod}GL_{n_\lambda}(\F_q) \times \underset{\lambda=(\lambda_1,\lambda_1)\vdash n}\prod GL_{n_{\lambda,+}}(q)\times GL_{n_{\lambda,-}}(q)$ factorizes through the morphism
$$\Phi_n: \mathcal{A}_{D_n} \rightarrow SU_{n-1}(q^\frac{1}{2}) \times SU_n(q^\frac{1}{2}) \times \underset{\lambda_1>\lambda_2}{\underset{\lambda\in \epsilon_n,\lambda>\varphi(\lambda)}\prod} SU_{n_\lambda}(q^\frac{1}{2})\times \underset{\lambda_1> \lambda_2}{\underset{\lambda\in \epsilon_n,\lambda=\varphi(\lambda)}\prod}\widetilde{OSP}(\lambda)'\times$$ $$\underset{\lambda>\varphi(\lambda)}{\underset{\lambda=(\lambda_1,\lambda_1)\in \epsilon_n}\prod}SU_{\frac{n_\lambda}{2}}(q^\frac{1}{2})^2\times \underset{\lambda=\varphi(\lambda)}{\underset{\lambda=(\lambda_1,\lambda_1)\in \epsilon_n}\prod}\widetilde{OSP}(\lambda,+)'^2.$$ 
If $\F_q=\F_p(\alpha)=\F_p(\alpha+\alpha^{-1})$ and $n \equiv 2~ (\bmod4)$, then the morphism from $\mathcal{A}_{D_n}$ to $\mathcal{H}_{D_n,\alpha}^\times \simeq \underset{\lambda_1>\lambda_2}{\underset{\lambda \vdash\vdash n}\prod}GL_{n_\lambda}(\F_q) \times \underset{\lambda=(\lambda_1,\lambda_1)\vdash n}\prod GL_{n_{\lambda,+}}(q)\times GL_{n_{\lambda,-}}(q)$ factorizes through the morphism
$$\Phi_n: \mathcal{A}_{D_n} \rightarrow SU_{n-1}(q^\frac{1}{2}) \times SU_n(q^\frac{1}{2}) \times \underset{\lambda_1>\lambda_2}{\underset{\lambda\in \epsilon_n,\lambda>\varphi(\lambda)}\prod} SU_{n_\lambda}(q^\frac{1}{2})\times \underset{\lambda_1> \lambda_2}{\underset{\lambda\in \epsilon_n,\lambda=\varphi(\lambda)}\prod}\widetilde{OSP}(\lambda)'\times$$ $$\underset{\lambda>\varphi(\lambda)}{\underset{\lambda=(\lambda_1,\lambda_1)\in \epsilon_n}\prod}SU_{\frac{n_\lambda}{2}}(q^\frac{1}{2})^2\times \underset{\lambda=\varphi(\lambda)}{\underset{\lambda=(\lambda_1,\lambda_1)\in \epsilon_n}\prod}SU_{\frac{n_\lambda}{2}}(q^\frac{1}{2}).$$ 
In all of the above, $\widetilde{OSP}(\lambda)$ is the group of isometries associated with the bilinear form over $\F_{q^\frac{1}{2}}$ obtained from the one in Proposition \ref{bilin2} using Proposition \ref{coolprop}.
\end{theo}

Those two theorems (except for the surjectivity) follow Propositions \ref{bilin2}, \ref{transpose2}, \ref{chaise}, \ref{unitary2}, \ref{timberwolves}, Theorem \ref{empty} and Proposition \ref{isomorphisme2}. It now remains to check that $\Phi_n$ is surjective in all cases.

\subsection{The case $n=4$}

In this subsection, we prove the result for $n=4$.

The double-partitions to consider for $n=4$ are $([4],\emptyset)$, $([3,1],\emptyset)$, $([2,2],\emptyset)$, $([2,1,1],\emptyset)$, $([1^4],\emptyset)$, $([3],[1])$, $([2,1],[1])$, $([1^3],[1])$, $([2],[2])$, $([2],[1^2])$ and $([1^2],[1^2])$.

By Proposition \ref{transpose2}, if we know the image for $\lambda$, we know the image for $\varphi(\lambda)$. By Proposition \ref{empty}, we know the image for doubles-partitions with an empty component. By Proposition \ref{timberwolves}, we know the image for $([1^3],[1])$ and by Proposition \ref{chaise}, we know the image for $([2],[1^2])$ using the image of $([1^3],[1])$. The only double-partitions left to consider are $([1^2],[1^2])$ and $([2,1],[1])$. 

\begin{lemme}\label{Berkeley}
If $\F_q=\F_p(\alpha)=\F_p(\alpha+\alpha^{-1})$, then $R_{[2,1],[1]}(\mathcal{A}_{D_4}) =SP_8(q)$.

If $\F_q=\F_p(\alpha)\neq \F_p(\alpha+\alpha^{-1})$, then $R_{[2,1],[1]}(\mathcal{A}_{D_4}) =SP_8(q^\frac{1}{2})$.
\end{lemme}

\begin{proof}
Assume first that $\F_q=\F_p(\alpha)=\F_p(\alpha+\alpha^{-1})$. Using Proposition \ref{bilin2}, there exists $P\in GL_8(q)$ such that $G=PR_{[2,1],[1]}(\mathcal{A}_{D_4})P^{-1}\subset SP_8(q)$. Using Proposition \ref{branch}, we have that $R_{[2,1],[1]}(\mathcal{A}_{D_3})=R_{[2],[1]}\times R_{[1^2],[1]}\times R_{[2,1],\emptyset}(\mathcal{A}_{D_3})\simeq SL_3(q)\times SL_2(q)$, where $SL_3(q)$ is in a twisted diagonal embedding and $SL_2(q)$ is in a natural representation using Goursat's Lemma and the previous arguments. Using the same arguments as before and Lemma \ref{field} with the natural representation of $SL_2(q)$, we know $G$ is primitive, tensor-indecomposable, irreducible, perfect and cannot be realized in a natural representation over a proper subfield of $\F_q$. This implies that $G$ cannot be included in a maximal subgroup of class $\mathcal{C}_1, \mathcal{C}_2,\mathcal{C}_4$ or $\mathcal{C}_5$. Using the fact that the order of $SL_3(q)\times SL_2(q)$ is $q^4(q^2-1)^2(q^3-1)$ and the fact that $\alpha$ is of order greater than 16, we have that $q>17$ so $q>19$ and $\vert G\vert \geq 19^4(19^2-1)(19^3-1)$. Looking at the Tables 8.48. and 8.49. in \cite{BHRC}, we have by cardinality arguments that $G$ can be included in no maximal subgroup of $SP_8(q)$, so $G=SP_8(q)$.

Assume now $\F_q=\F_p(\alpha)\neq \F_p(\alpha+\alpha^{-1})$. There exists $P\in GL_8(q)$ such that $G=PR_{[2,1],[1]}(\mathcal{A}_{D_4})P^{-1}\subset SP_8(q^\frac{1}{2})$ and $G$ contains $SU_3(q^\frac{1}{2})\times SU_2(q^\frac{1}{2})$, where $SU_3(q^\frac{1}{2})$ is in a twisted diagonal embedding and $SU_2(q^\frac{1}{2})$ is in a natural representation. We can no longer use Lemma \ref{field} in this case, but since $\epsilon(\alpha)=\alpha^{-1}$, we have up to conjugation that $\diag(I_6,\begin{pmatrix}
\alpha & 0\\
0 & \alpha^{-1}
\end{pmatrix})\in G$. It follows that $\alpha+
\alpha^{-1}$ belongs to the field generated by the traces of the elements of $G$. This shows that any field over which $G$ is realized in a natural representation contains $\F_{q^\frac{1}{2}}$. By the above, in this case, we have that $G$ is primitive, tensor-indecomposable, irreducible, perfect and cannot be realized in a natural representation over a proper subfield of $\F_{q^\frac{1}{2}}$. This implies that $G$ cannot be included in a maximal subgroup of $SP_{8}(q^\frac{1}{2})$ of class $\mathcal{C}_1, \mathcal{C}_2,\mathcal{C}_4$ or $\mathcal{C}_5$. Since $\alpha^{q^\frac{1}{2}}=\epsilon(\alpha)=\alpha^{-1}$, we have that $q^\frac{1}{2}+1>16$ and so $q^\frac{1}{2}\geq 17$ because $p\neq 2$. Looking again at the tables 8.48. and 8.49. in \cite{BHRC} we get that $G=SP_8(q^\frac{1}{2})$ or $G\subset GU_4(q^\frac{1}{2}).2$. Since $G$ is perfect, the latter would imply that $G\subset SU_4(q^\frac{1}{2})$ and that there is some copy of $SU_3(q^\frac{1}{2})\times SU_2(q^\frac{1}{2})$ inside $SU_4(q^\frac{1}{2})$. Looking at all the maximal subgroups in tables 8.10. and 8.11., we see using only cardinality arguments, that such a copy cannot be included in any maximal subgroup of $SU_4(q^\frac{1}{2})$ and so such a copy is equal to $SU_4(q^\frac{1}{2})$. This leads to a contradiction. This proves that $G=SP_8(q^\frac{1}{2})$ and concludes the proof of the lemma. 
\end{proof}

\begin{lemme}
If $\F_q=\F_p(\alpha)=\F_p(\alpha+\alpha^{-1})$, we have $R_{([1^2],[1^2]),+}(\mathcal{D}_4)=R_{([1^2],[1^2]),-}(\mathcal{D}_4)=SL_3(q)$.

If $\F_q=\F_p(\alpha)\neq \F_p(\alpha+\alpha^{-1})$, we have $R_{([1^2],[1^2]),+}(\mathcal{D}_4)=R_{([1^2],[1^2]),-}(\mathcal{D}_4)=SU_3(q^\frac{1}{2})$.
\end{lemme}

\begin{proof}
The result follows from Proposition \ref{branch} and the fact that $R_{([1^2],[1])}(\mathcal{D}_3)$ is equal to the group we want in both cases. 
\end{proof}

\subsection{Surjectivity of $\Phi_n$ for $n\geq 5$}

In this subsection, we use the previous subsections to prove by induction on $n$ the main results for type $D$.

Assume first that $\F_q=\F_p(\alpha)=\F_p(\alpha+\alpha^{-1})$. Using Proposition \ref{isomorphisme2}, by the same kind of arguments as for type $B$, we can use Goursat's Lemma to show the morphism is surjective upon each component. This means it is sufficient to show the following theorem.

\begin{theo}\label{Sandburg}
Let $\lambda=(\lambda_1,\lambda_2)\vdash\!\vdash n$ not a hook, such that $\lambda_1\geq \lambda_2$. We write $G(\lambda)=R_{\lambda}(\mathcal{A}_{D_n})$ if $\lambda_1> \lambda_2$, $G(\lambda,+)=R_{\lambda,+}(\mathcal{A}_{D_n})$ and $G(\lambda,-)=R_{\lambda,-}(\mathcal{A}_{D_n})$ otherwise. We then have the following possibilities.
\begin{enumerate}
\item If $\lambda=([2,1^{n-2}],\emptyset)$, then $G(\lambda)=SL_{n-1}(q)$.
\item If $\lambda=([1^{n-1}],[1])$, then $G(\lambda)=SL_n(q),$
\item If $\lambda \in \epsilon_n, \lambda_1>\lambda_2$ and $\lambda>\varphi(\lambda)$, then $G(\lambda)=SL_{n_\lambda}(q)$,
\item If $\lambda \in \epsilon_n, \lambda_1>\lambda_2$ and $\lambda=\varphi(\lambda)$, then we have the following possibilities.
\begin{enumerate}
\item If $\tilde{\nu}(\lambda)=-1$, then $G(\lambda)=SP_{n_\lambda}(q)$.
\item If $\tilde{\nu}(\lambda)=1$, then $G(\lambda)=\Omega_{n_\lambda}^+(q)$.
\end{enumerate}
\item If $\lambda=(\lambda_1,\lambda_1)\in \epsilon_n$, then we have the following possibilities.
\begin{enumerate}
\item If $\varphi(\lambda)>\lambda$, then $G(\lambda,+)=G(\lambda,-)=SL_{\frac{n_{\lambda}}{2}}(q)$.
\item If $\varphi(\lambda)=\lambda$, then we have the following possibilities.
\begin{enumerate}
\item If $n\equiv 0~(\bmod 4)$ then
\begin{enumerate}
\item  if $\tilde{\nu}(\lambda)=-1$ then $G(\lambda,+)=G(\lambda,-)=SP_{n_\lambda}(q)$,
\item if $\tilde{\nu}(\lambda)=1$ then $G(\lambda,+)=G(\lambda,-)=\Omega_{n_\lambda}^+(q)$.
\end{enumerate}
\item If $n\equiv 2~(\bmod 4)$ then $G(\lambda,+)=G(\lambda,-)=SL_{\frac{n_\lambda}{2}}(q)$.
\end{enumerate} 
\end{enumerate}
\end{enumerate}
\end{theo}

\begin{proof}
For $n=4$, we have the result by the previous section. Theorem \ref{empty} gives us the result for double-partitions with an empty component and Proposition \ref{timberwolves} gives us the result for double-partitions with two rows or two columns and one of the components of size one. For $n\geq 5$, we proceed by induction but we must first treat the following cases separately : $([2,2],[1]), ([1^3],[1^2])$ and $(([1^3],[1^3]),\pm)$.

By Lemma \ref{Berkeley}, Theorem \ref{empty}, Proposition \ref{branch} and Goursat's Lemma, we have $R_{([2,2],[1])}(\mathcal{A}_{D_4})=SP_8(q)\times SP_2(q)$. By Theorem 1 of \cite{SZ} and Lemma 5.6. of \cite{BMM}, we have that $R_{([2,2],[1])}(\mathcal{D}_5)\in \{SL_{10}(q),SU_{10}(q^\frac{1}{2}),SP_{10}(q)\}$. Since $([2,2],[1])=\varphi(([2,2],[1]))$ and $\tilde{\nu}([2,2],[1])=(-1)^{\frac{4-2}{2}}(-1)^{\frac{1-1}{2}}=-1$, we have up to conjugation that $R_{([2,2],[1])}(\mathcal{D}_5)\subset SP_{10}(q)$ by Proposition \ref{bilin2} and so up to conjugation, we have $R_{([2,2],[1])}(\mathcal{D}_5)=SP_{10}(q)$.

In the same way, we have that $R_{([1^3],[1^2])}(\mathcal{D}_4)=SL_4(q)\times SL_3(q)\times SL_3(q)$ so $G(([1^3],[1^2]))=R_{([1^3],[1^2])}(\mathcal{D}_5)$ is in $\{SL_{10}(q),SU_{10}(q^\frac{1}{2}),SP_{10}(q)\}$. By Proposition \ref{isomorphisme2}, we know that $G(([1^3],[1^2]))$ preserves no bilinear form, so we only have to exclude the unitary case. Assume $G(([1^3],[1^2]))$ is included up to conjugation in $SU_{10}(q^\frac{1}{2})$. There then exists an automorphism $\epsilon$ of order $2$ of $\F_q$ such that each $M$ in $G(\lambda)$ is conjugated to ${}^t\!\epsilon((M^{-1}))$. In particular $G(([1^3],[1^2]))$ contains a natural $SL_2(q)$. This implies that $\diag(I_8,\begin{pmatrix}\alpha & 0\\
0 & \alpha^{-1}
\end{pmatrix})$ is conjugated to $\diag(I_8,{}^t\!\epsilon((\begin{pmatrix}
\alpha & 0\\
0 & \alpha^{-1}
\end{pmatrix})^{-1}))$. Taking the traces of those matrices implies that $\epsilon(\alpha+\alpha^{-1}+8) =\alpha+\alpha^{-1}+8$. We have that $\F_q=\F_p(\alpha)=\F_p(\alpha+\alpha^{-1})$ so this shows that $\epsilon$ is trivial which is a contradiction. It follows that $G(([1^3],[1^2]))=SL_{10}(q)$.

By Proposition \ref{branch} and the fact that $R_{([1^3],[1^2])}(\mathcal{D}_5)=SL_{10}(q)$, we have that $SL_{10}(q)\subset R_{([1^3],[1^3]),\pm}(\mathcal{D}_6)\subset SL_{10}(q)$. It follows that $R_{([1^3],[1^3]),\pm}(\mathcal{D}_6)= SL_{10}(q)$.

We now proceed to the induction on $n$ using Theorem \ref{CGFS}.

Let $n\geq 5$ and $\lambda\vdash\!\vdash n$. Suppose the theorem is true for $n-1$. We use Proposition \ref{branch} for different possibilities to show that $G(\lambda)$ or $G(\lambda,\pm)$ contains a subgroup verifying the same properties as in type B.

\begin{enumerate}
\item If $\lambda=(\lambda_1,\lambda_2)$ and $\lambda_1>\lambda_2$ and $\lambda\neq \varphi(\lambda)$ then $\varphi(\lambda)=(\lambda_1',\lambda_2')$ because the order we picked for partitions of $\frac{n}{2}$ verifies that, if $\lambda_1\neq \lambda_2'$ and $\lambda_1>\lambda_2$, then $\lambda_1'>\lambda_2'$. We then have $\lambda_1\neq \lambda_1'$ or $\lambda_2\neq \lambda_2'$.
\begin{enumerate}
\item If $\lambda_2'\neq \lambda_2$, then there exists $\mu_2 \subset \lambda_2$ such that $\mu_2'\not\subset\lambda_2$. We have that $(\lambda_1',\mu_2')\not\subset (\lambda_1,\lambda_2)$ because $\mu_2'\not\subset \lambda_2)$ and $(\mu_2',\lambda_1') \not\subset (\lambda_1,\lambda_2)$ because otherwise $\lambda_1'=\lambda_2$ and so $\lambda_2'=\lambda_1$. This shows that $G(\lambda)$ contains a natural $SL_3(q)$.
\item If $\lambda_2=\lambda_2'$ and $\lambda_1\neq \lambda_1'$, then there exists $\mu_1\subset \lambda_1$ such that $\mu_1'\not\subset\lambda_1$. We then have $(\mu_1',\lambda_2') \not\subset (\lambda_1,\lambda_2)$ because $\mu_1'\not\subset \lambda_1$ and $(\lambda_2',\mu_1')\not\subset (\lambda_1,\lambda_2)$ because $\lambda_2'\neq \lambda_1$. This shows that $G(\lambda)$ also contains a natural $SL_3(q)$ in this case.
\end{enumerate}
\item If $\lambda=(\lambda_1,\lambda_2)=\varphi(\lambda)$ and $\lambda_1>\lambda_2$, then
\begin{enumerate}
\item If $\varphi(\lambda)=(\lambda_1',\lambda_2')$, then
\begin{enumerate}
\item If $\lambda_1$ and $\lambda_2$ are square partitions, then $R_{\lambda}(\mathcal{D}_{n-1})=G(\mu_1,\lambda_2)\times G(\lambda_1,\mu_2)$ and since \break$\tilde{\nu}(\mu_1,\lambda_2)=\tilde{\nu}(\lambda_1,\mu_2)=\tilde{\nu}(\lambda)$, we have that :
\begin{enumerate}
\item If $\tilde{\nu}(\lambda)=1$, then $\Omega_{n_{(\mu_1,\lambda_2)}}^+(q)\times \Omega_{n_{(\lambda_1,\mu_2)}}^+(q)\subset G(\lambda)\subset \Omega_{n_{(\lambda_1,\lambda_2)}}^+(q)$.
\item If $\tilde{\nu}(\lambda)=-1$ then $SP_{n_{(\mu,\lambda_2)}}(q)\times SP_{n_{(\lambda_1,\mu_2)}}(q) \subset G(\lambda)\subset SP_{n_\lambda}(q)$. It follows that $G(\lambda)$ is an irreducible group generated by transvections because it is normally generated by the group on the left of our inclusions so by Theorem 1 of Serezkin-Zalesskii \cite{SZ}, we have that $G(\lambda)$ is equal to the group on the right and the theorem is proved in this case.
\end{enumerate}
\item If $\lambda_1$ or $\lambda_2$ is not a square partition then there exists $\mu \subset \lambda$ such that $\varphi(\mu)\neq \mu$. It follows that $\varphi(\mu)\subset \lambda$ or $\sigma(\varphi(\mu))\subset \lambda$, so $G(\lambda)$ contains a twisted diagonal $SL_3(q)$.
\end{enumerate}
\item If $\varphi(\lambda)=(\lambda_2',\lambda_1')$, then if $\mu \subset \lambda_2$, we have that $(\lambda_1,\mu)\subset (\lambda_1,\lambda_2), \varphi((\lambda_1,\mu))=(\lambda_1',\mu')\not\subset (\lambda_1,\lambda_2)$ because $\lambda_1\neq\lambda_1'$. We have that $(\mu',\lambda_1')\subset (\lambda_1,\lambda_2)$ so $G(\lambda)$ contains a twisted diagonal $SL_3(q)$.
\end{enumerate}
\item If $\lambda=(\lambda_1,\lambda_1)\neq (\lambda_1',\lambda_1')$, then there exists $\mu_1\subset \lambda_1$ such that $\mu_1'\not\subset \lambda_1$. It follows that $(\lambda_1',\mu_1')\not\subset (\lambda_1,\lambda_1)$ and $(\mu_1',\lambda_1')\not\subset (\lambda_1,\lambda_1)$. This shows that $G(\lambda,\pm)$ contains a natural $SL_3(q)$.
\item If $\lambda=(\lambda_1,\lambda_1) =(\lambda_1',\lambda_1)$ and $\lambda_1$ is not a square partition, then there exists $\mu_1\subset \lambda_1$ such that $\mu_1\neq \mu_1'$ so $(\lambda_1,\mu_1)\neq \varphi((\lambda_1,\mu_1))=(\lambda_1',\mu_1')$. We have that $(\mu_1',\lambda_1')\subset(\lambda_1,\lambda_1)$ so $G(\lambda,\pm)$ contains a twisted diagonal $SL_3(q)$.
\item If $\lambda=(\lambda_1,\lambda_1)=(\lambda_1',\lambda_1')$ and $\lambda_1$ is a square partition, then we have the two following possibilities.
\begin{enumerate}
\item If $n \equiv 0~(\bmod 4)$, then for all $\mu\subset \lambda$, we have that $\tilde{\nu}(\lambda)=\nu(\lambda_1)^2(-1)^{(\frac{n}{2})^2}=1=\tilde{\nu}(\mu)$. This is because if $\lambda_1$ is a square, then the only sub-partition $\mu_1$ of $\lambda_1$ verifies $\nu(\mu_1)=\nu(\lambda_1)$. By the branching rule, we have that $\Omega^+_{\frac{n_\lambda}{2}}(q)\subset G(\lambda,\pm) \subset \Omega_{\frac{n_\lambda}{2}}^+(q)$. It follows that $G(\lambda)=\Omega_{\frac{n_\lambda}{2}}^+(q)$ and the theorem is proved in this case.
\item If $n \equiv 2~(\bmod 4)$, then $\tilde{\nu}(\mu)=\nu(\lambda_1)^2=1$ for all $\mu\subset \lambda$. The branching rule shows that $\Omega_{\frac{n_\lambda}{2}}^+(q)\subset G(\lambda,\pm) \subset SL_{\frac{n_\lambda}{2}}(q)$. By Proposition \ref{isomorphisme2}, $G(\lambda,\pm)$ preserves no bilinear form so $G(\lambda,\pm)=SL_{\frac{n_\lambda}{2}}(q)$.
\end{enumerate}
\end{enumerate}

In all the cases where $G(\lambda)$ or $G(\lambda,\pm)$ contains a natural $SL_3(q)$ or a twisted diagonal $SL_3(q)$, we can use exactly the same arguments as in \cite{BMM} beacuse if the morphism $A_{A_n}$ to $A_{D_n}$ defined  by $S_i \mapsto S_i$ is trivial then $A_{D_n}$ is trivial. 

The only case we need to treat separately is $(([2,1],[2,1]),\pm)$ because $n=6$. We need a separate argument to show that $G(([2,1],[2,1]),\pm)$ is tensor-indecomposable. In this case $R_{([2,1],[2,1])}(\mathcal{A}_{D_5})=G([2,1],[1^2])\times G([2,1],[2])=SL_{20}(q)\times SL_{20}(q)$. If $G(([2,1],[2,1]),\pm) \subset SL_{40}(q) \otimes SL_2(q)$, then the morphism from \break$R_{([2,1],[2,1])}(\mathcal{D}_5)$ to $SL_2(q)$ is trivial. Since $R_{([2,1],[2,1])}(\mathcal{D}_5)$ normally generates $G(([2,1],[2,1]),\pm)$, $G(([2,1],[2,1]),\pm)$ is included in $SL_{40}(q)\times SL_{40}(q)$. This contradicts its irreducibility.

This shows that it is sufficient to take care of case $2.a.i.A$. Assume we are in case $2.a.i.A$. We then have that $G(\lambda)\subset \Omega_{n_\lambda}^+(q)$ is generated by a conjugacy class of long root elements and $G(\lambda)$ is irreducible. Since $p\neq 2$, if we check that $O_p(G(\lambda))\subset [G,G]\cap Z(G)$, then we can apply Theorem I in Kantor's article \cite{K}. Applying Clifford's Theorem (Theorem 11.1 of \cite{C-R}), we have that $Res_{O_p(G(\lambda))}^{G(\lambda)}(V)$ is semisimple and since $O_p(G(\lambda))$ is a $p$-group, its only irreducible representation over $\F_q$ is the trivial one. This shows that $Res_{O_p(G(\lambda))}^{G(\lambda)}(V)$ is trivial so $O_p(G(\lambda))= 1$ and all the assumptions of Theorem I of Kantor are verified (the minimal dimension in this case is greater than or equal to the dimension of $([3,3,3],[2,2])$ and the dimension of $([4,4,4,4],[1])$ which are $42\times 2\times \binom {13} {4}\geq 5$ and $17\times 24024\geq 5$). This shows that we are in one of the following cases :
\begin{enumerate}
\item $G(\lambda)=\Omega_{n_\lambda}^{+}(q'), q'|q$,
\item $G(\lambda)=\Omega_{n_\lambda}^{-}(q') \subset \Omega^{+}_{n_\lambda}(q'^2), q'^2 | q$ and $n_\lambda$ is even,
\item $G(\lambda)=SU_{\frac{n_\lambda}{2}}(q') \subset \Omega^{+}_{n_\lambda}(q'), n_\lambda \equiv 0~ (\bmod 4)$ and $q'|q$,
\item $G(\lambda)\subset \Omega_8^+(q'), q'|q$,
\item $G(\lambda)=[G_2(q'),G_2(q')] \subset \Omega_7(q'), q'|q$,
\item $G= ~^3 D_4(q')\subset \Omega_8^+(q'^3), q'^3|q$.
\end{enumerate}
Since $n\geq 13$, $\alpha^{q-1}=1$ and $\alpha$ is of order greater than $2n$, we have $q\geq 29$ and $n_\lambda \geq \min(84\binom{13}{4},17\times 24024)$. This proves that Cases $4$, $5$ and $6$ are excluded by cardinality arguments.

Let us show that $3.$ is also excluded by cardinality arguments. We write $\vert G\vert_p$ the cardinal of a Sylow $p$-subgroup of a group $G$, so that $\vert SU_{\frac{n_\lambda}{2}}(q') \vert_p=q'^{\frac{\frac{n_\lambda}{2}(\frac{n_\lambda}{2}-1)}{2}}$. We know that $G(\lambda)$ contains $\Omega_{n_1}^+(q)\times \Omega_{n_2}^+(q)$. It follows that if $\lambda_1$ is the square partition of $r$ and $\lambda_2$ is the square partition of $n-r<r$, writing $a_l$ for the number of standard tableaux associated with a square partition of $l\in \N^\star$, we have that $n_\lambda=\binom {n} {r} a_{r}a_{n-r}$, $n_1=\binom {n-1} {r-1} a_{r}a_{n-r}$ and $n_2=\binom {n-1} {r} a_{r}a_{n-r}$. Note that $a_r$ is even because $r>1$ and using the branching rule twice, we get that $a_r$ is equal to twice the dimension of the two partitions we get by removing first the only extremal node and then one of the two extremal nodes of the resulting partition. It follows that $\vert \Omega_{n_1}^+(q)\times \Omega_{n_2}^+(q)\vert_p = q^{\frac{n_1}{2}(\frac{n_1}{2}-1)+\frac{n_2}{2}(\frac{n_2}{2}-1)}$. To exclude $3$, it is sufficient to show that this quantity is strictly greater than $q^{\frac{\frac{n_\lambda}{2}(\frac{n_\lambda}{2}-1)}{2}}$. If we write $A$ the $q$-logarithm of the quotient of those two quantities, we have that :
\begin{eqnarray*}
A & = & \frac{n_1}{2}(\frac{n_1}{2}-1)+\frac{n_2}{2}(\frac{n_2}{2}-1)-\frac{\frac{n_\lambda}{2}(\frac{n_\lambda}{2}-1)}{2}\\
 & = & \frac{n_1}{2}^2+\frac{n_2}{2}^2-\frac{\frac{n_\lambda}{2}^2}{2}-\frac{n_\lambda}{4}\\
 & = & \frac{\frac{n_\lambda}{2}^2}{2}-2\frac{n_1}{2}\frac{n_\lambda-n_1}{2}-\frac{n_\lambda}{4}\\
 & = & \frac{(\frac{n_\lambda}{2}-n_1)^2}{2}-\frac{n_\lambda}{4}\\
 & = & \frac{\left(\frac{\binom {n} {r} a_ra_{n-r}}{2}-\binom {n-1} {r-1}a_r a_{n-r}\right)^2}{2}-\frac{\binom {n} {r} a_r a_{n-r}}{4}\\
 & = &  a_ra_{n-r} \left(\frac{(\binom {n} {r} (\frac{1}{2}-\frac{r}{n}))^2}{2}a_ra_{n-r}-\frac{\binom{n}{r}}{4}\right)\\
 & = & \frac{a_r a_{n-r}}{4} \binom {n} {r}(2a_r a_{n-r} \binom {n} {r} \left(\frac{2r-n}{2n})^2-1\right). 
\end{eqnarray*} 

This shows that $A > 0$ if and only if $2a_r a_{n-r} \binom {n} {r} \frac{(2r-n)^2}{4n^2}>1$. Using the branching rule and the hook formula, we get  : $a_1=1, a_4=2, a_9=42, a_{16}=\frac{16!}{7\times 6^2\times 5^3\times 4^4\times 3^3\times 2^2}=24024>81\times 16^2, a_{25}=701149020> 81\times 25^2$ and $a_{36}>81\times 36^2$. Let $k\geq 6$, assume $a_{k^2} > 81(k^2)^2$. The branching rule shows that $a_{(k+1)^2}> 2a_{k^2} > 81(2k^4)> 81(k^4+4k^3+6k^2+4k+1)=81((k+1)^2)^2$, the last inequality being true because $k\geq 6$. It follows that for all $k\geq 4$, we have that $a_{k^2}>81\times(k^2)^2$. In our case, we have that $r\geq 16$ or $r=9$ and $n-r=4$. If $r\geq 16$ and $n-r\geq 2$ then we have $a_ra_{n_r}\geq r^2(n-r)^2\geq 4r^2\geq 2r^2+2r(n-r)\geq (r+n-r)^2\geq n^2$. It follows that $2a_r a_{n-r} \binom {n} {r} \frac{(2r-n)^2}{4n^2}\geq 2n^2\binom {n} {r} \frac{1}{4n^2}=\frac{\binom {n} {r}}{2}>1.$ If $r\geq 16$ and $n-r=1$, then $2a_r a_{n-r} \binom {n} {r} \frac{(2r-n)^2}{4n^2}=\frac{8na_{n-1}}{4n^2}=\frac{2a_{n-1}}{n}>162>1.$ If $r=9$ and $n-r=4$, then $2a_r a_{n-r} \binom {n} {r} \frac{(2r-n)^2}{4n^2}=2\times 42\times 2 \binom {13} {9} \frac{(18-13)^2}{4\times 13^2}>1.$ This shows that independently of $r$ and $n-r$, we have that $A>0$. This proves that $3.$ is excluded.

We have that $\vert \Omega_{n_\lambda}^+(q^\frac{1}{2})\vert_p=q^{\frac{\frac{n_\lambda}{2}(\frac{n_\lambda}{2}-1)}{2}}$. The previous arguments show that $2.$ is also impossible.

The only remaining possibility is $1$ and using again the same arguments, we get $q'>q^\frac{1}{2}$ so $q'=q$ and this concludes the proof of Theorem \ref{Sandburg}.
\end{proof}

In the unitary case, i.e. $\F_q=\F_p(\alpha)=\F_p(\alpha+\alpha^{-1})$, all the arguments are analogous.

\bibliographystyle{plain}
\bibliography{mybib}

\end{document}